[1994/12/01]% LaTeX date must December 1994 or later
\documentclass{ijmart}

\usepackage{amsmath,amssymb,amsfonts,amssymb,amsthm}
\usepackage{rotating, verbatim,hyperref,url,tikz,tikz-qtree,ifthen,cancel,array,graphicx,adjustbox,appendix,pifont,enumitem,vruler}
\usetikzlibrary{positioning}
\usepackage{mathrsfs}

\chardef\bslash=`\\ % p. 424, TeXbook
\newtheorem{thm}{Theorem}[section]
\newtheorem{prop}[thm]{Proposition}
\newtheorem{lem}[thm]{Lemma}
\newtheorem{cor}[thm]{Corollary}

\newtheorem*{thm.indivisibility}{Theorem \ref{thm.indivisibility}}

\newtheorem{fact}[thm]{Fact}

\newtheorem*{CharBRDSDAP}{Simple Characterization of  big Ramsey degrees}

\theoremstyle{remark}
\newtheorem{rem}[thm]{Remark}

\theoremstyle{definition}
\newtheorem{defn}[thm]{Definition}
\newtheorem{convention}[thm]{Convention}

\newtheorem{notation}[thm]{Notation}

\newtheorem{example}[thm]{Example}
\newtheorem{question}[thm]{Question}
\newtheorem{problem}[thm]{Problem}

\theoremstyle{remark}

\newtheoremstyle{refnocolon}%
{}
{}
{}%bodyfont
{}%indent
{\itshape}%headfont
{}%head punctuation
{ }%space after head
{}
\theoremstyle{refnocolon}
\newtheorem*{conv}{Convention}

\newcommand{\om}{\omega}

\newcommand{\sse}{\subseteq}
\newcommand{\contains}{\supseteq}

\newcommand{\bC}{\mathbb{C}}
\newcommand{\bD}{\mathbb{D}}

\newcommand{\bJ}{\mathbf{J}}
\newcommand{\bK}{\mathbf{K}}

\newcommand{\bM}{\mathbf{M}}

\newcommand{\bQ}{\mathbb{Q}}

\newcommand{\bT}{\mathbb{T}}
\newcommand{\bS}{\mathbb{S}}

\newcommand{\bU}{\mathbb{U}}

\newcommand{\M}{\mathrm{M}}
\newcommand{\N}{\mathrm{N}}
\newcommand{\re}{\!\restriction\!}
\newcommand{\bA}{\mathbf{A}}
\newcommand{\bB}{\mathbf{B}}
\newcommand{\bF}{\mathbf{F}}
\newcommand{\bG}{\mathbf{G}}

\newcommand{\bfA}{\mathbf{A}}
\newcommand{\bfB}{\mathbf{B}}
\newcommand{\bfC}{\mathbf{C}}
\newcommand{\bfD}{\mathbf{D}}
\newcommand{\bfE}{\mathbf{E}}
\newcommand{\bfM}{\mathbf{M}}
\newcommand{\bfN}{\mathbf{N}}
\newcommand{\bfO}{\mathbf{O}}

\newcommand{\wsim}{\stackrel{w}{\sim}}

\newcommand{\plussim}{\stackrel{+}{\sim}}
\newcommand{\Lplussim}{\stackrel{\mathrm{L}+}{\sim}}
\newcommand{\Lsim}{\stackrel{\mathrm{L}}{\sim}}
\newcommand{\Lwsim}{\stackrel{\mathrm{Lw}}{\sim}}

\newcommand{\ra}{\rightarrow}

\newcommand{\Lra}{\Longrightarrow}

\newcommand{\lgl}{\langle}
\newcommand{\rgl}{\rangle}

\newcommand{\rl}{\!\downarrow\!}

\newcommand{\Fraisse}{Fra{\"{i}}ss{\'{e}}}
\newcommand{\Hubicka}{Hubi{\v{c}}ka}
\newcommand{\Lauchli}{L{\"{a}}uchli}

\newcommand{\Nesetril}{Ne{\v{s}}et{\v{r}}il}

\newcommand{\Rodl}{R{\"{o}}dl}
\newcommand{\cmark}{\ding{51}}%
\newcommand{\xmark}{\ding{55}}%
\newcommand{\EEAP}{SDAP}
\newcommand{\EEAPnonacronym}{Substructure  Disjoint Amalgamation Property}
\newcommand{\SFAP}{SFAP}
\newcommand{\SFAPnonacronym}{Substructure  Free  Amalgamation Property}
\newcommand{\BEEAP}{BSDAP}

\newcommand{\TwoEEAP}{$2$-SDAP}

\newcommand{\noprint}[1]{\relax}
\DeclareMathOperator{\Sim}{Sim}
\DeclareMathOperator{\Ext}{Ext}

\DeclareMathOperator{\type}{tp}

\DeclareMathOperator{\cl}{cl}

\DeclareMathOperator{\Emb}{Emb}
\DeclareMathOperator{\Forb}{Forb}

\newcommand{\eval}[2][\right]{\relax%
  \ifx#1\right\relax \left.\fi#2#1\rvert}

\begin{document}
\title[Simply characterized big Ramsey structures]{Fra{\"{i}}ss{\'{e}} structures with SDAP$^+$, Part II:\newline Simply characterized big Ramsey structures}
%[Indivisible Fra{\"{i}}ss{\'{e}} structures]{Fra{\"{i}}ss{\'{e}}  structures with SDAP$^+$, Part I:  Indivisibility}

\author{R. Coulson}
\address{United State Military Academy, West Point\\
Department of Mathematical Sciences, Thayer Hall 233, West Point, USA}
\email{rebecca.coulson@westpoint.edu}
% This is the correct URL for Becky!!
%\urladdr{\url{https://urldefense.com/v3/__https://www.westpoint.edu/mathematical-sciences/profile/rebecca_coulson__;!!NCZxaNi9jForCP_SxBKJCA!HwWeJtzB9RwKIYWqQrudCUhwx5IDHYmXmCQ4OqiI09mzN4qgFo68iElkw6LsA2lOAnpc$ }}
%\urladdr{\url{https://urldefense.com/v3/__https://www.westpoint.edu/mathematical-sciences/profile/rebecca_coulson__;!!NCZxaNi9jForCP_SxBKJCA!E2nu0Rcsfw15sLu6DrXLF-lHQGsf2b2Bh3Muun-jup9REtFQIVMZEMLjZ8j0vRkB-kAM$ }}
%\urladdr{\url{https://www.westpoint.edu/mathematical-sciences/profile/rebecca_coulson}}
\urladdr{\url{https://sites.google.com/view/rebeccacoulson/home}}

\author{N. Dobrinen}
\address{University of Denver\\%
Department of Mathematics, 2390 S. York St., Denver, CO USA}
\email{natasha.dobrinen@du.edu}
  \urladdr{\url{http://math.du.edu/~ndobrine}}

  \author{R. Patel}
  \address{%
 African Institute for Mathematical Sciences\\%
    M'bour-Thi\`{e}s, Senegal}
\email{rpatel@aims-senegal.org}

\thanks{The second author is grateful for support from   National Science Foundation Grant DMS-1901753, which also supported research visits to the University of Denver by the first and third authors.
She also is grateful for support
from Menachem Magidor for hosting her visit to
 The Hebrew University of Jerusalem  in December 2019, during which some of the ideas in this paper were formed.% 	
The third author's work on this paper was supported by the National Science Foundation under Grant No. DMS-1928930 while she was in residence at the Mathematical Sciences Research Institute in Berkeley, California, during the Fall 2020 semester.
}

%Becky: Took this out, it didn't render
%\subjclass[2010]{05D10, 05C55,  05C15, 05C05,  03C15, 03E75}

%Becky: this too
%\keywords{Ramsey theory, \Fraisse\ structure,  canonical partitions, big Ramsey degrees, trees}

%%%%Becky: took these out, don't seem relevant yet
%\date{Received on MONTH, YEAR}
%\issueinfo{VOL}{NUM}{MONTH}{YEAR}
%\doiinfo{10.1007/DOI-NUMBER}

\begin{abstract}
This is Part II of a two-part series regarding Ramsey properties of \Fraisse\ structures satisfying 
a property called SDAP$^+$, which strengthens the Disjoint  Amalgamation Property.
In Part I, we prove that every \Fraisse\ structure
in a finite relational language with relation symbols of any finite arity satisfying this property is indivisible.
In Part II, we prove 
 that
 every \Fraisse\  structure
 in a finite relational language with relation symbols of arity at most two having
this  property
has finite big Ramsey degrees which have a  simple characterization.
It follows that
any such \Fraisse\ structure
admits
 a big  Ramsey structure.
 Part II utilizes the notion of coding trees of $1$-types developed in Part I and a theorem from Part I which functions as a pigeonhole principle for induction arguments in this paper. 
Our approach yields
a direct
characterization of  the  degrees  without appeal to the standard method of ``envelopes''.
 This work offers a
 streamlined  and unifying approach to  Ramsey theory on some seemingly disparate classes of \Fraisse\ structures.
\end{abstract}
\maketitle
\tableofcontents

\section{Introduction}

\label{sec.intro}

This is Part II of a two-part series 
on a  property called SDAP$^+$ and its applications in the Ramsey theory of \Fraisse\ structures.
An overview of the area and motivations 
%is 
are 
provided in Section 1 of Part I, \cite{CDPI}.  
Here, we concentrate on big Ramsey degrees, building on 
 work developed in Part I.

The field of  big Ramsey degrees seeks to answer the question of which infinite structures carry analogues of the infinite Ramsey Theorem.

\begin{thm}[Ramsey, \cite{Ramsey30}]\label{thm.RamseyInfinite}
Given integers $k,r\ge 1$ and a coloring
of the $k$-element subsets of the natural numbers into $r$ colors,
there is an infinite set of natural numbers, $N$, such that all $k$-element subsets of $N$ have the same color.
\end{thm}

For infinite structures, exact analogues of  Ramsey's theorem usually fail for colorings of finite structures of size two or more, even when the class of finite substructures has the Ramsey property.
This is due to some unseen structure  which persists in  every
infinite substructure isomorphic to the original, but  which dissolves
when considering Ramsey properties of
classes of finite substructures.
The quest to characterize  this  often hidden but essential  structure is the area of
 {\em big Ramsey degrees}.

Given an infinite structure
$\bfM$, we say that
 $\bfM$ has {\em finite big Ramsey degrees} if for each finite substructure
$\bfA$ of $\bfM$, there is an integer $T$
such that the following  holds:
For any coloring of
the copies of $\bfA$ in $\bfM$
into finitely many colors, there is a
 substructure $\bfM'$ of $\bfM$
such that $\bfM'$ is isomorphic to $\bfM$, and
 the copies of $\bfA$ in $\bfM'$ take no more than $T$ colors.
When a $T$ having this property exists, the least  such value
is called the {\em big Ramsey degree} of $\bfA$ in
$\bfM$,
denoted
$T(\bfA, \bM)$.
In particular, if
the big Ramsey degree of $\bfA$ in $\bfM$ is one,
then any finite coloring of the copies of
$\bfA$ in $\bfM$ is
constant
on some subcopy of $\bfM$.

Big Ramsey degrees on infinite structures trace back to Sierpi\'{n}ski's result in the 1930's that
the big Ramsey degree for unordered pairs of rationals is at least two \cite{Sierpinski}.
For several decades, progress has been slow and sporadic.
However, big Ramsey degrees have received renewed focus due to
the flurry of
results in
 \cite{Laflamme/NVT/Sauer10},
\cite{Laflamme/Sauer/Vuksanovic06}, \cite{NVT08},
and \cite{Sauer06}
in tandem with
the publication
of
 \cite{Kechris/Pestov/Todorcevic05}, in which
 Kechris, Pestov, and Todorcevic
asked for an analogue of their
correspondence between
 the Ramsey property  of \Fraisse\ classes and  extreme amenability to the setting of
 big Ramsey degrees for \Fraisse\ limits.
 This was
 addressed
 by Zucker in \cite{Zucker19},
 where he proved a
 connection
between
 \Fraisse\
 limits
 with finite big Ramsey degrees and completion flows in topological dynamics.
Zucker's results apply to {\em big Ramsey structures},
expansions of  \Fraisse\
limits
in which the big Ramsey degrees of the \Fraisse\
limits
can be exactly characterized using the additional structure induced by the expanded language.
This additional structure
involves a well-ordering, and
characterizes
the  essential structure which persists  in every infinite subcopy of the \Fraisse\
limit.
It is this essential structure  we seek to understand
 in the study of big Ramsey degrees.

In Part I, 
we  described
an amalgamation property,
called the \EEAPnonacronym\ (\EEAP),
forming a strengthened version of  disjoint  amalgamation.
The \Fraisse\ limit of a \Fraisse\ class satisfying \EEAP\ is said to satisfy \EEAP$^+$ if it satisfies two additional properties,
 which  we call
the Diagonal Coding Tree Property and the Extension Property.
A related  property, called the Labeled Substructure Disjoint Amalgamation Property$^+$  (LSDAP$^+$), was also introduced in Part I. 
We will recall the main definitions in Section \ref{subsec.Fcrs},
referring the reader to Part I for the full exposition.

A \Fraisse\ limit is  called
{\em indivisible} if  every one-element substructure of $\bK$ has big Ramsey degree equal to one.
In Part I, we proved indivisibility for  all 
\Fraisse\ limits in finite relational languages with relation symbols of any finite arity satisfying \EEAP$^+$.

\begin{thm}\label{thm.indivisibility}
Suppose  $\mathcal{K}$ is a  \Fraisse\ class
in a finite relational language
with relation symbols  in any arity such that
its  \Fraisse\ limit $\bK$  satisfies
\EEAP$^+$.
Then $\bK$ is indivisible.
\end{thm}

In this paper, 
we characterize the exact big Ramsey degrees for all \Fraisse\  limits in finite relational languages
with relation symbols of arity at most two
satisfying \EEAP$^+$ or LSDAP$^+$.
 Our characterization, together with results of
Zucker in \cite{Zucker19}, imply that such  \Fraisse\ limits further admit big Ramsey structures, and their automorphism groups have  metrizable universal completion flows.

\begin{thm}\label{thm.main}
Let $\mathcal{K}$ be a  \Fraisse\  class
in a finite relational language
with relation symbols of arity at most two
%such that $\mathcal{K}$ satisfies
%\EEAP, and 
such that the \Fraisse\ limit  $\bK$ of $\mathcal{K}$ 
%has
%the Diagonal Coding Tree Property and the Extension Property.
has \EEAP$^+$ or LSDAP$^+$.
Then
$\bK$
has finite big Ramsey degrees which have a simple characterization and, moreover,
admits  a big Ramsey structure.
Hence, the topological group
$\rm{Aut}(\bK))$
 has a metrizable universal completion flow, which is unique up to isomorphism.
\end{thm}

Theorem \ref{thm.main} provides new classes of examples
 of big Ramsey structures while
 recovering results in
 \cite{DevlinThesis},
\cite{HoweThesis},
 \cite{Laflamme/NVT/Sauer10}, and
 \cite{Laflamme/Sauer/Vuksanovic06}
and extending special cases of  results in \cite{Zucker20} to obtain exact big Ramsey degrees.
It will also follow from Theorem \ref{thm.main} that \Fraisse\ limits satisfying LSDAP$^+$ are indivisible.
Theorem  5.4  in \cite{CDPI} (the Level Set Ramsey Theorem  from Part I) will serve as the starting point for the proof of
Theorem \ref{thm.main}.

We now discuss several previous
 theorems which are recovered by  Theorem   \ref{thm.main},
 as well as  new examples
 obtained  from our results.
% A fuller description is provided in Section \ref{sec.EEAPClasses}.
In Proposition \ref{prop.LSVSFAP} we will  show 
 that \EEAP$^+$ holds for  disjoint
 amalgamation classes which are ``unrestricted''
(see Definition \ref{defn.unconst}), as well as their ordered versions.
Examples 
of unrestricted structures
include
 classes of
  structures with finitely many unary and binary relations such as
graphs, directed graphs,
tournaments, graphs with finitely many edge relations, etc.
In particular, Theorem \ref{thm.main}  recovers 
the work of Laflamme, Sauer, and Vuksanovic in  \cite{Laflamme/Sauer/Vuksanovic06}, which characterized the big Ramsey degrees of the unrestricted \Fraisse\ classes with finitely many binary relations, and provides new results for their ordered versions. 
Theorem \ref{thm.main} also applies to  $k$-partite graphs as well as their ordered versions, as these structures also satisfy SDAP$^+$.
These big Ramsey degree results are presented in Theorem \ref{thm.supercool}.

The existence of upper bounds for $k$-partite graphs follows from a more general result obtained by Zucker in \cite{Zucker20}, where he found upper bounds for the big Ramsey degrees for \Fraisse\ classes with relations of arity at most two satisfying free amalgamation. 
After the  announcement of our results in Parts I and  II in the 2020 version   \cite{CDP20}, our result on the exact big Ramsey degrees for $k$-partite graphs has  been recovered in the 2021 work of Balko, Chodounsk\'{y}, Dobrinen, Hubi\v{c}ka, Kone\v{c}n\'{y}, Vena, and Zucker in \cite{Balko7}, which yields exact big Ramsey degrees for \Fraisse\ classes with relations of arity at most two satisfying free amalgamation. 
We give here a succinct characterization of the big Ramsey degrees for k-partite graphs.

In Proposition \ref{prop.FA}
we will  show that \EEAP$^+$ holds for \Fraisse\
limits
of  free amalgamation classes which  forbid $3$-irreducible substructures, namely, substructures in which any three distinct elements appear  in a tuple of which some relation holds, as well as their ordered versions.
This provides a large class of indivisible \Fraisse\ limits, by Theorem \ref{thm.indivisibility} in Part I.
%For binary relational structures, this class is subsumed in the class of unrestricted \Fraisse\ classes and their ordered versions. 

Certain
\Fraisse\ structures
derived from the rational linear order
 have enough rigidity, similarly to $\bQ$,  for either
\EEAP$^+$ or  LSDAP$^+$ to hold, hence producing big Ramsey structures with simple characterizations.
 These results are consolidated in Theorem \ref{thm.LOEqRels}.

Theorem \ref{thm.LOEqRels} 
 shows
that  the structure
$\bQ_{\bQ}$ admits a big Ramsey structure,
answering  a question raised by Zucker at the 2018 Banff Workshop on {\em Unifying Themes in Ramsey Theory}.
This structure $\bQ_{\bQ}$ is the dense linear order without endpoints with an equivalence relation such that  all equivalence classes are convex copies of the rationals.
More generally,  Theorem \ref{thm.LOEqRels} 
 applies to
members of
a natural hierarchy  of infinite structures with  finitely many convexly ordered equivalence relations, where each  successive  equivalence relation coarsens the previous one;
these also admit big Ramsey structures with simple characterizations.

Known results  which
Theorem \ref{thm.LOEqRels} 
recovers include
Devlin's  characterization of  the  big Ramsey degrees of the rationals \cite{DevlinThesis};
 results
 of Laflamme, Nguyen Van Th\'{e}, and Sauer  in
 \cite{Laflamme/NVT/Sauer10}
 characterizing the big Ramsey degrees of
 the $\bQ_n$;
and  a result of Zucker in \cite{Zucker19}  showing that 
$\bQ_n$, the rational
linear order with a partition into $n$ dense
pieces,
admits a big Ramsey structure with a simple characterization.

While many of the known big Ramsey degree results use sophisticated versions of Milliken's Ramsey theorem for trees \cite{Milliken79},
and while proofs using the method of forcing to produce new pigeonhole principles in ZFC have appeared in \cite{DobrinenRado19}, \cite{DobrinenH_k19}, \cite{DobrinenJML20}, and  \cite{Zucker20},
 our approach produces  a clarity about  big Ramsey degrees for structures satisfying SDAP$^+$ or LSDAP$^+$.
Given a \Fraisse\ class $\mathcal{K}$,
we fix an  enumerated \Fraisse\ limit of $\mathcal{K}$, which we denote by $\bK$.
By {\em enumerated \Fraisse\ limit}, we mean that the universe of $\bK$ is ordered via the natural numbers. 
By working 
with  trees of  quantifier-free $1$-types 
and the Level Set Ramsey Theorem   from Part I,
we will find the exact big Ramsey degrees directly from the  diagonal coding trees of $1$-types, without appeal to the standard method of ``envelopes''.
This means that
the upper bounds which we  find via  induction starting with the Level Set Ramsey Theorem 
 are shown to be exact.

Using trees
of
quantifier-free
$1$-types
(partially ordered by inclusion)
allows us to prove   a  characterization of  big Ramsey degrees for \Fraisse\ classes with \EEAP$^+$ or LSDAP$^+$ which is a simple extension of the so-called ``Devlin types'' for the rationals in \cite{DevlinThesis},
and of the  characterization of the big Ramsey degrees of the  Rado graph achieved by Laflamme, Sauer, and Vuksanovic in \cite{Laflamme/Sauer/Vuksanovic06}.
Here, we present  the characterization for  structures  without unary relations.
The full  characterization is given in Theorem \ref{thm.bounds}.

\begin{CharBRDSDAP}
Let $\mathcal{L}$   be a  language
consisting of
finitely many  relation symbols, each of arity
two.
Suppose $\mathcal{K}$ is a \Fraisse\ class in $\mathcal{L}$
such that the \Fraisse\ limit
$\bK$
of $\mathcal{K}$ satisfies
\EEAP$^+$ or LSDAP$^+$.
Fix a structure  $\bfA\in\mathcal{K}$.
 Let
 $(\bfA,<)$ denote
  $\bfA$ together with a fixed enumeration
  $\lgl\mathrm{a}_i:i<n\rgl$
 of the universe
 of $\bfA$.
We say that a tree $T$ is a
{\em diagonal tree coding $(\bfA,<)$}
if  the following hold:
\begin{enumerate}
\item
$T$ is a finite tree with $n$ terminal nodes and
 branching degree two.
\item
$T$ has  at most one branching node  in
any given
level,
and
no two distinct nodes from among the branching nodes and terminal nodes have the same length.
Hence, $T$ has $2n-1$ many levels.
\item
Let $\lgl \mathrm{d}_i:i<n\rgl$ enumerate the terminal nodes in $T$ in order of increasing length.
Let $\bfD$ be the $\mathcal{L}$-structure induced on the set $\{\mathrm{d}_i:i<n\}$ by the increasing
bijection from   $\lgl\mathrm{a}_i:i<n\rgl$ to $\lgl \mathrm{d}_i:i<n\rgl$, so that $\bfD \cong \bfA$.
Let $\tau_i$ denote the quantifier-free
$1$-type of $\mathrm{d}_i$ over
$\bfD_i$,
the substructure of
$\bfD$ on vertices $\{\mathrm{d}_m:m<i\}$.
Given  $i<j<k<n$,
if $\mathrm{d}_j$ and $\mathrm{d}_k$
both
extend
some node
in $T$
that is at the same level as $\mathrm{d}_i$,
then $\mathrm{d}_j$ and $\mathrm{d}_k$
 have the same quantifier-free $1$-types over
$\mathbf{D}_i$.
That is,
$\tau_j\re \mathbf{D}_i
=\tau_k\re \mathbf{D}_i$.
\end{enumerate}
Let $\mathcal{D}(\bfA,<)$ denote the number of distinct diagonal trees coding $(\bfA,<)$;
let $\mathcal{OA}$ denote a set
consisting of
one representative from each isomorphism class of ordered copies of $\bfA$.
Then
 $$
 T(\bfA,\bK)
=
\sum_{(\bfA,<)\in\mathcal{OA}}\mathcal{D}(\bfA,<)
$$
\end{CharBRDSDAP}

If $\mathcal{L}$ also has unary relation symbols,
in the case that
$\mathcal{K}$ is a free amalgamation class, the simple characterization above holds when modified to  diagonal coding trees with the same number of roots as unary relations.
In the case that $\mathcal{K}$ contains a transitive relation,
then the  above characterization still holds.
We  show
in Theorem \ref{thm.apply}
that there is a simple way of recovering Zucker's criterion   for existence of big Ramsey structures (which uses colorings of embeddings; see Theorem 7.1 in \cite{Zucker19})
  from
 our
canonical partitions for colorings of copies of a structure.

 We see  our main contribution in this paper as providing a clear and   unified analysis of
a wide class of \Fraisse\ structures with relations of
arity at most two for which the big Ramsey degrees have a simple characterization.
 \vskip.1in

%%%%%%%%%%%%%%%%%%%%%%%%%%%%%%%%%%
%%%%%%%%%%%%%%%%%%%%%%%%%%%%%%%%%%%%%
%%%%%%%%%%%%%%%%%%%%%%%%%%%%%%%%%%%%%%%

%{\color{purple}  Put an intro here.}

\section{Big Ramsey degrees and structures, and brief background from Part I}\label{subsec.Fcrs}

All relations in this paper will be of arity one or two,
and all languages will consist of finitely many
relation symbols (and no constant or function symbols).
%We use the set-theoretic notation $\om$ to denote the set of natural numbers, $\{0,1,2,\dots\}$ and treat $n \in \om$ as the set $\{i\in\om:i<n\}$.
Subsection 2.1  of \cite{CDPI} provides details on \Fraisse\ theory.

Given a \Fraisse\ class $\mathcal{K}$ and substructures
$\bfM,\bfN$
of $\bK$  (finite or infinite)
 with $\bfM\le\bfN$,
we use
 %${\bfN\choose\bfM}$
 $\binom{\bfN}{\bfM}$
to denote the set of all substructures of
$\bfN$ which are isomorphic to
$\bfM$. Given
$\bfM\le\bfN\le\bfO$,
substructures of $\bK$, we write
$$
\bfO\ra(\bfN)_{\ell}^{\bfM}
$$
to denote that for each coloring of
$\binom{\bfO}{\bfM}$
%${\bfO\choose \bfM}$
into $\ell$ colors, there is an
 %$\bfN' \in {\bfO\choose\bfN}$
 $\bfN' \in \binom{\bfO}{\bfN}$
 %{\bfO\choose\bfN}$
 such that
$\binom{\bfN'}{\bfM}$
%${\bfN'\choose\bfM}$
is  {\em monochromatic}, meaning that
 all members of
 $\binom{\bfN'}{\bfM}$
% ${\bfN'\choose\bfM}$
 have the same color.

\begin{defn}\label{defn.RP}
A \Fraisse\ class  $\mathcal{K}$ has the {\em Ramsey property} if  for any two structures $\bfA\le\bfB$ in $\mathcal{K}$ and any  $\ell \ge 2$,
there is a $\bfC\in\mathcal{K}$ with $\bfB\le\bfC$ such that
$\bfC\ra (\bB)^{\bfA}_\ell$.
\end{defn}

Equivalently, $\mathcal{K}$ has the Ramsey property if for any two structures $\bfA\le \bfB$ in $\mathcal{K}$,
\begin{equation}\label{eq.RP}
\forall \ell\ge 2,\ \ {\bK}\ra ({\bfB})^{\bfA}_{\ell}.
\end{equation}
This equivalent formulation makes comparison with big Ramsey degrees, below, quite  clear.

\begin{defn}[\cite{Kechris/Pestov/Todorcevic05}]\label{defn.bRd}
Given a \Fraisse\ class $\mathcal{K}$ and its \Fraisse\ limit $\bK$,
for any $\bfA\in\mathcal{K}$,
write
\begin{equation}\label{eq.bRd}
\forall \ell\ge 1,\ \ {\bK}\ra ({\bK})^{\bfA}_{\ell,T}
\end{equation}
when there is an
 integer $T\ge 1$ such that
for any integer $\ell \ge 1$,
given any coloring of 
$\binom{\bK}{\bfA}$
%${\bK\choose \bfA}$ 
into $\ell$ colors,
there is a
 substructure $\bK'$ of $\bK$, isomorphic to $\bK$,  such that 
 $\binom{\bK'}{\bfA}$
 %${\bK'\choose \bfA}$
 takes no more than $T$ colors.
We say that
 $\bK$ has {\em finite big Ramsey degrees} if for each
 $\bfA \in \mathcal{K}$,
there is an integer $T\ge 1$
such that
equation
(\ref{eq.bRd}) holds.
For a given finite $\bfA\le\bK$, when
such a $T$ exists,
we  let
$T(\bfA,\bK)$  denote the least
one, and call this number the {\em big Ramsey degree} of $\bfA$ in $\bK$.
\end{defn}

Comparing equations (\ref{eq.RP}) and
 (\ref{eq.bRd}),
 we see
that the
 difference between the Ramsey property and having finite big Ramsey degrees  is that
 the former finds a substructure of $\bK$ isomorphic to the {\em finite} structure $\bfB$ in which all copies of $\bfA$ have the {\em same} color, while
  the latter finds an {\em infinite} substructure of $\bK$ which is  isomorphic to $\bK$ in which the copies of $\bfA$ take {\em few} colors.
  It is only when  $T(\bfA,\bK)=1$ that
  there is  a subcopy of $\bK$ in which all copies of $\bfA$ have the same color.

It is normally the case  that for structures $\bfA$ with universe of size greater than one, $T(\bfA,\bK)$ is at least two, if it exists at all.
The fundamental reason
for this stems from Sierpi\'{n}ski's example that $T(2,\bQ)\ge 2$:
The enumeration of the universe $\om$ of $\bK$ plays against the relations in the structure to preserve more than one color in every subcopy of $\bK$.

A proof that $\bK$ has finite big Ramsey degrees amounts to showing that the numbers $T(\bfA, \bK)$ exist by finding upper bounds for them.
When  a method for producing the numbers
 $T(\bfA,\bK)$
 is given,
 we will say that the {\em exact big Ramsey degrees}
 have been {\em characterized}.
In all known cases where exact big Ramsey degrees have been characterized, this has been done by finding {\em  canonical partitions}
for the finite substructures of $\bK$.

\begin{defn}[Canonical Partition]\label{defn.cp}
Let $\mathcal{K}$ be a \Fraisse\ class with \Fraisse\ limit $\bK$, and let
$\bA\in\mathcal{K}$ be given.
A partition
$\{P_i:i<n\}$
 of 
$\binom{\bK}{\bA}$
% ${\bK\choose \bA}$
 is a  {\em canonical partition}
if the following hold:
\begin{enumerate}
\item
For every subcopy
$\bJ$ of $\bK$
and each $i < n$,
$P_i\cap {\binom{\bJ}{\bA}}$
% $P_i\cap {\bJ\choose\bA}$
 is non-empty.
 This property is called {\em persistence}.
\item
For each finite coloring $\gamma$ of
$\binom{\bK}{\bA}$
%${\bK\choose \bA}$
there is a subcopy
$\bJ$ of $\bK$
such that
for each $i<n$, all members of $P_i\cap \binom{\bJ}{\bA}$
%{\bJ\choose\bA}$
are assigned the same color by $\gamma$.
 \end{enumerate}
\end{defn}

\begin{rem}\label{rem.embvscopy}
In many papers on big Ramsey degrees, including the 
 foundational results in \cite{DevlinThesis}, \cite{Sauer06}, and \cite{Laflamme/Sauer/Vuksanovic06}, 
authors color {\em copies} of a given  $\bfA\in\mathcal{K}$ inside $\bK$, working with Definition \ref{defn.bRd}.
 More recently,  especially in papers  with direct ties to topological dynamics of automorphism groups  as in \cite{Zucker19} and \cite{Zucker20},
 authors color
{\em embeddings}  of $\bfA$ into $\bK$.
The relationship between these approaches is  simple:
A structure  $\bfA\in\mathcal{K}$ has
big
Ramsey degree $T$
for copies if and only if $\bfA$ has
big
Ramsey degree $T \cdot |$Aut$(\bfA)|$
for embeddings.
Thus, one can use whichever formulation most suits the context.
Furthermore,  we show
in Theorem \ref{thm.apply}
that there is a simple way of recovering Zucker's criterion   for existence of big Ramsey structures (which uses colorings of embeddings; see Theorem 7.1 in \cite{Zucker19})
  from
 our
canonical partitions for colorings of copies of a structure.
\end{rem}

The majority of  results on  big Ramsey degrees have been proved  using some auxiliary structure, usually trees, and recently sequences of parameter words  (see \cite{Hubicka_CS20}),  to characterize
 the persistent superstructures which code
 the finite structure
 $\bfA$.
 The exception is  the recent use of category-theoretic approaches (see for instance \cite{Barbosa20},
  \cite{Masulovic18}, and \cite{Masulovic_RBS20}).
 These superstructures  fade away
 in the case of finite structures with the Ramsey property.
An example of how this works can be seen in Theorem \ref{thm.SESAPimpliesORP}, where we recover the ordered Ramsey property  for ages of  \Fraisse\
structures
with \EEAP$^+$ from
 their big Ramsey degrees.
 However, for
 big Ramsey degrees of
  \Fraisse\ limits,
 these superstructures
possess some    essential features which persist,
leading to big Ramsey degrees
greater than one.
The following notion of Zucker  deals with such superstructures via expanded languages.

Let $\mathcal{L}$ be a relational language, $\M$ a set, $\bfN$ an $\mathcal{L}$-structure, and $\iota : \M \ra \N$ an injection.
Write $\bfN \cdot \iota$ for the unique $\mathcal{L}$-structure having underlying set $\M$ such that $\iota$ is an embedding of
$\bfN \cdot \iota$ into $\bfN$.

\begin{defn}[Zucker, \cite{Zucker19}]\label{defn.bRs}
Let $\bK$ be a \Fraisse\ structure in a
relational
language  $\mathcal{L}$ with $\mathcal{K}=$ Age$(\bK)$.
We say that $\bK$ {\em admits a big Ramsey structure} if there is a
relational
language $\mathcal{L}^* \contains\mathcal{L}$ and an $\mathcal{L}^* $-structure $\bK^*$ so that the following hold:
\begin{enumerate}
\item
The reduct of $\bK^*$ to the language $\mathcal{L}$ equals $\bK$.
\item
Each $\bfA\in\mathcal{K}$ has finitely many
expansions to an $\mathcal{L}^*$-structure
$\bfA^*\in$ Age$(\bK^*)$;
denote the set of such expansions by $\bK^*(\bfA)$.
\item
For each $\bfA \in \mathcal{K}$,
$T(\bfA, \bK) \cdot |\text{Aut}(\bfA)| = |\bK^*(\bfA)|$
\item\label{witnessing}
For each $\bfA \in \mathcal{K}$, the function
$\gamma:\Emb(\bfA,\bK)\ra\bK^*(\bfA)$
given by $\gamma(\iota)=\bK^*\cdot \iota$
witnesses the fact that
$$
T(\bfA, \bK) \cdot |\text{Aut}(\bfA)| \ge |\bK^*(\bfA)|,
$$
in the following sense:
For every subcopy $\bK'$ of $\bK$, the image of the restriction of $\gamma$ to
$\text{Emb}(\bfA, \bK')$ has size  $|\bK^*(\bfA)|$.
\end{enumerate}
Such a structure $\bK^*$ is called
a {\em big Ramsey structure} for $\bK$.
\end{defn}

Note that the definition of a big Ramsey structure for $\bK$ presupposes that $\bK$ has finite big Ramsey degrees.  The big Ramsey structure $\bK^*$, when it exists, is  a device for storing information about all the big Ramsey degrees in
$\bK$ together in a uniform way.

While the study of big Ramsey degrees has been progressing for many decades,
a recent   compelling
motivation for finding big Ramsey structures  is the following theorem.

\begin{thm}[Zucker, \cite{Zucker19}]\label{thm.Zucker}
Let $\bK$ be a \Fraisse\ structure which admits a big Ramsey structure, and let $G=$ {\rm Aut}$(\bK)$.
Then the topological group $G$ has a metrizable universal completion flow, which is unique up to isomorphism.
\end{thm}

This theorem
 answered  one direction of a question in \cite{Kechris/Pestov/Todorcevic05}
 which asked for an analogue, in the context of finite big Ramsey degrees, of the
 Kechris-Pestov-Todorcevic correspondence between the Ramsey property for a \Fraisse\ class and extreme amenability of
 the automorphism group of its \Fraisse\ limit;
 Zucker's theorem provides a  connection
 between finite big Ramsey degrees and universal completion flows.
The notion of big Ramsey degree in \cite{Zucker19}
involves colorings of embeddings of structures instead of just colorings of  substructures.
As
described
in Remark \ref{rem.embvscopy},
 this poses no problem  when applying  our results  on big Ramsey degrees,
 which involve coloring copies of a structure,
 to Theorem  \ref{thm.Zucker}.

%%%%%%%%%%%%%%%%%%%%%%%%%%%%%%%%%%%%%%%
%%%%%%%%%%%%%%%%%%%%%%%%%%%%%%%%%%%%%%%
%%%%%%%%%%%%%%%%%%%%%%%%%%%%%%%%%%%%%%%

\subsection{Brief background on  SDAP$^+$ and LSDAP$^+$}

This Subsection recalls some key notions  from \cite{CDPI} for the reader's convenience.   The reader is referred to Part I for the full exposition.

Recall the two amalgamation properties first introduced in  Subsection 2.2 of Part I, \cite{CDPI}.

\begin{defn}[\SFAP]\label{defn.SFAP}
A \Fraisse\ class $\mathcal{K}$ has the
{\em  \SFAPnonacronym\ (\SFAP)} if $\mathcal{K}$ has free amalgamation,
and given  $\bfA,\bfB,\bfC,\bfD\in\mathcal{K}$, the following holds:
Suppose
\begin{enumerate}
\item[(1)]
$\bfA$  is a substructure of $\bfC$, where
 $\bfC$ extends  $\bfA$ by two vertices,
say $\mathrm{C}\setminus\mathrm{A}=\{v,w\}$;

\item[(2)]
 $\bfA$  is a substructure of $\bfB$ and
 $\sigma$ and $\tau$  are
 $1$-types over $\bfB$  with   $\sigma\re\bfA=\type(v/\bfA)$ and $\tau\re\bfA=\type(w/\bfA)$; and
\item[(3)]
$\bfB$ is a substructure of $\bfD$ which extends
 $\bfB$ by one vertex, say $v'$, such that $\type(v'/\bfB)=\sigma$.

\end{enumerate}
  Then there is
an   $\bfE\in\mathcal{K}$ extending  $\bfD$ by one vertex, say $w'$, such that
  $\type(w'/\bfB)=\tau$, $\bfE\re (
  \mathrm{A}\cup\{v',w'\})\cong \bfC$,
  and $\bfE$ adds no other relations over $\mathrm{D}$.
\end{defn}

\begin{defn}[\EEAP]\label{defn.EEAP_new}
A \Fraisse\ class $\mathcal{K}$ has the
{\em  \EEAPnonacronym\ (\EEAP)} if $\mathcal{K}$
has disjoint amalgamation,
and the following holds:
Given   $\bfA, \bfC\in\mathcal{K}$, suppose that
 $\bfA$ is a substructure of $\bfC$, where   $\bfC$ extends  $\bfA$ by two vertices, say $v$ and $w$.
Then there exist  $\bfA',\bfC'\in\mathcal{K}$, where
$\bfA'$
contains a copy of $\bfA$ as a substructure
and
$\bfC'$ is a disjoint amalgamation of $\bfA'$ and $\bfC$ over $\bfA$, such that
letting   $v',w'$ denote the two vertices in
 $\mathrm{C}'\setminus \mathrm{A}'$ and
assuming (1) and (2), the conclusion holds:
 \begin{enumerate}
 \item[(1)]
Suppose
$\bfB\in\mathcal{K}$  is any structure
 containing $\bfA'$ as a substructure,
and let
 $\sigma$ and $\tau$  be
  $1$-types over $\bfB$  satisfying    $\sigma\re\bfA'=\type(v'/\bfA')$ and $\tau\re\bfA'=\type(w'/\bfA')$,
\item[(2)]
Suppose
$\bfD\in \mathcal{K}$  extends  $\bfB$
by one vertex, say $v''$, such that $\type(v''/\bfB)=\sigma$.
\end{enumerate}
Then
  there is
an  $\bfE\in\mathcal{K}$ extending   $\bfD$ by one vertex, say $w''$, such that
  $\type(w''/\bfB)=\tau$ and  $\bfE\re (
  \mathrm{A}\cup\{v'',w''\})\cong \bfC$.
\end{defn}

The  definitions of \SFAP\  and \EEAP\ can be stated using embeddings rather than substructures
in the standard way, but this presentation is more in-line with our applications.
Recall from Part I that 
\SFAP\ implies \EEAP\
 and  that \SFAP\ and \EEAP\ are each
preserved  under free superposition.

The following notion of coding tree of $1$-types was presented in Definition 3.1 of Part I.

\begin{defn}[The  Coding  Tree of $1$-Types, $\bS(\bK)$]\label{defn.treecodeK}
The {\em coding  tree of $1$-types}
$\bS(\bK)$
for an enumerated \Fraisse\ structure $\bK$
 is the set of all complete
  $1$-types over initial segments of $\bK$
along with a function $c:\om\ra \bS(\bK)$ such that
$c(n)$ is the
$1$-type
of $v_n$
over $\bK_n$.
The tree-ordering is simply inclusion.
\end{defn}

The next several definitions were presented in Subsection 4.1 of Part I.

\begin{defn}[The Unary-Colored  Coding Tree of $1$-Types,
$\bU(\bK)$]\label{defn.ctU}
Let $\mathcal{K}$ be a \Fraisse\ class in language $\mathcal{L}$ and $\bK$ an enumerated \Fraisse\ structure for $\mathcal{K}$.
For $n < \om$,
let $c_n$ denote the $1$-type
of $v_n$
over $\bK_n$
(exactly as in the definition of $\bS(\bK)$).
Let $\mathcal{L}^-$ denote
the collection of all relation symbols in $\mathcal{L}$ of arity greater than one,
and let $\bK^-$ denote the reduct of $\bK$ to $\mathcal{L}^-$
and $\bK_n^-$ the reduct of $\bK_n$ to $\mathcal{L}^-$.

For $n<\om$, define the {\em $n$-th level, $\bU(n)$}, to  be
the collection
of all $1$-types
$s$ over $\bK^-_n$
in the language $\mathcal{L}^-$
such that
for some $i\ge n$,
$v_i$ satisfies $s$.
Define $\bU$
 to be
$\bigcup_{n<\om}\bU(n)$.
The tree-ordering on $\bU$ is simply inclusion.
The {\em unary-colored coding tree of $1$-types}
is
the tree $\bU$ along with the  function $c:\om\ra \bU$ such that $c(n)=c_n$.
Thus,
$c_n$ is the $1$-type
(in the language $\mathcal{L}^-$) of $v_n$
 in $\bU(n)$  along with the additional ``unary color''
$\gamma\in\Gamma$ such that
$\gamma(v_n)$ holds in $\bK$.
 \end{defn}

\begin{defn}[Diagonal tree]\label{def.diagskew}
We call a subtree $T\sse \bS$
or $T\sse\bU$
{\em diagonal}
if each level of $T$ has at most one splitting node,
each splitting node in $T$ has degree two (exactly two immediate successors), and
coding node levels in $T$ have no splitting nodes.
\end{defn}

\begin{notation}\label{notn.cong<}
Given a diagonal subtree $T$
(of $\bS$ or $\bU$)
with coding nodes,
we let
$\lgl c^T_n: n<N\rgl$,
where $N\le\om$,
  denote the enumeration of the coding nodes in $T$ in order of increasing length.
Let $\ell^T_n$ denote  $|c^T_n|$, the {\em length} of $c^T_n$.
We shall call a node in $T$ a {\em critical node} if it is either a splitting node or a coding node in $T$.
Let
\begin{equation}
\widehat{T}=\{t\re n:t\in T\mathrm{\ and \ } n\le |t|\}.
\end{equation}
Given
$s\in T$ that is not a splitting node in $T$,
we let $s^+$ denote the immediate successor of $s$ in $\widehat{T}$.
Given any
$\ell$,
we let $T\re\ell$ denote the set of those nodes in $\widehat{T}$ with length $\ell$,
and we let
  $T\rl \ell$
  denote the
  union of the
  set of nodes in $T$ of length less than  $\ell$
  with the set $T\re\ell$.
\end{notation}

%Extending
%Notation \ref{notn.KreA} to subtrees $T$ of either $\bS$ or $\bU$, 
We write $\bK\re T$ to denote the substructure of
$\bK$ on $\mathrm{N}^T$, the set of vertices of $\bK$ represented by the coding nodes in $T$.

\begin{defn}[Diagonal Coding Subtree]\label{defn.sct}
A subtree $T\sse\bU$ is called a {\em diagonal coding subtree} if $T$ is diagonal and  satisfies the following properties:
  \begin{enumerate}
 \item
   $\bK\re T\cong\bK$.
  \item
  For each $n<\om$, the collection of $1$-types
  in
  $T\re (\ell^T_n+1)$ over $\bK\re (T\rl \ell^T_n)$
   is in
    one-to-one correspondence with the collection of
   $1$-types in $\bU(n+1)$.
   \item[(3)]
Given $m<n$
and letting
  $A:=T\rl(\ell^T_m-1)$,
if $c^T_n\contains c^T_m$
 then
 $$
  (c^{T}_n)^+(c^{T}_n; A)
  \sim
  (c^{T}_m)^+(c^{T}_m; A).
$$
\end{enumerate}
Likewise, a subtree $T\sse\bS$ is a {\em diagonal coding subtree} if the above hold with $\bU$ replaced by $\bS$.
\end{defn}

Recall that 
requirement (3) can be  met by
the \Fraisse\ limit of
any \Fraisse\ class satisfying \EEAP.

We say that a tree  $T$ is  {\em perfect}  if $T$ has  no terminal nodes, and  each node in  $T$ has  at least  two  incomparable extensions in $T$.

\begin{defn}[Diagonal Coding Tree Property]\label{defn.DCTP}
A \Fraisse\ class $\mathcal{K}$ in language $\mathcal{L}$  satisfies the {\em Diagonal Coding Tree Property}
if  given any enumerated \Fraisse\ structure $\bK$
for $\mathcal{K}$,
there is a diagonal coding
subtree $T$ of either $\bS$ or $\bU$
such that $T$ is perfect.
\end{defn}

\begin{defn}[The Space of  Diagonal Coding Trees of $1$-Types, $\mathcal{T}$]\label{def.subtree}
Let $\bK$ be any enumerated \Fraisse\  structure
and let $\bT$ be a fixed diagonal coding subtree of $\bU$.
Then the space of coding trees
 $\mathcal{T}(\bT)$ consists of all   subtrees $T$ of $\bT$ such that
 $T\sim\bT$.
Members of $\mathcal{T}(\bT)$ are called simply {\em coding trees}, where diagonal is understood to be  implied.
We shall usually simply write $\mathcal{T}$ when $\bT$ is clear
from context.
For $T\in \mathcal{T}$,  we write
$S\le T$ to mean that  $S$ is a  subtree of $T$ and $S$ is a member of $\mathcal{T}$.
\end{defn}

We will work in  a diagonal coding subtree of $\bS$ whenever such a subtree exists.
This is always the case for \Fraisse\ classes satisfying SFAP.
For \Fraisse\ limits   with no unary relations satisfying SDAP,
note that  $\bS=\bU$; so in this case, a diagonal coding subtree of $\bU$ is the same as a diagonal coding subtree of $\bS$.
If $\bK$ is a \Fraisse\ class   with unary relations  satisfying SDAP and there is 
 a diagonal coding subtree of $\bU$ but 
no diagonal coding subtree of $\bS$,
then there  are subsets $P_0,\dots, P_j$ of the unary 
 relation symbols of $\mathcal{K}$ 
and a  diagonal coding
subtree $T\sse \bU$ such that  at some level  $\ell$  below the first coding node of $T$,  the following hold:
$T\re\ell$ has exactly $j+1$ nodes, say $t_0,\dots,t_j$, and for each $i\le j$, 
every coding node in the tree $T$ restricted above $t_i$  has unary relation in $P_i$ and moreover, each of the unary relations in $P_i$ occurs densely in  $T$  restricted above $t_i$.
By possibly adding unary relation symbols, we may assume that $P_0,\dots,P_j$ is a partition of the unary relation symbols. 
Thus, without loss of generality, we will hold to the following  convention  for the remainder of this article.

\begin{convention}\label{conv.Gamma_ts}
Let $\mathcal{K}$ be a \Fraisse\ class in a language $\mathcal{L}$ and $\bK$ a \Fraisse\ limit of $\mathcal{K}$.
Either there is a diagonal coding subtree  $\bT$ of $\bS(\bK)$,
or else there is a diagonal coding subtree  $\bT$ of $\bU(\bK)$  and if there are any unary relations, then  each  unary relation occurs densely in $\bT$.
 \end{convention}

The following definitions are from Subsection 4.2 of Part I.

The following extends Notation
\ref{notn.cong<}
to
subsets of trees.
For a  finite subset $A\sse\bT$,  let
\begin{equation}
 \ell_A=\max\{|t|:t\in A\}\mathrm{\ \ and\ \ } \max(A)=
\{s\in A: |s|=\ell_A\}.
 \end{equation}
For $\ell\le \ell_A$,
let
\begin{equation}
A\re \ell=\{t\re \ell : t\in A\mathrm{\ and\ }|t|\ge \ell\}
\end{equation}
 and let
\begin{equation}
A\rl \ell=\{t\in A:|t|< \ell\}\cup A\re \ell.
\end{equation}
Thus, $A\re \ell$ is a level set, while $A\rl \ell$ is the set of nodes in $A$ with length less than $\ell$ along with the truncation
to $\ell$ of the  nodes in $A$ of length at least
 $\ell$.
Notice that
$A\re \ell=\emptyset$ for $\ell>\ell_A$, and
 $A\rl \ell=A$  for  $\ell\ge \ell_A$.
 Given $A,B\sse T$, we say that $B$ is an {\em initial segment} of $A$ if   $B=A\rl \ell$
 for some $\ell$ equal to
   the length of some node in $A$.
   In this case, we also say that
   $A$ {\em end-extends} (or just {\em extends}) $B$.
If $\ell$ is not the length of any node in $A$, then
  $A\rl \ell$ is not a subset  of $A$, but  is  a subset of $\widehat{A}$, where
  $\widehat{A}$ denotes $\{t\re n:t\in A\mathrm{\ and\ } n\le |t|\}$.

Define $\max(A)^+$ to be the set of nodes
$t$ in $T\re (\ell_A+1)$ such that $t$ extends $s$ for some $s \in \max(A)$.
Given a node $t\in T$ at the level of a coding node in $T$, $t$ has exactly one immediate successor in $\widehat{T}$, which  we recall
from Notation \ref{notn.cong<}
is  denoted as
$t^+$.

\begin{defn}[$+$-Similarity]\label{def.plussim}
Let $T$ be a diagonal coding tree for
the \Fraisse\ limit $\bK$ of
a \Fraisse\
class $\mathcal{K}$, and
suppose $A$ and $B$
are finite subtrees of $T$.
We write  $A\plussim B$ and say that
 $A$ and $B$ are
 {\em $+$-similar} if and only if
  $A\sim B$ and
 one of the following two cases holds:
 \begin{enumerate}
 \item[]
   \begin{enumerate}
\item[\bf Case 1.]
 If $\max(A)$  has a splitting node in $T$,
 then so does $\max(B)$,
  and the similarity map from $A$ to $B$  takes the splitting node in $\max(A)$ to the splitting node in $\max(B)$.
    \end{enumerate}
     \end{enumerate}
       \begin{enumerate}
    \item[]
        \begin{enumerate}
  \item[\bf Case 2.]
If $\max(A)$ has a coding node, say
 $c^A_n$,
and  $f:A\ra B$ is  the similarity map,
then
  $s^+(n;A)\sim f(s)^+(n;B)$ for each $s\in \max(A)$.
 \end{enumerate}
    \end{enumerate}

Note that $\plussim$ is an  equivalence relation, and  $A\plussim B$ implies $A\sim B$.
When $A\sim B$ ($A\plussim B$), we say that they have the same {\em similarity type} ({\em $+$-similarity type}).
\end{defn}

\begin{rem}\label{rem.SplussimT}
For infinite trees $S$ and $T$ with no terminal  nodes, $S\sim T$ implies that for
each $n$, letting $d^S_n$ and $d^T_n$ denote the $n$-th critical nodes of $S$ and $T$, respectively,
$S\re|d_n^S|\plussim T\re|d_n^T|$.
\end{rem}

Let $T$ be a diagonal coding tree for the \Fraisse\ limit $\bK$ of some \Fraisse\
class $\mathcal{K}$.
We adopt the following notation  from topological Ramsey space theory (see \cite{TodorcevicBK10}).
Given  $k<\om$,
we define
$r_k(T)$ to be  the
restriction of $T$ to the levels of the first $k$ critical nodes of $T$;
that is,
\begin{equation}
r_k(T)=\bigcup_{m<k}T(m),
\end{equation}
where $T(m)$ denotes the set of all nodes in $T$ with length equal to $|d^T_m|$.
It follows from Remark \ref{rem.SplussimT} that  for any
 $S,T\in \mathcal{T}$,
$r_k(S)\plussim r_k(T)$.
Define $\mathcal{AT}_k$ to be the set of {\em $k$-th approximations}
to members of $\mathcal{T}$;
that is,
\begin{equation}
\mathcal{AT}_k=\{r_k(T):T\in\mathcal{T}\}.
\end{equation}
For $D\in\mathcal{AT}_k$
and $T\in\mathcal{T}$,
define the set
\begin{equation}
[D,T]=\{S\in \mathcal{T}:r_k(S)=D\mathrm{\ and\ } S\le T\}.
\end{equation}
Lastly, given
 $T\in\mathcal{T}$,
 $D=r_k(T)$,  and $n>k$,
  define
\begin{equation}
r_n[D,T]=\{r_n(S):S\in [D,T]\}.
\end{equation}

\begin{defn}[Extension Property]\label{defn.ExtProp}
We say that
$\bK$
has the {\em Extension Property} when the following holds:
\begin{enumerate}
\item[(EP)]
Suppose $A$ is a finite or infinite subtree of   some
$T\in\mathcal{T}$.
Let $k$ be given  and suppose
$\max(r_{k+1}(A))$ has a splitting node.
Suppose that $B$ is a $+$-similarity copy of $r_k(A)$ in $T$.
Let  $u$ denote the splitting node in $\max(r_{k+1}(A))$,
and let
 $s$ denote  the node in $\max(B)^+$ which must be extended to a splitting node in order to obtain a $+$-similarity copy of $r_{k+1}(A)$.
If $s^*$ is a splitting node  in $T$  extending $s$,
then  there are extensions of the rest of the nodes in $\max(B)^+$ to the same length as $s^*$ resulting in a $+$-similarity copy
of $r_{k+1}(A)$ which
can be extended to a copy of $A$.
\end{enumerate}
\end{defn}

\begin{defn}[\EEAP$^+$]\label{defn_EEAP_newplus}
A \Fraisse\
structure $\bK$
has the
{\em  \EEAPnonacronym$^+$ (\EEAP$^+$)} if
its age
$\mathcal{K}$ satisfies \EEAP,
and $\bK$ has
the Diagonal Coding Tree Property and the Extension Property.
\end{defn}

The  coding tree version of \EEAP$^+$ follows from \EEAP$^+$ and is the version  used in proofs.

\begin{defn}[\EEAP$^+$, Coding Tree Version]\label{def.EEAPCodingTree}
A \Fraisse\ class $\mathcal{K}$ satisfies the
 {\em
Coding Tree Version of  \EEAP$^+$}
 if and only if   $\mathcal{K}$ satisfies
 the disjoint amalgamation property
 and,
 letting $\bK$ be any enumerated \Fraisse\ limit of $\mathcal{K}$, $\bK$ satisfies
 the Diagonal Coding Tree Property,
the Extension Property, and
the following
condition:

Let $T$ be any  diagonal coding  subtree   of $\bU(\bK)$  (or of $\bS(\bK)$), and let $\ell<\om$ be given.
Let
$i,j$ be any distinct integers such that
$\ell<\min(|c^T_i|,|c^T_j|)$,
and let
$\bfC$ denote  the substructure of $\bK$  represented by the coding nodes in $T\rl \ell$ along with
 $\{c^T_i,c^T_j\}$.
Then there are $m \ge\ell$
and
$s',t'\in T\re m$ such that
$s'\contains s$ and $t'\contains t$
 and,
assuming (1) and (2), the conclusion holds:
 \begin{enumerate}
\item[(1)]
Suppose $n\ge m$ and
$s'',t''\in T\re n$
 with $s''\contains s'$ and
$t''\contains t'$.

\item[(2)]
Suppose $c^T_{i'}\in T$ is any coding node extending $s''$.
\end{enumerate}
Then
there is a coding node $c^T_{j'}\in T$, with $j'>i'$,
such that
$c_{j'}\contains t''$ and
the substructure  of $\bK$  represented by the coding nodes in $T\rl \ell$ along with $\{c^T_{i'},c^T_{j'}\}$ is isomorphic to
$\bfC$.
\end{defn}

We now recall  the Labeled Substructure Disjoint Amalgamation Property$^+$ from Subsection  4.4 of Part I.

\begin{defn}[Labeled Diagonal Coding Tree]\label{defn.LDCT}
A diagonal coding tree $T$ is {\em labeled} if the following hold:
There
is some $2\le q<\om$,
and a function
 $\psi$ defined on the set of splitting nodes in $\bT$ and  having range $q$,
 such that  the following holds:
 \begin{enumerate}
 \item[(a)]
 If $s\sse t$ are splitting nodes in $T$, then $\psi(s)\ge \psi(t)$.
 \item[(b)]
 For each splitting node $s\in T$ and each $n>|s|$, there is a splitting node $t\contains s$ with $|t|\ge n$ such that $\psi(t)=\psi(s)$.
\item[(c)]
The language for $\bK$ has at least one binary relation symbol (besides equality), and the
value  of $\psi$ is determined  by some partition of all  pairs   of
 partial $1$-types  involving only binary
 relation symbols
 over a one-element structure into pieces
 $Q_0,\dots,Q_{q-1}$,
such that
whenever $s$ is a splitting node in $\bT$,
$\psi(s)=m$ if and only if the following hold:
whenever  $c^{\bT}_j,c^{\bT}_k$ are coding nodes in $\bT$ with $c^{\bT}_j\wedge c^{\bT}_k=s$,
then
the pair of partial $1$-types of $v^{\bT}_j$ and  $v^{\bT}_k$ over
$\bK\re\{v_|s|\}$  is in $Q_m$.
\item[(d)]
The maximal splitting node $s$ below a coding node in $\bT$ has $\psi(s)=0$.
\end{enumerate}
\end{defn}

Given (a) and (b), the function $\psi$ can be extended to all nodes of $T$ as follows:
For each non-splitting node $t\in T$, define $\psi(t)$ to equal $\psi(s)$, where $s$ is the maximal splitting node in $T$ such that $s\sse t$.

\begin{notation}
For a labeled diagonal coding tree $\bT$, for $S,T$ subtrees of $\bT$,
write $S\Lsim T$ to mean that $S\sim T$ and the similarity map $f:S\ra T$ preserves $\psi$,
meaning that for each $s\in S$,
$\psi(s)=\psi(f(s))$.
\end{notation}

\begin{defn}[${L+}$-Similarity]\label{def.Lplussim}
Let $T$ be a labeled diagonal coding tree with labeling function $\psi$ for
the \Fraisse\ limit $\bK$ of
a \Fraisse\
class $\mathcal{K}$, and
suppose $A$ and $B$
are finite subtrees of $T$.
We write  $A\Lplussim B$ and say that
 $A$ and $B$ are
 {\em $\mathrm{L}+$-similar} if and only if
  $A\plussim B$ and 
  $A\Lsim B$.
\end{defn}

\begin{defn}[Labeled Extension Property]\label{LEP}
We say that $\bK$ has the {\em Labeled Extension Property} when the following condition (LEP) holds:
\begin{enumerate}
\item[(LEP)]
There
is some $2\le q<\om$
and a labeling function
 $\psi$ taking $\bT$ onto $q$ satisfying Definition \ref{defn.LDCT}
 such that  the following holds:
Suppose $A$ is a finite or infinite subtree of    some $T\in\mathcal{T}$.
Let $k$ be given  and suppose
$\max(r_{k+1}(A))$ has a splitting node.
Suppose that $B$ is an $\mathrm{L}+$-similarity copy of $r_k(A)$ in $T$.
Let  $u$ denote the splitting node in $\max(r_{k+1}(A))$,
and let
 $s$ denote  the node in $\max(B)^+$ which must be extended to a splitting node in order to obtain a $+$-similarity copy of $r_{k+1}(A)$, and note that $\psi(s)\ge \psi(u)$.
 Then  for each  $s'\contains s$ in $T$ with $\psi(s')\ge \psi(u)$,
 there exists  a splitting node $s^*\in T$ extending $s'$ such that $\psi(s^*)=\psi(u)$.
 Moreover,
 given such an $s^*$,
 there are extensions of the rest of the nodes in $\max(B)^+$ to the same length as $s^*$ resulting in an $\mathrm{L}+$-similarity copy of $r_{k+1}(A)$.
\end{enumerate}
\end{defn}

\begin{defn}[LSDAP$^+$]\label{defn_LSDAP_plus}
A \Fraisse\
structure $\bK$
has the
{\em  Labeled Substructure Disjoint Amalgamation Property$^+$ (LSDAP$^+$)} if
its age
$\mathcal{K}$ satisfies SDAP,
and $\bK$ has  a labeled diagonal coding tree satisfying 
the Diagonal Coding Tree Property and the Labeled Extension Property.
\end{defn}

\begin{defn}[The space of diagonal coding trees for LSDAP$^+$ structures]\label{spacectLSDAP}
If $\bK$ satisfies LSDAP$^+$, then given a diagonal coding tree $\bT$ for $\bK$ with labeling $\psi$,
we let $\mathcal{T}$ denote the set of all subtrees $T$ of $\bT$ such that $T\Lsim \bT$.
\end{defn}

%%%%%%%%%%%%%%%%%%%%%%%%%%%%%%%%%%%%%%%%%%%%%%
%%%%%%%%%%%%%%%%%%%%%%%%%%%%%%%%%%%%%%%%%%%%%%
%%%%%%%%%%%%%%%%%%%%%%%%%%%%%%%%%%%%%%%%%%%%%%%%%%

\section{Exact upper bounds for big Ramsey degrees}\label{sec.FRT}

This section contains the Ramsey theorem
for colorings of copies of a given
finite substructure of a \Fraisse\ structure  satisfying
\EEAP$^+$ or LSDAP$^+$.
Theorem
\ref{thm.onecolorpertype}
provides
upper bounds for the big Ramsey degrees
of such structures
when the language has relation symbols of arity at most two, and these
turn out to be exact.
The proof of exactness  will be given   in Section \ref{sec.brd}.

The proof of Theorem
\ref{thm.onecolorpertype} proceeds by induction arguments starting with the Level Set Ramsey Theorem (Theorem 5.4 from Part I).
We now  recall  notation and  definitions from Part I.

%It suffices to hold to 
%the following convention, which  appears as  Convention 4.12 in \cite{CDPI}:
Recall the following convention, which appears as Convention 4.12 in \cite{CDPI}.

%I changed this so that it used the same number as before--Becky
\begin{conv}\ref{conv.Gamma_ts}: %\label{conv.Gamma_ts}
Let $\mathcal{K}$ be a \Fraisse\ class in a language $\mathcal{L}$ and $\bK$ a \Fraisse\ limit of $\mathcal{K}$.
Either there is a diagonal coding subtree of $\bS(\bK)$,
or else there is a diagonal coding subtree of $\bU(\bK)$ in which all unary relations occur densely.
%If
%(a)
%$\mathcal{K}$  satisfies \SFAP,
%or
%(b)
%$\bK$
% satisfies \EEAP$^+$ and either has no unary relations or there is a subtree , then we  work inside a diagonal coding subtree $\bT$ of $\bS$. Otherwise, we work inside a diagonal coding subtree $\bT$ of $\bU$.
 \end{conv}
 
 Note that for any \Fraisse\ structure $\bK$, $\bU(\bK)$ contains a similarity copy of $\bS(\bK)$.  Thus, we will simply write $\bU(\bK)$ from now on.

%\begin{convention}\label{conv.Gamma_ts}
%Let $\mathcal{K}$ be a \Fraisse\ class in a language $\mathcal{L}$ and $\bK$ a \Fraisse\ limit of $\mathcal{K}$.
%If
%(a)
%$\mathcal{K}$  satisfies \SFAP,
%or
%(b)
%$\bK$
% satisfies \EEAP$^+$ and either has no unary relations or has no transitive relations, then we  work inside a diagonal coding subtree $\bT$ of $\bS$. Otherwise, we work inside a diagonal coding subtree $\bT$ of $\bU$.
% \end{convention}

Given any $A\sse T$, we will abuse notation and use $r_k(A)$ to denote the first $k$ levels of the tree induced by the meet-closure of $A$.
By an {\em antichain} of coding nodes, we mean a set of  coding nodes  which is pairwise  incomparable with respect to  the tree partial order of inclusion.
\vskip.1in

\noindent\bf{Set-up for the Level Set Ramsey Theorem.} \rm
Let $T$ be a diagonal coding tree  in $\mathcal{T}$.
Fix a finite antichain  of coding nodes  $\tilde{C}\sse T$.
We abuse notation and also write  $\tilde{C}$ to denote  the tree that its meet-closure induces in $T$.
Let $\tilde{A}$ be a fixed proper initial segment of
$\tilde{C}$, allowing for $\tilde{A}$ to be the empty set.
Thus, $\tilde{A}= \tilde{C}\rl \ell$, where $\ell$
is the length of some splitting or  coding node in
$\tilde{C}$
(let $\ell=0$ if $\tilde{A}$ is empty).
Let $\ell_{\tilde{A}}$ denote this $\ell$, and note that
any non-empty
 $\max(\tilde{A})$ either  has a coding node or a splitting node.
Let $\tilde{x}$ denote the shortest splitting or coding node in $\tilde{C}$ with length greater than $\ell_{\tilde{A}}$,
and define  $\tilde{X}=\tilde{C}\re |\tilde{x}|$.
Then $\tilde{A}\cup\tilde{X}$ is an initial segment of
 $\tilde{C}$; let $\ell_{\tilde{X}}$ denote $|\tilde{x}|$.
There are two cases:

\begin{enumerate}
\item[]
\begin{enumerate}
\item[\bf Case (a).]
$\tilde{X}$ has a splitting node.
\end{enumerate}
\end{enumerate}

\begin{enumerate}
\item[]
\begin{enumerate}
\item[\bf Case (b).]
$\tilde{X}$ has a coding node.
\end{enumerate}
\end{enumerate}

Let $d+1$ be the number of nodes in $\tilde{X}$ and  index these nodes as $\tilde{x}_i$, $i\le d$,
where $\tilde{x}_d$ denotes  the critical  node (recall that {\em critical node}  refers to  a splitting or  coding node).
Let
\begin{equation}\label{eq.tildeB}
\tilde{B}=\tilde{C} \re (\ell_{\tilde{A}}+1).
\end{equation}
Then
  $\tilde{X}$ is a level set   equal to or  end-extending  the level set $\tilde{B}$.
  For each $i\le d$, define
  \begin{equation}
  \tilde{b}_i=\tilde{x}_i\re \ell_{\tilde{B}}.
\end{equation}
Note that we
consider nodes in  $\tilde{B}$  as simply nodes to be extended; it
does not matter
whether the nodes in $\tilde{B}$ are coding, splitting, or neither in $T$.

\begin{defn}[Weak similarity]\label{defn.weaksim}
Given finite subtrees $S,T\in\mathcal{T}$ in which each coding node is terminal,
we say that  $S$ is {\em weakly similar} to $T$, and write $S\wsim T$,
if and only if
$S\setminus \max(S)\plussim T\setminus\max(T)$.
We say that  $S$ is {\em L-weakly similar} to $T$, and write $S\Lwsim T$,
if and only if
$S\setminus \max(S)\Lplussim T\setminus\max(T)$.
\end{defn}

In what  follows, we put the technicalities for the LSDAP$^+$ case in parentheses.

\begin{defn}[$\Ext_T(B;\tilde{X})$]\label{defn.ExtBX}
Let $T\in\mathcal{T}$ be  fixed  and let $D=r_n(T)$ for some $n<\om$.
Suppose $A$ is a subtree of $D$ such that
$A \plussim \tilde{A}$ ($A \Lplussim \tilde{A}$)
and
 $A$ is extendible to a similarity ($L$-similarity) copy of $\tilde{C}$ in $T$.
Let $B$ be a subset of  the level set $\max(D)^+$ such that $B$ end-extends  or equals
$\max(A)^+$
and $A\cup B\wsim\tilde{A}\cup\tilde{B}$
($A\cup B\Lwsim\tilde{A}\cup\tilde{B}$).
Let $X^*$ be a level set end-extending $B$ such that $A\cup X^*\plussim \tilde{A}\cup\tilde{X}$ 
($A\cup X^*\Lplussim \tilde{A}\cup\tilde{X}$).  
Let $U^*=T\rl (\ell_B-1)$.
Define
$\Ext_T(B;X^*)$ to be the collection of all level sets $X\sse T$
such that
\begin{enumerate}
\item
$X$ end-extends $B$;
\item
$U^*\cup X\plussim U^*\cup X^*$ ($U^*\cup X\Lplussim U^*\cup X^*$);
\item
$A\cup X$ extends to a copy of $\tilde{C}$.
\end{enumerate}
\end{defn}

For Case (b), condition (3) follows from (2).
For Case (a),
the Extension Property (Labeled Extension Property)
guarantees that
for any level set  $Y$ end-extending $B$,
there is a level set $X$ end-extending $Y$ such that $A\cup X$ satisfies
condition (3).
In both cases, condition (2) implies that $A\cup X\plussim \tilde{A}\cup\tilde{X}$ ($A\cup X\Lplussim \tilde{A}\cup\tilde{X}$).

\begin{thm}[Level Set Ramsey Theorem]\label{thm.matrixHL}
Suppose that
$\mathcal{K}$
has \Fraisse\ limit $\bK$ satisfying
\EEAP$^+$ (or LSDAP$^+$), and
$T\in\mathcal{T}$ is given.
Let $\tilde{C}$ be a finite antichain of coding nodes in $T$,   $\tilde{A}$ be an initial segment of $\tilde{C}$,
and $\tilde{B}$ and $\tilde{X}$ be defined as above.
Suppose
 $D=r_n(T)$ for some $n<\om$,
 and
 $A\sse D$
and $B\sse \max(D^+)$
satisfy
$A\cup B\wsim \tilde{A}\cup\tilde{B}$
 ($A\cup B\Lwsim\tilde{A}\cup\tilde{B}$).
Let  $X^*$ be a level set end-extending $B$ such that $A\cup X^*\plussim \tilde{A}\cup\tilde{X}$ ($A\cup X^*\Lplussim \tilde{A}\cup\tilde{X}$).
Then given any coloring
  $h:  \Ext_T(B;X^*)\ra 2$,
  there is a coding tree $S\in [D,T]$ such that
$h$ is monochromatic on $\Ext_S(B;X^*)$.
\end{thm}

\begin{rem}\label{rem.plussimfornoncn}
It follows from the  proof in Part I  of  Theorem \ref{thm.matrixHL} that in Case (b), the coding nodes in any member $X\in\Ext_S(B;X^*)$ extend the coding node $t^*_d$.
It then follows from
(3) in Definition \ref{defn.sct}
that for every level set $X\sse S$ with $A\cup X\sim\tilde{A}\cup X^*$, the coding node $c$ in $X$  automatically satisfies
$c^+(c;A)\sim  (t^*_d)^+(t^*_d;A)
\sim \tilde{x}_d^+(\tilde{x}_d;\tilde{A})$, where $\tilde{x}_d$ denotes the  coding node in $X^*$.
Thus, $A\cup X\plussim \tilde{A}\cup X^*$ 
($A\cup X\Lplussim \tilde{A}\cup X^*$) if and only if the
non-coding nodes in $X$ have immediate successors with similar passing types over $A\cup\{c\}$ as their counterparts in $X^*$ have over $\tilde{A}\cup\{\tilde{x}_d\}$.

Moreover, for languages with only unary and binary relations, in Case (b)  the set
$\Ext_T(B;X^*)$ is exactly the set of all end-extensions $X$ of $B$ such that $A\cup X\plussim \tilde{A}\cup\tilde{X}$ ($A\cup X\Lplussim \tilde{A}\cup\tilde{X}$).
These observations
will be useful in the proof of  next theorem.
\end{rem}

Recall that two antichains of coding nodes are considered similar (L-similar) if the trees induced by their meet-closures are similar (L-similar).

\begin{thm}\label{thm.onecolorpertreetype}
Suppose that $\mathcal{K}$
is a \Fraisse\ class in a language with
relation symbols  of arity at most two,
and suppose that $\mathcal{K}$
has a \Fraisse\ limit satisfying
\EEAP$^+$ (LSDAP$^+$).
Let
$\bT$ be  a diagonal coding subtree of $\bU(\bK)$,
let
$\tilde{C}\sse\bT$ be an antichain of coding nodes,
 and let $T\in\mathcal{T}$ be fixed.
Given any coloring of the set $\{C\sse T:C\sim \tilde{C}\}$
($\{C\sse T:C\Lsim \tilde{C}\}$),
there is an $S\le T$ such that all  members  of
$\{C\sse S:C\sim \tilde{C}\}$ 
($\{C\sse S:C\Lsim \tilde{C}\}$) have the same color.
\end{thm}

\begin{proof}
We write the proof  for \Fraisse\ limits satisfying SDAP$^+$, noting that for LSDAP$^+$, one just replaces the uses of similarity and $+$-similarity with L-similarity and L$+$-similarity, respectively. 
The proof is by reverse induction on the number of levels in 
a finite tree $\tilde{C}$ in which all coding nodes  are maximal nodes.

Suppose that $\tilde{C}$ has $n\ge 1$ levels. 
Let
$\tilde{X}$ denote $\tilde{C}\re \ell_{\tilde{C}}$, the maximum level of $\tilde{C}$.
Let $\tilde{A}$ denote $\tilde{C}\setminus\tilde{X}$; that is, $\tilde{A}$ is the initial segment of  all but the maximum level of $\tilde{C}$.
Let $m_0$ be the least integer  such that  $r_{m_0}(T)$ contains a $+$-similarity copy of $\tilde{A}$ extending to a copy of $\tilde{C}$,
 and
 let $D_0=r_{m_0}(T)$.
Let  $A^0,\dots,A^j$ list
  those   $A\sse T$ such that
  $\max(A)\sse\max(D_0)$ and
 $A$ extends to a similarity copy of $\tilde{C}$.
 For $i\le j$, let $B^i$ denote $(A^i)^+$,
 which we recall is the tree consisting of the nodes in $A^i$ along with all immediate successors of nodes in $\max(A^i)$.
 (These immediate successors are the same whether we consider them in $\bT$ or in $T$.)
 Each $B^j$ is a subtree of $(D_0)^+$.
  Apply the Level Set Ramsey Theorem
   to obtain a $T_{0}^0\in [D_0,T]$ such that $h$ is monochromatic on $\Ext_{T_0^0}(B^0;\tilde{X})$.
Repeat this process,  each time thinning the previous tree
to obtain $T_0^{i+1}\in [D_0,T_0^i]$  so that
for each $i\le j$,  $\Ext_{T_0^j}(B^i;\tilde{X})$ is monochromatic.
Let $T_0$ denote $T_0^j$.
Then for each $A^i$, $i\le j$,
every extension of $A^i$ to a similarity copy of $\tilde{A}\cup\tilde{X}$ inside $T_0$ has the same color.

Given $k<\om$ and $T_k$, let $m_{k+1}$ be the least integer greater than $m_k$
 such that $r_{m_{k+1}}(T_k)$ contains a $+$-similarity copy of $\tilde{A}$  extending to a copy of $\tilde{C}$.
 Let $D_{k+1}=r_{m_{k+1}}(T_k)$, and
 index those $A$  with $\max(A)\sse \max(D_{k+1})$
such that
 $A$ extends to a similarity copy of $\tilde{C}$
  as $A^i$, $i\le j$ for some $j$.
 Repeat the above process applying  the
 Level Set Ramsey Theorem
  finitely many times to obtain
 a $T_{k+1}\in
[D_{k+1},T_k]$  with the property that
for each  $i\le j$,
all similarity copies of  $\tilde{A}\cup\tilde{X}$ in $T_{k+1}$
 extending $A^i$
have the same color.

Since each $T_{k+1}$ is a member of $ [D_{k+1}, T_k]$,  the union $\bigcup_{k<\om} D_k$ is a member of $\mathcal{T}$, call it $S_1$.
This induces a well-defined coloring of the copies of $\tilde{A}$ in $S_1$ as follows:
Given
$A\sse S_1$ a similarity copy of $\tilde{A}$ extending to a copy of $\tilde{C}$,
let $k$ be least such that $A$ is contained in $r_{m_k}(S_1)$.
Then $\max(A)$ is contained in $\max(D_k)$,
and   $S_1\in [D_k,T_k]$ implies that
for each level set extension $X$ of $A$ in $S_1$ such that
$A\cup X\sim \tilde{A}\cup\tilde{X}$,
these similarity copies of $\tilde{C}$ have the same color.

This now induces a coloring on $+$-similarity copies of
$\tilde{A}$ inside $S_1$.
Let $\tilde{C}_{n-1}$ denote this $\tilde{A}$,
$\tilde{X}_{n-1}$ denote $\max(\tilde{C}_{n-1})$,
and
 $\tilde{A}_{n-1}$ denote $\tilde{C}_{n-1}\setminus
 \tilde{X}_{n-1}$.
Repeat the  argument in the previous
three paragraphs
to obtain $S_2\le S_1$ such that
for each $+$-similarity copy of $\tilde{A}_{n-1}$ in $S_2$,
all extensions to $+$-similarity copies of $\tilde{C}_{n-1}$ in $S_2$ have the same color.

 At the end of the reverse induction, we obtain an $S:=S_n\le T$
such that all similarity copies of $\tilde{C}$ in $S$
 have the same color.
\end{proof}

%\begin{rem}\label{rem.Gamma_ts}
%Recall Convention \ref{conv.Gamma_ts} that if
%(a)
%$\mathcal{K}$  satisfies \SFAP,
%or
%(b)
%$
%\bK$ satisfies \EEAP$^+$ and either has no  unary relations or has no transitive relations, then we work inside a diagonal coding subtree  $\bT$ of $\bS$. Otherwise, we have a transitive relation as well as unary relations and we work inside $\bU$. In this case, any subset $D\sse\bU$ for which $\bK\re D$ contains a copy of $\bK$ will contain a subset $D'\sse D$ such that $\bK\re D'\cong \bK$ and for each non-terminal node $s\in D'$ and each $\gamma\in\Gamma$, there is a coding node $c\in D'$ extending $s$  such that $\gamma(v)$ holds in $\bK$, where $v$ is the vertex of $\bK$ represented by $c$.
%\end{rem}

The next lemma shows that
if $\mathcal{K}$
has \Fraisse\ limit $\bK$ satisfying
\EEAP$^+$, then
within any diagonal coding tree, there is an antichain of coding nodes representing a copy of $\bK$.
(Recall Convention \ref{conv.Gamma_ts}.)

\begin{lem}\label{lem.bD}
Suppose a \Fraisse\ class $\mathcal{K}$
has \Fraisse\ limit $\bK$ satisfying
\EEAP$^+$ or LSDAP$^+$.
If $\mathcal{K}$ satisfies \SFAP\, or
 $\bK$ either has no transitive relation or  has no  unary relations, let $T$ be
 a diagonal coding subtree of $\bS(\bK)$; otherwise, let $T$ be a  diagonal
 coding subtree of $\bU(\bK)$.
Then  there is an infinite  antichain of coding nodes  $\bD\sse T$  so that $\bK\re \bD\cong^{\om}\bK$.
\end{lem}

\begin{proof}
We will use
$c^{\bD}_n$ to denote the $n$-th coding node in $\bD$,
and $v^{\bD}_n$ to denote the vertex in $\bK$ coded by
$c^{\bD}_n$.
The antichain $\bD$ will look almost exactly like $T$ in the following sense:
For each $n$,
the level set of $\bD$ containing the $n$-th coding node,
denoted $\bD\re |c^{\bD}_n|$,
will have exactly one more node than $T\re |c^T_n|$, and the $\prec$-preserving bijection between $T\re |c^T_n|$ and
$(\bD\re |c^{\bD}_n|)\setminus\{c^{\bD}_n\}$ will preserve passing types of the immediate successors.
(This is not necessary to the results on big Ramsey degrees, but since we can do this, we will.)
Moreover, letting $T'$ be the coding tree
obtained by
deleting the coding nodes in $\bD$ and declaring
 the node  $t$ in   $(\bD\re |c^{\bD}_n|)\setminus\{c^{\bD}_n\}$ which has  $t\wedge c^{\bD}_n$ of maximal length  to be the $n$-th coding node in $T'$,
 then $T'\sim T$.

Let  $m_n$ denote the integer such that the $n$-th coding node in $T$ is in the $m_n$-th level of $T$;
that is, $c^T_n$ is in the maximal level of $r_{m_n+1}(T)$.
To construct $\bD$, begin by
taking  the first $m_0$ levels of $\bD$ to equal those of $T$; that is, let $r_{m_0}(\bD)= r_{m_0}(T)$.
Each of these  levels contains a splitting node.
Let $X$ denote the set of immediate successors in
$\widehat{T}$  of
the maximal nodes in
$r_{m_0}(\bD)$.
By \EEAP\ (LSDAP$^+$),
whatever we choose to be $c^{\bD}_0$,
each node in $X$ can extend to a node in $T$ with  the desired passing type at  $c^{\bD}_0$.

Let $s$ denote the node in $X$ which extends to $c^T_0$.
It only remains to  find a splitting node  extending $s$
whose immediate successors
can be   extended to a coding node
$c^{\bD}_0$ (which will be terminal in $\bD$)
 and another node $z$ of length $|c^{\bD}_0|+1$
satisfying
$z(c^{\bD}_0)
\sim
c_0^+(c_0)$
($z(c^{\bD}_0)
\Lsim
c_0^+(c_0)$), so that
$z\re (\bK\re \{v^{\bD}_0\})$ is the same as the type of $c^+_0$ over $\bK\re \{v_0\}$.

To do this, we  utilize  \EEAP\ (LSDAP$^+$):
  $\bfA$ is the empty structure and
$\bfC$ is the structure
$\bK\re \{v^T_0,v^T_i\}$ for any $i>0$ such that
 $c^T_i$ extends $c^T_0$.
Extend $s$ to some splitting node $s'\in T$
long enough so that
the structure $\bK\re (T\re|s'|)$
acts as $\bfA'$ as in the set-up of (B) in \EEAP\ (and $\psi(s)=0$ in the case of LSDAP$^+$).
In (B1), we take $\bfC'$ to be a copy of $\bfC$ represented by some coding nodes $c^T_j,c^T_k$, where
 $s'\sse c^T_j\sse c^T_k$.
 In (B2), we let $\bfB=\bfA'$,
 and take $\sigma=\tau=s'$.
Let $t_0,t_1$ denote the immediate successors of $s'$ in $\widehat{T}$,
and take a coding node in $T$, which we denote $c^{\bD}_0$,  extending $t_0$.
(The vertex $v^{\bD}_0$ which
$c^{\bD}_0$
 represents is the $v''$ in (B3).)
Then by  \EEAP\ (LSDAP$^+$),
there is a coding node $c^T_m$ extending $t_1$ such that
$c^T_m(c^{\bD}_0)
\sim
c_0^+(c_0)$
($c^T_m(c^{\bD}_0)
\Lsim
c_0^+(c_0)$).
We let $y=c^T_m\re|c^{\bD}_0|$ and $z=c^T_m\re(|c^{\bD}_0|+1)$.
The passing type of $z$ at $c^{\bD}_0$ is the desired passing type.
We let $\bD\re(|c^{\bD}_0|+1)$ consist of the node
$z$ along with  extensions of the nodes in
$X\setminus \{s\}$ to the length
of $z$ so that their passing types at $c^{\bD}_0$
are as desired;
that is, the $\prec$-preserving bijection between
$T\re (|c^T_0|+1)$ and
$\bD\re (|c^{\bD}_0|+1)$
preserves passing types at $c^T_0$ and $c^{\bD}_0$, respectively.
We let
$\bD\re |c^{\bD}_0|$ equal $\{c^{\bD}_0\}\cup \bD\re |c^{\bD}_0|$.

For the general construction stage, given $\bD$ up to the level of $|c^{\bD}_n|+1$,
let
$X$ denote the level set  $\bD\re(|c^{\bD}_n|+1)$.
Extend the nodes in $X$ in the same way that the nodes in $T\re |c^T_{n+1}|$ extend the nodes in
$T\re (|c^T_{n}|+1)$.
Let
$s$ denote the node in $X$ which needs to be extended to the next coding node $c^{\bD}_{n+1}$, and repeat the argument above find a suitable  splitting node and extensions to a coding node $c^{\bD}_{n+1}$ as well as a non-coding node of the same height with the desired passing type at $c^{\bD}_{n+1}$ over $\{c^{\bD}_i:i\le n\}$.
By \EEAP\ (LSADP$^+$), the other nodes in $X$ extend to have the desired passing types.
\end{proof}

By Remark \ref{rem.plussimfornoncn}, given two antichains of coding nodes $C$ and $C'$,
it follows that
$C\sim C'$ ($C\Lsim C'$) if and only if for any $k$, the first $k$ levels of the trees induced by $C$ and $C'$, respectively, are $+$-similar (L$+$-similar).

%Let $\mathcal{K}$ be a \Fraisse\ class with \Fraisse\ limit $\bK$ satisfying \EEAP$^+$.
%If $\mathcal{K}$ satisfies \SFAP, or $\bK$ either has no transitive relation or has no unary relations, then let $\bT$ be a diagonal coding subtree of $\bS(\bK)$. Otherwise, let $\bT$ be a diagonal coding subtree of $\bU(\bK)$.
Recalling that we may identify a subset of $\bT$ with the subtree it induces,
given an antichain of coding nodes $C\sse \bT$,
we  let $\Sim(C)$  denote the set of all antichains $C'$ of coding nodes in $\bT$ such that $C'\sim C$.
Thus, $\Sim(C)$ is a $\sim$-equivalence class, and  we call $\Sim(C)$ a {\em similarity type}.
For $S\sse \bT$, we write $\Sim_S(C)$ for the set of $C'\sse S$ such that $C'\sim C$.
Note in the case of  LSDAP$^+$, $C'\sim C$ if and only if $C'\Lsim C$; for the relations between two vertices  in $\bK$ completely determines the $\psi$ value  of the meet of the coding nodes in $\bT$ representing those two vertices.

\begin{defn}\label{defn.SimTC}
We say that $C$ {\em represents} a copy of a structure  $\bG\in\mathcal{K}$ when
$\bK\re C\cong \bG$.
Given $\bG\in\mathcal{K}$,
let  $\Sim(\bG)$ denote a set consisting of
one representative  from each
 similarity type
 $\Sim(C)$ of diagonal  antichains of coding nodes $C\sse\bT$ representing
a copy of
$\bG$.
\end{defn}

The next theorem  providing upper bounds follows immediately from  Theorem \ref{thm.onecolorpertreetype}
and  Lemma \ref{lem.bD}.

\begin{thm}[Upper Bounds]\label{thm.onecolorpertype}
Suppose $\mathcal{K}$
is a \Fraisse\ class
with relations of arity at most two and
 with Fraisse\ limit $\bK$ satisfying
\EEAP$^+$ or LSDAP$^+$.
Then for each  $\bG\in \mathcal{K}$,
the big Ramsey degree of $\bG$ in $\bK$ is bounded by the  number of similarity types of diagonal antichains of coding nodes representing $\bG$; that is,
 $$T(\bG,\bK)\le |\Sim(\bG)|.
$$
Moreover, given any finite collection  $\mathcal{G}$
 of structures in $\mathcal{K}$ and any
 coloring of all copies of   each $\bG\in\mathcal{G}$  in $\bK$ into finitely many colors,
  there is a substructure $\bJ$ of $\bK$
  such that
  $\bJ\cong^{\om}\bK$ and
  each
  $\bG\in\mathcal{G}$ takes at most  $|\Sim(\bG)|$ many colors in $\bJ$.
\end{thm}

\begin{proof}
Let $\mathcal{G}$ be a finite collection of structures in $\mathcal{K}$.
 Given any $T\in\mathcal{T}$,
 apply
Theorem  \ref{thm.onecolorpertreetype}
finitely many times to obtain a coding subtree $S\le T$ such that the coloring  takes one color on
the set
$\Sim_S(C)$,
 for each $C\in \bigcup\{\Sim(\bG):\bG\in\mathcal{G}\}$.
Then apply  Lemma \ref{lem.bD} to take an antichain of coding nodes, $\bD\sse S$, such that $\bK\re \bD\cong^{\om}\bK$.
Letting $\bJ=\bK\re\bD$, we see that
there are at most $|\Sim(\bG)|$ many colors  on the copies of $\bG$ in $\bJ$.
\end{proof}

In the next section, we will show that these bounds are exact.

\section{Simply characterized  big Ramsey degrees and  structures}\label{sec.brd}

In this section we prove
that
if a  \Fraisse\ limit $\bK$ of a \Fraisse\ class $\mathcal{K}$
with relations of arity at most two
satisfies \EEAP$^+$ or LSDAP$^+$, then we can characterize the exact big Ramsey degrees of $\bK$; furthermore, $\bK$ admits a big Ramsey structure.
We first show, in Theorem \ref{thm.persistence}, that each of the similarity types in
Theorem \ref{thm.onecolorpertype} persists, and hence these similarity types form canonical partitions.  From this, we obtain
a succinct characterization of the exact big Ramsey degrees of $\bK$. We then prove, in Theorem \ref{thm.apply}, that
canonical partitions
characterized via similarity types
satisfy a condition of Zucker
(\cite{Zucker19})
guaranteeing
the existence of big Ramsey structures.
This involves showing how
Zucker's condition, which is
phrased in terms of colorings of embeddings of a given structure, can be
met
by
canonical partitions
that are
in terms of
colorings of copies of a given structure.
The big Ramsey structure for $\bK$ thus obtained also has a simple characterization.
From these
results,
 we deduce Theorem \ref{thm.main}.

\begin{rem}
We point out that Theorem \ref{thm.persistence} also provides lower bounds for  the big Ramsey degrees in a
 \Fraisse\ limit $\bK$ of a \Fraisse\ class $\mathcal{K}$
with relations of any arity.
\end{rem}

Recall from Definition \ref{defn.cp} the notion of {\em persistence}.
We first show, in Theorem \ref{thm.persistence}, that
given $\bG\in\mathcal{K}$,
each of the similarity types in
 $\Sim(\bG)$ persists in any subcopy of $\bK$.
From this, it will follow
  that
  the big Ramsey degree
 $T(\bG,\mathcal{K})$ is exactly the cardinality of
  $\Sim(\bG)$ (Theorem \ref{thm.bounds}).
The proof of  Theorem  \ref{thm.persistence} follows the outline and many ideas of the proof of
Theorem 4.1 in
 \cite{Laflamme/Sauer/Vuksanovic06},
where Laflamme, Sauer, and Vuksanovic
proved persistence of diagonal antichains for
unrestricted
binary relational structures.

Recall that $\Gamma$ denotes  the set of all complete
$1$-types
of elements of $\bK$
over the empty
set.
For $\gamma\in\Gamma$,
we  let $\bC_\gamma$ denote the set of coding nodes $c_n$ in $\bS$
such that
$\gamma(v_n)$ holds in $\bK$, where $v_n$ is the vertex of $\bK$ represented by $c_n$;
 let $\gamma_{c_n}$ denote this  $\gamma$.
The next definition extends the notion of ``passing number preserving map'' from
Theorem 4.1 in \cite{Laflamme/Sauer/Vuksanovic06}.

\begin{defn}\label{defn.ptp}
Given two subsets $S,T\sse \bS$ with
coding nodes
$\lgl c^S_n:n<M\rgl$ and $\lgl c^T_n:n<N\rgl$, respectively, where $M\le N\le\om$,
 we say that
a map $\varphi:S\ra T$ is {\em passing type preserving (ptp)} if and only if  the following hold:
\begin{enumerate}
\item
$|s|<|t|$ implies  that $|\varphi(s)|<|\varphi(t)|$.
\item
$\varphi$ takes each coding node in $S$ to a  coding node in $T$,
and $\gamma_{\varphi(c^S_n)}=\gamma_{c^S_n}$ for each $n\le M$.
\item
$\varphi$  preserves passing types:
For any  $s\in S$  and
 $m<M$ with $|c^S_{m-1}|<|s|$,\\
$\varphi(s)(\varphi(c^S_m); \{\varphi(c^S_0),\dots \varphi(c^S_{m-1})\})
\sim
s(c^S_m;\{c^S_0,\dots, c^S_{m-1}\})$.
\end{enumerate}
\end{defn}

\begin{thm}[Persistence]\label{thm.persistence}
Let $\mathcal{K}$ be a \Fraisse\ class
and $\bK$ an enumerated \Fraisse\ structure for
$\mathcal{K}$.
Suppose that $\bK$ satisfies \EEAP$^+$ or LSDAP$^+$. Let $\bT$ be
 a diagonal coding tree representing a copy of $\bK$,
 let
$\bD\sse \bT$  be any antichain of coding nodes  representing $\bK$,
and let
 $A$ be any    antichain of coding nodes in $\bD$.
Then for any subset $D\sse\bD$ representing a copy of $\bK$,
there is a similarity copy of $A$ in $D$; that is,
$A$ persists in $D$.
\end{thm}

\begin{proof}
We shall be working under the assumption  that either (a) there is an  antichain  of coding nodes $\bD\sse\bT\sse\bS$
such that $\bK\re\bD\cong \bK$,
or
(b)
that for every  antichain  of coding nodes $\bD\sse\bT\sse \bU$
such that $\bK\re \bD\cong\bK$,
there is a subset $D$ also coding $\bK$
with
the property that for each non-terminal node $t\in D$
 and for each $\gamma\in \Gamma$,
 there is a coding node  in $D\cap\bC_\gamma$ extending $t$.
%(Recall Remark \ref{rem.Gamma_ts}.)
In either case, we  let   $\bD$ be an antichain of coding nodes  in $\bT$ representing a copy of $\bK$,
where $\bD$ is constructed as in Lemma \ref{lem.bD}.
Throughout, we shall use the notation $\bU$, but keep in mind that if (a) above holds, then we are working in $\bS$.

Without loss of generality, we may assume that
 $\bK\re \bD \cong^{\om}\bK$,
 by thinning $\bD$ if necessary.
Let $D\sse\bD$ be any subset
 such that $\bK\re D\cong\bK$;  let
 $\bJ$ denote $\bK\re D$.
Again, without loss of generality, we may assume that
$\bJ\cong^{\om}\bK$.
Let $\bC=\{c_n:n<\om\}$ denote the set of all coding nodes in $\bU$, and note that $D\sse\bD\sse\bC$.
Then  the map $\varphi:\bC\ra  D$ via
$\varphi(c_n)=c^{D}_n$ is passing type preserving, where $\lgl c^{D}_n:n<\om\rgl$ is the enumeration of the nodes in $D$ in order of increasing length.
(Note that in the case of LSDAP$^+$, a passing type preserving map  automatically preserves the $\psi$-value of the meets of coding nodes.)

Define
 \begin{equation}
 \overline{D}
 =\{c^{D}_n\re |c^{D}_m|: m\le n<\om\}.
 \end{equation}
 Then $\overline{D}$ is a union of level sets  (but is not meet-closed).
 We extend the map $\varphi$ to a map $\bar{\varphi}:\bU\ra\overline{D}$ as follows:
 Given $s\in \bU$, let $n$ be least such that $c_n\contains s$ and   $m$ be the integer such that
 $|s|=|c_m|$,
  and define $\bar{\varphi}(s)=\varphi(c_n)\re|\varphi(c_m)|$;
in other words, $\bar{\varphi}(s)=c^{D}_n\re|c^{D}_m|$.

\begin{lem}\label{lem.barvarphiptp}
$\bar{\varphi}$ is passing type preserving.
\end{lem}

\begin{proof}
For $s\in \bU(m)$,
let $n>m$ be  least such that $s=c_n\re |c_m|$.
Then for any
for $i<m$,
\begin{align}
\bar{\varphi}(s)(\varphi(c_i); \{\varphi(c_0),\dots,\varphi(c_{i-1})\})
&=(\varphi(c_n)\re |\varphi(c_m)|)(\varphi(c_i); \{\varphi(c_0),\dots,\varphi(c_{i-1})\})\cr
&=(c^{D}_n\re|c^{D}_m|)(c^{D}_i; \{c^{D}_0,\dots,c^{D}_{i-1}\})\cr
&=c^{D}_n(c^{D}_i; \{c^{D}_0,\dots,c^{D}_{i-1}\})\cr
&\sim c_n(c_i; \{c_0,\dots,c_{i-1}\})\cr
&= (c_n\re|c_m|)(c_i; \{c_0,\dots,c_{i-1}\})\cr
&=s(c_i; \{c_0,\dots,c_{i-1}\})
\end{align}
where the $\sim$ holds
since $\varphi:\bC\ra D$ is ptp.
Therefore, $\bar{\varphi}$ is ptp.
\end{proof}

Given a fixed subset $S\sse \bU$ and $s\in S$, we let $\widehat{s}$  denote the set of all $t\in S$ such that $t\contains s$.
The ambient set $S$ will either be
$\bU$
 or $\overline{D}$,  and will  be clear from the context.
We say that a set $X$ is {\em cofinal} in $\widehat{s}$  (or {\em cofinal} above $s$)  if and only if
for each $t\in\widehat{s}$, there is some  $u\in X$ such that $u\contains t$.
 A subset $L\sse D$ is  called {\em large} if and only if there is some $s\in\bU$ such that
 $\varphi^{-1}[L]$ is cofinal in
 $\widehat{s}$.
 We point out that since $D$ is a set of coding nodes,
 for any $L\sse D$,
 $\varphi^{-1}[L]$ is a subset of $\bC$.

\begin{lem}\label{lem.lemma1}
Let $n<\om$
and $L\sse D$
be given.
Suppose  $L=\bigcup_{i<n}L_i$ for some $L_i\sse D$.  If $L$ is large,
then there is an $i<n$ such that $L_i$ is large.
\end{lem}

\begin{proof}
Suppose not.
Since $L$ is large, there is some $t\in\bU$ such that $\varphi^{-1}[L]$ is cofinal above $t$.
Since $L_0$ is not large,
there  is some $s_0\contains t$ such that $\varphi^{-1}[L_0]\cap \widehat{s_0}=\emptyset$.
Given $i<n-1$ and $s_i$, since $L_{i+1}$ is not large, there is some $s_{i+1}\contains s_i$ such that
$\varphi^{-1}[L_{i+1}]\cap \widehat{s_{i+1}}=\emptyset$.
At the end of this recursive construction, we obtain an $s_{n-1}\in\bS$ such that for all $i<n$,
$\varphi^{-1}[L_{i}]\cap \widehat{s_{n-1}}=\emptyset$.
Hence, $\varphi^{-1}[L]\cap \widehat{s_{n-1}}=\emptyset$, contradicting that
$\varphi^{-1}[L]$ is cofinal above $t$.
\end{proof}

Thus,  any partition of a large set into finitely many pieces
contains
at least one  piece which is large.

Given a subset $I\sse\om$,
let $\bK\re I$ denote the
substructure  of $\bK$ on vertices $\{v_i:i\in I\}$.
Recalling that $\bJ$ denotes $\bK\re D$, we let  $\bJ\re I$ denote the
substructure of $\bJ$ on
vertices $\{v^{D}_i:i\in I\}$, where $v^D_i$ is the vertex  represented by the coding node $c^D_i$.
The next lemma  will be applied in two important ways.
First, it will
aid in  finding  splitting nodes  in the meet-closure of $D$ (denoted by $\cl(D)$)
as  needed   to construct a similarity copy of a given antichain of coding nodes   $A$ inside $D$.
 Second,
 it
 will guarantee that we can find nodes in $D$ which have the needed passing types in order to continue building a similarity copy of $A$ in
 $D$.

Given a subset $L\sse\overline{D}$,
we say that $L$ is {\em large}
exactly when $L\cap D$ is large.
 Note that since $\varphi$ has range $D$,
 $\varphi^{-1}[L]$ is always a subset of $\bC$.
Given  a finite set
$I\sse\om$
and $1$-types $\sigma,\tau$
over $\bJ\re I$ and $\bK\re I$, respectively,
 we write $\sigma\sim\tau$ exactly when
 for each $i\in I$,
 $\sigma(v^D_i; \bJ\re I_i)\sim
 \tau(v_i;\bK\re I_i)$,
 where $I_i=\{j\in I:j<i\}$.

\begin{lem}\label{lem.lemma2}
Suppose  $t$ is in $\overline{D}$
 and     $\widehat{t}$  is   large.
 Let $s_*\in\bU$ be such that
 $\varphi^{-1}[\,\widehat{t}\,]$ is cofinal in $\widehat{s_*}$.
 Let  $i$ be    the index such that
$|s_*|=|c_{i}|$,  and
let  $I\sse i$,
 $n\ge i$,
 $I'=I\cup\{n\}$, and
$\ell=|c^{D}_n|$ be given.
For any
complete
 $1$-type $\sigma$ over
 $\bJ\re I'$ such that $\sigma\re (\bJ\re I)\sim s_*\re (\bK\re  I)$,
let
\begin{equation}
L_{\sigma}=
\bigcup\{\widehat{u}:u\in \widehat{t}\re \ell\mathrm{\ and\ }
u\re (\bJ\re I')\sim\sigma\}.
\end{equation}
Then
$L_{\sigma}$ is large.
\end{lem}

\begin{proof}
Fix an $s\contains s_*$ with $|s|>|c_n|$ such that
 $s\re(\bK\re I')\sim \sigma$ holds.
 Suppose towards a contradiction that
$L_{\sigma}$ is not large, and
fix  an extension  $s'\contains s$ such that
$\varphi^{-1}[L_{\sigma}]
\cap \widehat{s'}=\emptyset$.
Since $\varphi^{-1}[\,\widehat{t}\,]$ is cofinal in $\widehat{s_*}$,
there is a coding node $c_j$
in  $\varphi^{-1}[\,\widehat{t}\,]
$
extending $s'$.
Notice that $c_j$ being in $\varphi^{-1}[\,\widehat{t}\,]$ implies that $\varphi(c_j)$ extends $t$.
Moreover,
since
$c_j$ extends $s$
and  $\varphi$ is passing type preserving,
it follows that
$\varphi(c_j)\re (\bJ\re I')\sim\sigma$.
Thus,
$\varphi(c_j)$
 is in $L_{\sigma}$
 and hence,
$c_j$ is in $\varphi^{-1}[L_{\sigma}]$.
But  then
$c_j\in \varphi^{-1}[L_{\sigma}]
\cap \widehat{s'}$, a contradiction.
\end{proof}

For the remainder of the proof,
fix  a diagonal  antichain of coding nodes $A\sse\bD$.
Let $\lgl c^A_i:i<p\rgl$  enumerate the  nodes in $A$ in order of increasing length, where $p\le \om$, noting that  each $c^A_i$ is a coding node.
For each $i<p$,
let $\gamma_i$ denote $\gamma_{c^A_i}$.

Let $B$ denote the meet-closure of $A$;
 label the nodes of $B$ as $\lgl b_i:i<q\rgl$ in increasing order of length, where $q\le \om$.
Thus, each node in  $B$ is either a member of $A$ (hence, a coding node) or else a splitting node of degree two which is the meet of two  nodes in $A$.
Our goal is to build a similarity copy of $B$ inside the meet-closure of $D$, denoted $\cl(D)$;
that is, we aim to build a  similarity map $f$  from $B$ into $\cl(D)$ so that  $f[B]\sim B$.
Now the map
$\varphi$ is already passing type preserving.
The challenge is to get a $\prec$- and meet- preserving map which is still passing type preserving from  $B$ into $\cl(D)$.

First notice that  $B\re 1=A\re 1$.
If we are working in $\bS$, then $B\re 1$ is a subset of
$D\re 1=\bS(0)=\Gamma$ with  possibly more than one node.
If we are working in $\bU$, then $B\re 1$ is the singleton $D\re 1$.
Without loss of generality, we may assume
that $|c^A_0|>1$.
Let
$f_{-1}$ be the  empty map,
let
$T_{-1}$ denote $B\re 1$,
 let $N_{-1}=1$, and
let $\psi_{-1}$ be the identity map on
$T_{-1}$.
Let $\widehat{D}$ be the tree induced by cl$(D)$.
Let $M_{-1}=1$, and for each $k<q$, let
$M_k=|b_{k-1}|+1$, where we make the convention $|b_{-1}|=0$.

For each $k<q$
we will recursively define meet-closed sets $T_k\sse\widehat{D}$, maps $f_k$ and $\psi_k$, and $N_k<\om$ such that the following hold:
\begin{enumerate}
\item
$f_k$ is a $+$similarity embedding of $\{b_i:i<k\}$ into $T_k$.
\item
$|t|\le N_k$ for all $t\in T_k$.

\item
All maximal nodes of $T_k$ are either in $T_k\re N_k$, or else in the range of $f_k$.

\item
$\theta_k$ is a $\prec$ and passing type preserving  bijection of
$B\re M_k$ to $T_k\re N_k$.

\item
$T_{k-1}\sse T_k$, $f_{k-1}\sse f_k$, and $N_{k-1}<N_k$.
\end{enumerate}

The idea behind $T_k$ is that it will contain a similarity  image of $\{b_i:i<k\}\cup (B\re M_k)$,
the nodes in the image of $ B\re M_k$ being the ones we need to continue extending in order to build a similarity copy of $B$ in $\cl(D)$.
(In the case of LSDAP$^+$,  we further stipulate in (1) that $f_k$ is an L$+$-similarity embedding.)

Assume  now that  $k<q$, and (1)--(6) hold for all $k'<k$.
We have two cases.
\vskip.1in

\noindent\bf Case I. \rm
$b_k$ is a splitting node.
\vskip.1in

Let $i<j<p$ be such that  $b_k=c^A_i\wedge c^A_j$.
Let $t_k=\theta_k( b_k\re M_k)$,
recalling that by (5),
 $t_k$ is a member of $T_k\re N_k$.
By (3), $\widehat{t_k}$ is large,
so we can fix  a coding node $c_{n}\in \varphi^{-1}[\widehat{t_k}]$.
Then
 $c^{D}_n=\varphi(c_n)\contains t_k$.
Let $N_{k+1}=|c^{D}_{n+1}|$.

Our goal is to find two incomparable nodes which extend $t_k$
and have cones which are
 large.
Recalling that $N_k=|t_k|$,
 let
 \begin{equation}
 I=\{i<\om: |c_i^{D}|< N_k\},
 \end{equation}
 and let $I'=I\cup\{n\}$.
Let $\sigma$ and $\tau$ be distinct $1$-types over
 $\bJ( I')$
 such that
 both
  $\sigma\re  \bJ( I)$
  and
    $\tau\re \bJ( I)$
 equal  $t_k\re  \bJ( I)$.
For each    $\mu\in\{\sigma,\tau\}$,
let
\begin{equation}
L_{\mu}=\bigcup\{\widehat{u}:u\in\widehat{t_k}\re N_{k+1}\mathrm{\  and \ }
u(c_n^{D}; \bJ( I))=\mu\}.
\end{equation}
By Lemma \ref{lem.lemma2},
both
$L_\sigma$ and  $L_\tau$ are large.
It then follows from Lemma \ref{lem.lemma1} that
there are
 $t_{\sigma},t_{\tau}\in\widehat{t_k}\re N_{k+1}$
such that
$t_{\sigma}\in L_\sigma$ and
$t_{\tau}\in L_\tau$,
and both
$\widehat{t_{\sigma}}$ and $\widehat{t_{\tau}}$  are
large.
Since $\sigma\ne\tau$, it follows that $t_\sigma\ne t_\tau$.
Hence, $t_\sigma$ and $t_\tau$ are incomparable, since they have the same length, $N_{k+1}$.
Since
both $t_\sigma\contains t_k$ and $t_\tau\contains t_k$,
we have $t_\sigma\wedge t_\tau\contains t_k$.

As $(B\cap \widehat{b_k})\re M_k$ has size exactly two,
define $\theta_{k+1}$
on  $(B\cap \widehat{b_k})\re M_{k+1}$
to be the unique $\prec$-preserving map
onto $\{t_\sigma,t_\tau\}$.
Let $E_k$ denote $(B\setminus\widehat{b_k})\re M_k$.
For $s\in E_k$, choose some
$t_s\in\widehat{\theta_k(s)}\re N_{k+1}$ such that
$\widehat{t_s}$ is large.
This is possible by Lemma \ref{lem.lemma1}, since
$\bigcup\{\widehat{t}:t\in \widehat{\theta_k(s)}\re N_{k+1}\}$
is large.
Every $s\in E_k$ has a unique extension  $s'\in B\re M_{k+1}$.
Define $\theta_{k+1}(s')=t_s$.
Let $f_{k+1}$ be the extension of $f_k$ which sends
$b_k$ to $t_\sigma\wedge t_\tau$,
 and let
\begin{equation}
T_{k+1}=T_k\cup \{t_\sigma,t_\tau, t_\sigma\wedge t_\tau\}\cup
\{t_s:s\in E_k\}.
\end{equation}

(In the case of LSDAP$^+$, if $\psi(b_k)=m$ then we take $\sigma$ and $\tau$ above so that the pair $\{\sigma\re (\bK\re \{v^D_n\}), \tau\re (\bK\re \{v^D_n\})  \}$ corresponds to $\psi$-value $m$ in  the Labeled Extension Property.)

This completes Case I.
\vskip.1in

\noindent\bf Case II. \rm
$b_k$ is a coding  node.
\vskip.1in

In this case, $b_k=c^A_j$ for some $j<p$.
By the Induction Hypothesis,
for each $t\in T_k\re N_k$,
$\widehat{t}$ is large;
so we can choose  some $s_t\in\bU$ such that $\varphi^{-1}[\,\widehat{t}\,]$ is  cofinal above $s_t$.
Fix
$t_*=\theta_k(b_k\re  M_k)\in T_k\re N_k$.
Choose a coding node $c_n\contains s_{t_*}$ in $\bU$ such that $|c_n| >\max\{|s_t|:t\in T_k\re N_k\}$
and $\gamma_{c_n}=\gamma_j$, the $\gamma\in\Gamma$ which the vertex $v^A_j$ satisfies.
(In the case that $\bD\sse\bS$, this $\gamma_j$ is
already guaranteed  since $c_n\contains s_{t_*}\contains \gamma_j$.
If $\bD\sse\bU$, there are cofinally many coding nodes extending $s_{t_*}$ which satisfy $\gamma_j$.)
Let
 $d_k$ denote $c^{D}_n=\varphi(c_n)$,
 noting that $\gamma_{d_k}=\gamma_j$.
 Extend $f_k$ by defining
 $f_{k+1}(b_k)=d_k$,
 and
let $N_{k+1}=|c^D_{n+1}|$.
If $q<\om$ and $k=q-1$, we are done.
Otherwise,
we must extend the other members of
$(T_k\re N_k)\setminus\{t_*\}$  to nodes in
$\widehat{D}\re N_{k+1}$ so as to satisfy (1)--(6).

For each  $i\in\{k,k+1\}$,
let
$E_i=(B\re M_i)\setminus \{b_i\re M_i\}$.
Fix an   $s\in E_k$
and let $t=\theta_k(s)$, which is a node in $T_k\re N_k$.
Note that there is a unique $s'\in  E_{k+1}$
such that $s'\contains s$.
Let $A\rl j$ denote $\{c^A_i:i\le j\}$,
$\sigma$ denote $s'\re (A\rl j)$, and
$f_k[A\rl j]$ denote $\{f_k(c^A_i):i<j\}$.
Let
 $I=\{i<\om:c^{D}_{i}\in
f_k[A\rl j]\}$.
Our goal is to find a $t'\contains t$ with $|t'|>|d_k|$
such that $t'(d_k;\bJ(I))\sim \sigma$.

Take  $c_m$ to be any coding node in $\bS$  extending
$s$
such that
$|c_m|>|c_n|$ and $c_m(c_n;A\rl j)\sim\sigma$.
Such a  $c_m$  exists by  \EEAP.
Then  $ \varphi(c_m) (d_k; \bJ(I))\sim \sigma$,
 since $\varphi$ is passing type preserving.
By Lemma \ref{lem.lemma2},
\begin{equation}
L_\sigma:=
\bigcup\{\widehat{u}:u\in\widehat{t}\re N_{k+1}\mathrm{\ and\ }
u(d_k;f_k[A\rl j])
\sim
\sigma\}
\end{equation}
is
large.
Thus, by Lemma \ref{lem.lemma1},
there is some $u_s\in\widehat{t}\re N_{k+1}$ such that $\widehat{u}_s$ is large.
Define $\theta_{k+1}(s)=u_s$.
This builds
\begin{equation}
T_{k+1}=
T_k\cup\{d_k\}\cup\{\psi_{k+1}(s):s\in E_k\}
\end{equation}
and concludes the construction in Case II.
\vskip.1in

Finally, let $f=\bigcup_k f_k$.
Then $f$ is a similarity map from $B$ to $f[B]$,
and thus,
 the antichain of coding nodes in  $f[A]$ is similar to $A$.
Therefore, all similarity types of  diagonal  antichains of coding nodes  persist in $\bJ$.
\end{proof}

As the antichain in the previous theorem can be infinite, we immediately obtain the following corollary.

\begin{cor}\label{cor.presbD}
Suppose $\bK$
satisfies \EEAP$^+$ or LSDAP$^+$.
Given
$D$ a subset of $\bD$ which represents a copy of $\bK$,
there is a subset $D'$ of $D$ such that $D'\sim \bD$.
\end{cor}

Combining the previous
results,
we obtain canonical partitions  which
are simply described by similarity types.

\begin{thm}[Simply characterized big Ramsey degrees]\label{thm.bounds}
Let
$\bK$ be an enumerated \Fraisse\ structure for a
 \Fraisse\ class $\mathcal{K}$
with relations of arity  at most two
 such that $\bK$ satisfies
 \EEAP$^+$ or LSDAP$^+$.
Given  $\bG\in\mathcal{K}$,
the partition $\{\Sim(C):C\in\Sim(\bG)\}$ is
 a canonical partition of
 the copies of $\bG$ in $\bK$.
 It follows that
the big Ramsey degree $T(\bG,\bK)$
equals the  number of similarity types of antichains of coding nodes in $\bT$ representing $\bG$.
That is,
$$
T(\bG,\bK)=|\Sim(\bG)|.
$$
\end{thm}

\begin{proof}
Let   $\bG\in\mathcal{K}$ be given, and suppose $h$ is a coloring of all copies of $\bG$ in $\bK$ into finitely many colors.
By Theorem
\ref{thm.onecolorpertype},
there is an antichain of coding nodes $\bD\sse\bT$ which codes a copy of $\bK$, and moreover, for each $C\in\Sim(\bG)$,
$h$ is constant on $\Sim_{\bD}(C)$.
Let $\bJ=\bK\re\bD$.

Given any
subcopy
$\bJ'$ of $\bJ$,
Theorem   \ref{thm.persistence}
implies that $\Sim_D(C)\ne\emptyset$ for each $C\in\Sim(\bG)$,
where $D=\bS\re \bJ'$.
Thus,
$\{\Sim(C):C\in\Sim(\bG)\}$ is
 a canonical partition of
 the copies of $\bG$ in $\bK$.
It follows that
$T(\bG,\mathcal{K})= |\Sim(\bG)|$.
 \end{proof}

We now  apply
Theorem \ref{thm.bounds}
to show that \Fraisse\ structures
with \EEAP$^+$ or LSDAP$^+$
satisfy the conditions of Zucker's Theorem 7.1 in \cite{Zucker19},  yielding
 Theorem \ref{thm.main}.
Zucker used colorings of embeddings
rather than colorings of copies throughout \cite{Zucker19}.
Our task now is to translate
Theorem \ref{thm.bounds},
which uses
colorings of copies of a given structure, into the setting of \cite{Zucker19}.
To do so, we need to review the following  notions from
\cite{Zucker19}.

Let $\bK$ be an enumerated \Fraisse\ structure for a \Fraisse\ class $\mathcal{K}$.
An {\em exhaustion} of $\bK$ is a sequence
$\{\bfA_n:n<\om\}$ with each $\bfA_n\in\mathcal{K}$, $\bfA_n\sse\bfA_{n+1}\sse\bK$,  such that $\bK=\bigcup_{n<\om}\bfA_n$.
Given $m\le n$, write $H_m:=\Emb(\bfA_m,\bK)$ and $H^n_m:=\Emb(\bfA_m,\bfA_n)$.
For $f\in H^n_m$, the function $\hat{f}: H_n\ra H_m$ is defined by $\hat{f}(s)=s\circ f$, for each $s\in H_n$.
(Here we are using Zucker's notation, so $s$ is denoting an embedding rather than a node in $\bU$.)

The following  terminology is  found in
Definition 4.2  in \cite{Zucker19}.
A set $S\sse H_m$ is {\em unavoidable} if for each embedding $\eta:\bK\ra\bK$, we have $\eta^{-1}(S)\ne\emptyset$.
Fix $k\le r<\om$ and let
 $\gamma: H_m\ra r$  be a coloring.
 We call
 $\gamma$  an
{\em unavoidable $k$-coloring}  if
the image of $\gamma$, written $\mathrm{Im}(\gamma)$, has cardinality $k$,
and for each $i<r$,
we have $\gamma^{-1}(\{i\})\sse H_m$ is either empty
or unavoidable.
Thus, an unavoidable coloring is essentially the same concept as persistence, with the addition that attention is also given to the embedding.

The following is taken from Definition 4.7 in \cite{Zucker19}:
Let $\gamma$ and $\delta$ be colorings of $H_m$.
We say that $\delta$ {\em refines} $\gamma$ and write $\gamma\le\delta$ if whenever $f_0,f_1\in H_m$ and $\delta(f_0)=\delta(f_1)$,
then $\gamma(f_0)=\gamma(f_1)$.
For $m\le n<\om$, $\gamma$ a coloring of $H_m$, and $\delta$ a coloring of $H_n$,
we say that $\delta$ {\em strongly refines} $\gamma$ and write $\gamma\ll\delta$ if for every $f\in H^n_m$, we have that $\gamma\circ \hat{f}\le \delta$.

Theorem 7.1 in \cite{Zucker19}, which we state next,  provides conditions for showing that a \Fraisse\ limit admits a big Ramsey structure.
%We will then apply this theorem to show that all \Fraisse\ structures with relations of arity at most two that have \EEAP$^+$ admit big Ramsey structures, thus giving Theorem \ref{thm.main}.

\begin{thm}[Zucker,  \cite{Zucker19}]\label{thm.Zucker7.1}
Let $\bK=\bigcup_{n<\om}\bfA_n$ be a \Fraisse\ structure,
where $\{ \bA_n : n < \om\}$ is an exhaustion of $\bK$,
and suppose each $\bfA_n$ has
finite big Ramsey degree
$R_n$
in $\bK$.
Assume that for each $m<\om$, there is an unavoidable $R_m$-coloring $\gamma_m$ of $H_m$ so that
$\gamma_m\ll\gamma_n$ for each $m\le n<\om$.
Then $\bK$ admits a big Ramsey structure.
\end{thm}

Now we show how to translate our results so as to apply Theorem \ref{thm.Zucker7.1}.
Given an enumerated \Fraisse\ structure $\bK$,  we point out that  $\{\bK_n:n<\om\}$ is an exhaustion of $\bK$.
Theorem \ref{thm.bounds} shows that
 $\bK_n$ has finite big Ramsey degree
 $T(\bK_n,\bK)=|\Sim(\bK_n)|$ for colorings of {\em copies} of $\bK_n$ in  $\bK$.
Recalling Remark \ref{rem.embvscopy},
 the big Ramsey degree for {\em embeddings} of $\bK_n$ into $\bK$ is
 $T(\bK_n,\bK)\cdot|\mathrm{Aut}(\bK_n)|$.

\begin{thm}\label{thm.apply}
Suppose $\mathcal{K}$ is a
\Fraisse\ class  with
\Fraisse\ limit $\bK$ and with
canonical partitions characterized via diagonal antichains of coding nodes in a coding tree of $1$-types.
 Then the conditions of Theorem \ref{thm.Zucker7.1} are satisfied.
\end{thm}

\begin{proof}
Recalling that $\bD$ denotes the diagonal antichain of coding nodes constructed in Lemma \ref{lem.bD},
we shall abuse notation and use
$\bK$  to denote the  structure $\bK\re \bD$.
Thus, the universe of $\bK$ will (without loss of generality) be $\om$,
and embeddings  $s$ of initial segments $\bK_n$ into $\bK$ will produce diagonal antichains
$\bD\re s[\bK_n]\sse \bD$.
Given $n<\om$,
let $T_n:=T(\bK_n,\bK)$, and
let $\lgl C^n_0,\dots, C^n_{T_n-1}\rgl$ be an enumeration of $\Sim(\bK_n)$, a set of representatives of the similarity types of diagonal antichains of coding nodes representing a copy of $\bK_n$.
Let Aut$(\bK_n)$ denote the set of automorphisms of $\bK_n$.

As $\bK_n$ has vertex set $n=\{0,\dots,n-1\}$,
its vertex set is linearly ordered.
Given $s\in H_n$,
let $\bfA:=s[\bK_n]$, with vertex set
$\lgl \mathrm{a}_0,\dots, \mathrm{a}_{n-1}\rgl$ written in increasing order as a subset of $\om$.
Let $p_s$ denote  the permutation of $n$
defined by $s(j)=\mathrm{a}_{p_s(j)}$, for $j<n$.
Given  $\ell<T_n$,
let $\bfC^n_\ell$ denote the structure $\bK\re C^n_\ell$,
and
 let  $\lgl v^\ell_0,\dots, v^\ell_{n-1}\rgl$  denote the vertex set  of $\bfC^n_\ell$
  in increasing order as a subset of $\om$.
  Let $P_\ell$
 be the set of permutations $p$ of $n$ such that
 the map $j \mapsto v^\ell_{p(j)}$, $j<n$,
 induces an isomorphism from $\bK_n$ to $\bfC^n_\ell$.
Note that   $|P_\ell|=|$Aut$(\bK_n)|$.

Letting $R_n=T(\bK_n,\bK)\cdot|\mathrm{Aut}(\bK_n)|$, we define an unavoidable coloring $\gamma_n:H_n\ra R_n$ as follows:
For $s\in H_n$,
define $\gamma_n(s)=\lgl t,p_s\rgl$,
where
$t<T_n$ is the index satisfying
$\bD\re \bfB_s\sim C^n_t$.
Then $\gamma_n$ is an unavoidable coloring, by Theorem \ref{thm.persistence}.

Let $m\le n<\om$.
To show that $\gamma_m\ll \gamma_n$, we start by fixing  $f\in H^n_m$ and $s,t\in H_n$ such that
$\gamma_n(s)=\gamma_n(t)$.
Note that $f:\bK_m\ra\bK_n$ is completely determined by
its behavior on the sets of vertices.
Thus, we equate $f$ with its induced injection from $m$ into $n$.
Let $\bfA,\bfB$ denote the structures $s[\bK_n],t[\bK_n]$, respectively.
Let  $A=\bD\re\bfA$ and $B=\bD\re \bfB$, the diagonal antichains of coding nodes representing the structures $\bfA,\bfB$, respectively.
Since $\gamma_n(s)=\gamma_n(t)$, it follows that
$A\sim B$ and $p_s=p_t$.
It follows that $p_s\circ f=p_t\circ f$.

Our task is to show that $\gamma_m(\hat{f}(s))=\gamma_m(\hat{f}(t))$.
Letting $\lgl \mathrm{a}_0,\dots,\mathrm{a}_{n-1}\rgl$ denote the increasing enumeration of the vertices in $\bfA$,
we see that $s\circ f$ is an injection from $m$ into
$\{\mathrm{a}_j:j<n\}$.
Letting $\bar{m}=\{j<n:\exists i<m\, (\mathrm{a}_j=s\circ f(i))\}$,
and letting $\mu$ be the strictly increasing injection from $\bar{m}$ into $m$,
we see that
$p_{\hat{f}(s)}$ is the permutation of $m$ given by
$p_{\hat{f}(s)}(i)=
\mu\circ f\circ p_s(i)$.
Likewise,  $t\circ f$ is an injection from $m$ into
$\{\mathrm{b}_j:j<n\}$, where
$\lgl \mathrm{b}_0,\dots,\mathrm{b}_{n-1}\rgl$ denotes the increasing enumeration of the vertices in $\bfB$.
Since $p_s=p_t$,
we see that $f\circ p_s=f\circ  p_t$,
and hence, the set of indices
$\{j<n:\exists i<m\, (\mathrm{b}_j=t\circ f(i))\}$
equals $\bar{m}$.
Thus, $p_{\hat{f}(t)}(i)=
\mu\circ f\circ p_t(i)$ for each $i<m$.
Hence,
$p_{\hat{f}(s)}=p_{\hat{f}(t)}$.

$\hat{f}\circ s$ maps $\bK_m$ to the substructure $\bfA'$ of  $\bfA$ on vertices $\{\mathrm{a}_{p_s\circ f(i)}:i<m\}$.
This substructure induces the antichain of coding nodes $A':=\{c^A_{p_s\circ f (i)}:i<m\}\sse A$;
that is, $A'=A\re \bfA'$.
Similarly,
$t\circ \hat{f}$ maps $\bK_m$ to the  substructure
$\bfB'$ of  $\bfB$ on vertices $\{\mathrm{b}_{p_t\circ f(i)}:i<m\}$;
this induces  the antichain of coding nodes  $B':=B\re\bfB'=\{c^B_{ p_s\circ f(i)}:i<m\}\sse B$.
Since $p_s=p_t$,
we have $p_s\circ f= p_t\circ f$, and since $A\sim B$, it follows that $A'\sim B'$.
Let $\ell<T_m$ be the index such that $A'\sim B'\sim C^m_{\ell}$.
Then   $\gamma_m(\hat{f}(s))=(\ell,p_{\hat{f}(s)})=\gamma_m(\hat{f}(t))$, since
$p_{\hat{f}(s)}=p_{\hat{f}(t)}$.
Therefore, $\gamma_m\ll \gamma_n$.
\end{proof}

\begin{rem}
We point out that Theorem \ref{thm.apply} holds for \Fraisse\ classes  with relations  of any arity.
It is not hard   to check that it applies to  the ternary betweenness relation. 
However, it is likely that most \Fraisse\ classes with non-trivial relations of arity at least three   will not satisfy the hypothesis of that theorem.
\end{rem}

For languages with relations of arity at most two,
the big Ramsey structure  of a  \Fraisse\
limit $\bK $
with \EEAP$^+$ 
or LSDAP$^+$
is obtained simply  by  expanding the language $\mathcal{L}$ of $\bK$ to the language
$\mathcal{L}^*=\mathcal{L}\cup\{\triangleleft,\mathscr{Q}\}$,
where $\triangleleft$ and $\mathscr{Q}$ are not in $\mathcal{L}$, $\triangleleft$ is a binary relation symbol, and $\mathscr{Q}$ is a quaternary relation symbol.
In fact, by
Theorem \ref{thm.apply},
this will be the case for any \Fraisse\ class with canonical partitions characterized via diagonal antichains of coding nodes in a  coding tree of $1$-types.
The big Ramsey $\mathcal{L}^*$-structure $\bK^*$ for $\bK$ is described as follows.

Let $\bD$ be the diagonal antichain of coding nodes from the proof of Theorem \ref{thm.apply}, and
recall the linear order $\prec$ on $\bS$ induced in the natural way from a linear order of the relation symbols in the language (see Subsection 3.2 of Part I for a detailed description).
Note that $(\bD,\prec)$
is isomorphic to the rationals as a linear order.
Following Zucker
in Section 6 of \cite{Zucker19},
let $R$ be the quaternary relation on $\bD$ given by:
For
$p\preccurlyeq q\preccurlyeq r\preccurlyeq s\in \bD$,
set
\begin{equation}
R(p,q,r,s)\Longleftrightarrow |p\wedge q|\le |r\wedge s|,
\end{equation}
where $p\preccurlyeq q$ means either $p\prec q$ or $p=q$.
Without loss of generality, we may
use $\bK$ to denote $\bK\re \bD$.
Define $\bK^*$
be
the expansion of $\bK$ to the
language
$\mathcal{L}^*$ in which
$\triangleleft$ is interpreted as $\prec$ and $\mathscr{Q}$ is interpreted as $R$.
Then we have the following.

\begin{thm}\label{thm.BRS}
Let $\mathcal{K}$ be a \Fraisse\ class in language $\mathcal{L}$
with relation symbols of arity at most two
and $\bK$ a \Fraisse\ limit of $\mathcal{K}$.
Suppose that $\bK$
satisfies
\EEAP$^+$ or LSDAP$^+$,
 and let   $\mathcal{L}^*=\mathcal{L}\cup\{\triangleleft,\mathscr{Q}\}$, where
 $\triangleleft$ is a binary relation symbol and $\mathscr{Q}$ is a quaternary relation symbol.
Then  the  $\mathcal{L}^*$-structure
$\bK^*$
is a big Ramsey structure for $\bK$.
\end{thm}

\begin{proof}
Theorems
\ref{thm.bounds} and
 \ref{thm.apply}
  imply the existence of a big Ramsey structure for $\bK$.
  Moreover, the proof of  Theorem
 \ref{thm.apply} shows that $\bK^*$ satisfies
  Definition \ref{defn.bRs}
  of a big Ramsey structure.
\end{proof}

%%%%%%%%%%%%%%%%%%%%%%
%%%%%%%%%%%%%%%%%%%%%%
%%%%%%%%%%%%%%%%%%%%%%
%%%%%%%%%%%%%%%%%%%%%%
%%%%%%%%%%%%%%%%%%%%%%
%%%%%%%%%%%%%%%%%%%%%%

We
 now
can quickly deduce Theorem \ref{thm.SESAPimpliesORP}
below:
The ordered expansion of the age of any \Fraisse\ structure
 with relations of arity at most two
satisfying \EEAP$^+$ is a Ramsey class.
This theorem offers  a new approach
 for  proving that
 such
\Fraisse\ classes have
ordered expansions which are Ramsey,
 complementing the much more general,
famous partite construction method of \Nesetril\ and \Rodl\ (see \cite{Nesetril/Rodl77} and \cite{Nesetril/Rodl83}) which is  at the heart  of   finite  structural  Ramsey  theory.

For the rest of this section, we work only with
\Fraisse\ classes in a finite relational  language $\mathcal{L}$
with relation symbols of arity at most two.
Let  $<$ be an additional binary relation symbol not in $\mathcal{L}$, and
let $\mathcal{L}'=\mathcal{L}\cup\{<\}$.
Let $\mathcal{K}^{<}$ denote
the class of all ordered expansions of structures in $\mathcal{K}$, namely,
the  collection of
all
$\mathcal{L}'$-structures in which
$<$ is interpreted as a linear order and whose reducts to the language $\mathcal{L}$  are members of $\mathcal{K}$.
Since  $\mathcal{K}$ has disjoint amalgamation by assumption, $\mathcal{K}^<$ will be a \Fraisse\ class
with disjoint amalgamation.
 We denote the \Fraisse\ limit of $\mathcal{K}^<$ by $\bK^<$, and note that $\bK^<$ is universal for all countable
 $\mathcal{L}'$-structures in which the relation symbol $<$ is interpreted as a linear order.
 We shall write
$\bfM' := \lgl \bfM,<'\rgl$
for any $\mathcal{L}'$-structure interpreting $<$  as a linear order; it will be understood that $\bfM$ is an $\mathcal{L}$-structure and that  $<'$ is
   the linear order on
   ${\rm M}$ interpreting $<$.

\begin{defn}\label{defn.comb}
Given   a \Fraisse\ class $\mathcal{K}$ and an enumerated \Fraisse\ structure $\bK$,
let $\bU$ be the unary-colored coding tree of $1$-types for $\bK$.
We call  a finite  antichain $C$ of coding nodes in $\bU$
 a {\em  comb} if and only if
for any two coding nodes $c,c'$ in $C$,
\begin{equation}
|c|<|c'|\ \Longleftrightarrow   \   c\prec c',
\end{equation}
where $\prec$ is the lexicographic order on $T$.
\end{defn}

\begin{thm}\label{thm.SESAPimpliesORP}
Let $\mathcal{K}$ be a \Fraisse\ class in a finite relational language $\mathcal{L}$
with relation symbols of arity at most two, and suppose that the \Fraisse\ limit of
$\mathcal{K}$ has \EEAP$^+$.
Then the ordered expansion $\mathcal{K}^<$ of $\mathcal{K}$ has the Ramsey property.
 \end{thm}

\begin{proof}
Let $\bK$ be any enumerated \Fraisse\ limit of $\mathcal{K}$.
Then
$\bK$ has universe $\om$, and may be
regarded  as a linearly ordered structure in order-type $\om$,
that is, as an $\mathcal{L}'$-structure $\lgl \bK,\in\rgl$
in which the relation symbol $<$ is interpreted as the order inherited from $\om$.
Let $\bU$ be the coding tree of $1$-types associated with $\bK$.

Let $\bfA',\bfB'$ be members of $\mathcal{K}^{<}$ such that
$\bfA'$ embeds into $\bfB'$.
Fix  a finite coloring $f$  of all
copies of
$\bfA'$ in $\lgl \bK,\in\rgl$.
Note that in this context,
a substructure $\lgl \bfA^*,\in\rgl$ of $\lgl \bK,\in\rgl$
is a
copy
of $\bfA'$
 when there is an
 $\mathcal{L}'$-isomorphism between $\lgl \bfA, <' \rgl$ and $\lgl \bfA^*, \in \rgl$.

Let $\bT$ be a diagonal coding subtree of $\bU$, and let $A\sse \bT$ be a comb representing $\bfA'$.
Thus,  if $\lgl c^A_i:i<m\rgl$ is the enumeration of $A$ in order of increasing length,
then  the coding node  $c^A_i$ represents the  $i$-th vertex  of $\bfA'$
(according to its linear ordering $<'$).
Let $f^*$ be the coloring  on $\Sim(A)$ induced by $f$.
By Theorem \ref{thm.onecolorpertreetype},
there is a diagonal coding subtree $T\sse \bT$ in which all
similarity copies of $A$ have the same $f^*$ color.

Let $D\sse T$ be an antichain of coding nodes
 representing  a copy of $\bK$.
(This is guaranteed by
Lemma \ref{lem.bD}.)
By Theorem \ref{thm.persistence},
there is a subset $B^*\sse  D$ such that $B^*$ is a comb representing a copy of  $\bfB'$ in the order inherited on the coding nodes in $B^*$.
Then every copy of $\bfA'$ represented by a set of coding nodes in $B^*$
is represented by a comb, and hence
 has the same $f$-color.
 Since
 $\lgl \bK,\in\rgl$
 is an $\mathcal{L}'$-structure interpreting the relation symbol $<$ as a linear order,  $\lgl \bK,\in\rgl$
 embeds into the \Fraisse\ limit of $\mathcal{K}^{<}$, and so
 it follows from
 Definition
 \ref{defn.RP}
that $\mathcal{K}^{<}$ has the Ramsey property.
\end{proof}

\begin{rem}
It is impossible for any comb   to represent a copy of
a
\Fraisse\ structure $\bK$
satisfying \EEAP$^+$
when
$\bK$ has at least one non-trivial relation
of arity at least two.
The contrast  between similarity types of diagonal  antichains of $1$-types persisting in every copy of $\bK$
in a coding tree
 and  combs (or any other fixed similarity type)
 being sufficient
 to prove
 the Ramsey property for the ordered expansion of its age
lies at the heart of the difference between big Ramsey degrees for $\bK$ and the Ramsey property for $\mathcal{K}^<$.
\end{rem}

In the paper \cite{Hubicka/Nesetril19},
\Hubicka\ and \Nesetril\ prove general theorems
 from
which the majority of  Ramsey classes can be deduced.
In particular,
Corollary 4.2 of \cite{Hubicka/Nesetril19} implies that every relational \Fraisse\ class with  free amalgamation  has
an ordered expansion with the Ramsey property.
So  for \Fraisse\ classes satisfying \SFAP,
 Theorem \ref{thm.SESAPimpliesORP} provides a new proof of special case of  a known result.
However, we  are not aware  of a  prior  result implying
Theorem \ref{thm.SESAPimpliesORP}
in its full generality.

A different approach to recovering the  ordered Ramsey property is given in  \cite{Hubicka_CS20}.
In that paper, \Hubicka's  results on big Ramsey degrees via the Ramsey theory  of   parameter spaces
recover
a special case of
the \Nesetril-\Rodl\ theorem \cite{Nesetril/Rodl77},
that the class of finite  ordered triangle-free graphs has the Ramsey property.

These approaches to proving the Ramsey property for  ordered \Fraisse\ classes may seem at first glance very different from the partite construction method.
However,
the  methods  must be  related at some fundamental level, similarly to the relationship between the  Halpern-\Lauchli\  and
Hales-Jewett theorems.
It will be interesting to see  if this could lead to new  Hales-Jewett theorems corresponding to the various forcing constructions
(in \cite{DobrinenJML20}, \cite{DobrinenH_k19}, \cite{Zucker19}, and this paper)
which have been used to determine finite and exact big Ramsey degrees.

%%%%%%%%%%%%%%%%%%%%%%%%%%%%%%%%%%
%%%%%%%%%%%%%%%%%%%%%%%%%%%%%%%%%%%
%%%%%%%%%%%The old section 3

\section{Examples of \Fraisse\ structures satisfying \texorpdfstring{\EEAP$^+$}{EEAP+} or \texorpdfstring{LSDAP$^+$}{LSDAP+}} \label{sec.EEAPClasses}

We now  investigate  \Fraisse\  classes
which  have \Fraisse\ structures satisfying  \EEAP$^+$ or LSDAP$^+$.
Such
classes
seem to  fall roughly into three categories:  Free amalgamation classes
of relational structures
in which
 any forbidden substructures  are $3$-irreducible
 (Definition \ref{defn.3irred}), and their ordered expansions;
 disjoint amalgamation classes 
which are
 unrestricted
 (Definition \ref{defn.unconst}), and  their ordered expansions;
 and
  disjoint amalgamation classes which are
 in some sense
 ``$\bQ$-like''.
At the end of this
section,
we provide
 a  catalogue of  \Fraisse\
 structures
 which have been investigated for indivisibility or for  big Ramsey degrees.
 The list is
 non-exhaustive, as research is ongoing, but it provides a view  of many of the main results currently known, including the new results from Parts I and II.

 First, we consider free amalgamation classes.
 The following definition appears in \cite{Conant17}, and occurs implicitly in work on indivisibility in \cite{El-Zahar/Sauer94}.

 \begin{defn}\label{defn.3irred}
Let $r \ge 2$, and let $\mathcal{L}$ be a finite relational language.  An $\mathcal{L}$-structure $\bF$ is {\em $r$-irreducible} if for any $r$ distinct elements $a_0, \ldots , a_{r-1}$ in $\mathrm{F}$ there is some $R \in \mathcal{L}$ and $k$-tuple $\bar{p}$ with entries from $\mathrm{F}$, where
$k \ge r$
is the arity of $R$, such that each $a_i$, $i < r$, is among the entries of $\bar{p}$, and $R^\bF(\bar{p})$ holds.  We  say $\bF$ is {\em irreducible} when $\bF$ is 2-irreducible.
 \end{defn}

Note that for $r > \ell \ge 2$, a structure that is $r$-irreducible need not be $\ell$-irreducible.  This is because for any structure $\mathbf{F}$
such that $|\mathrm{F}|< r$, it is vacuously the case that $\mathbf{F}$ is $r$-irreducible, but if $|\mathrm{F}| \ge \ell$, then $\mathbf{F}$ may not be
$\ell$-irreducible.

 Given a set $\mathcal{F}$ of  finite $\mathcal{L}$-structures,
 let $\Forb(\mathcal{F})$ denote the
 class of finite $\mathcal{L}$-structures $\bA$ such that no member of $\mathcal{F}$ embeds into $\bA$.
 It is a standard fact that a \Fraisse\ class $\mathcal{K}$ is a free amalgamation class if and only if $\mathcal{K} = \Forb(\mathcal{F})$ for some
 set $\mathcal{F}$ of finite irreducible $\mathcal{L}$-structures.  (See  \cite{Siniora/Solecki20}  for a proof).

Recall Theorem 4.3 from Part I, that if $\mathcal{K}$ is a \Fraisse\ class  satisfying SFAP,
then both the \Fraisse\ limit of $\mathcal{K}$ and the \Fraisse\ limit of its ordered expansion $\mathcal{K}^<$ satisfy SDAP$^+$.

\begin{prop}\label{prop.FA}
Let $\mathcal{L}$ be a
finite relational language
and $\mathcal{F}$ be a (finite or infinite)  collection of finite $\mathcal{L}$-structures which are
irreducible and
$3$-irreducible.
Then
$\Forb(\mathcal{F})$ satisfies \SFAP.
Hence both 
the \Fraisse\ limit of $\Forb(\mathcal{F})$ 
and  the \Fraisse\ limit of $\Forb(\mathcal{F})^<$
satisfy 
\EEAP$^+$.
\end{prop}

\begin{proof}
Since the structures in $\mathcal{F}$ are irreducible, $\Forb(\mathcal{F})$ is a free amalgamation class.
Fix  $\bfA,\bfB,\bfC\in \Forb(\mathcal{F})$  with  $\bfA$ a substructure of both  $\bfB$ and  $\bfC$
and
$\mathrm{C} \!\setminus\! \mathrm{A} = \{v, w\}$.
Let
$\sigma,\tau$
be
realizable
$1$-types over $\bfB$ with $\sigma\re\bfA=\type(v/\bfA)$
and  $\tau\re\bfA=\type(w/\bfA)$.
Suppose
 $\bfD\in \Forb(\mathcal{F})$ is
 a $1$-vertex extension of $\bfB$ realizing $\sigma$.
 Thus, $\mathrm{D}=\mathrm{B}\cup\{v'\}$ for some $v'$
 such that
  $\type(v'/\bfB)=\sigma$.

Extend $\bfD$ to an $\mathcal{L}$-structure $\bfE$
by one vertex $w'$ satisfying  $\type(w'/\bfB)=\tau$ such that  for each relation symbol $R\in\mathcal{L}$,
letting
$k$
denote the arity of $R$,
we have the following:
\begin{enumerate}
\item[(a)]
For each $k$-tuple
$\bar{p}$ with entries from $\mathrm{A}\cup \{v', w'\}$,
let $\bar{q}$ be the $k$-tuple with entries from $\mathrm{A}\cup\{v, w\}$
such that each occurrence of $v'$, $w'$ in $\bar{p}$ (if any) is replaced by
$v$, $w$, respectively,
 and all other entries remain the same.
Then we require that $R^\bfE(\bar{p})$ holds if and only if $R^\bfC(\bar{q})$ holds.
\item[(b)]
If $k\ge 3$, then for each $b\in\mathrm{B}\setminus\mathrm{A}$
and each $k$-tuple $\bar{p}$
with entries from $\mathrm{E}$ such that $b, v', w'$ are among the entries of $\bar{p}$, we require that
$\neg R^\bfE(\bar{p})$ holds.
\end{enumerate}
It follows from (a) that $\bfE\re(\mathrm{A}\cup\{v',w'\})\cong \bfC$.
It remains to show
 that $\bfE$ is a member of $\Forb(\mathcal{F})$.
To do so,
it suffices to show that
no
 $\mathbf{F}\in\mathcal{F}$
embeds
 into  $\bfE$.

Suppose toward a contradiction that
 some  $\mathbf{F}\in\mathcal{F}$
embeds into $\bfE$.
Let $\mathbf{F}'$ denote an
  embedded copy of $\mathbf{F}$, with universe
$\mathrm{F}'\sse\mathrm{E}$.
For what follows, it helps to recall that $\mathrm{E}=\mathrm{B}\cup\{v',w'\}$.
Since $\bfD$ is in $\Forb(\mathcal{F})$,
$\mathbf{F}$ does not embed into $\bfD$, so
$\mathrm{F}'$ cannot be contained in $\mathrm{D}$.
Hence $w'$ must be in $\mathrm{F}'$.
Likewise, since
$\tau$ is a realizable $1$-type over $\bfB$,
the substructure $\bfE\re(\mathrm{B}\cup\{w'\})$  is in $\Forb(\mathcal{F})$ and hence does not contain a copy of $\mathbf{F}$.
Therefore,  $v'$ must be in
$\mathrm{F}'$.
By (a), since $\bfC$ is in $\Forb(\mathcal{F})$, the substructure $\bfE \re ( \mathrm{A} \cup \{v', w' \}) $
does not contain a copy of $\mathbf{F}$. Hence there must be some $b \in \mathrm{B}\setminus\mathrm{A}$ such that
$b$ is in $\mathrm{F}'$.
Since $\mathbf{F}$ is  $3$-irreducible,
there must be  some
relation symbol $R\in\mathcal{L}$
with arity $k \ge 3$, and some $k$-tuple $\bar{p}$ with entries from $\mathrm{F}'$ and with $b, v', w'$ among its entries, such that $R^{\bF'}(\bar{p})$ holds.
However,
 (b) implies
 $\neg R^\bfE(\bar{b})$ holds,
contradicting that $\mathbf{F}'$ is a copy of $\mathbf{F}$ in $\bfE$.
Therefore, $\mathbf{F}$ does not embed into $\bfE$.
It follows that $\bfE$ is a member of
$\Forb(\mathcal{F})$.

We have established that $\Forb(\mathcal{F})$ has \SFAP.
The Proposition follows by 
 Theorem 4.20 of \cite{CDPI}.
\end{proof}

We now consider a
type of \Fraisse\ class that is a generalization, to arbitrary finite relational languages, of the \Fraisse\ classes in finite binary relational
languages that were considered in
 \cite{Laflamme/Sauer/Vuksanovic06}.

 \begin{defn}\label{defn.unconst}
Given a
relational
language $\mathcal{L}$, letting $n$ denote the highest arity of any relation symbol in $\mathcal{L}$, for each $1\le i\le n$, let $\mathcal{L}_i$ denote
the sublanguage consisting of
the relation symbols in $\mathcal{L}$ of arity $i$.
Let $\mathcal{C}_i$ be a
set of
structures in the language $\mathcal{L}_i$ with domain $\{0,\dots,i-1\}$
that is closed under isomorphism.
Following \cite{Laflamme/Sauer/Vuksanovic06}, we call
$ \mathcal{C}_i$ a {\em universal constraint set}.

Let  $\mathcal{U}_{\mathcal{C}_i}$ denote the class of all finite relational structures $\bfA$ in the language
$\mathcal{L}_i$  for which the following holds:
Every induced substructure of $\bfA$ of cardinality $i$ is isomorphic to one of the structures in $\mathcal{C}_i$.
Let
$\mathcal{C}:=\bigcup_{1\le i\le n}\mathcal{C}_i$
and let
$\mathcal{U}_{\mathcal{C}}$ denote the free superposition of the
classes
$\mathcal{U}_{\mathcal{C}_i}$, $1\le i\le n$.
We call such a class $\mathcal{U}_{\mathcal{C}}$ {\em unrestricted}.
  \end{defn}

It is straightforward to check that an unrestricted class $\mathcal{U}_{\mathcal{C}}$  is  a \Fraisse\ class with disjoint amalgamation.

In \cite{Laflamme/Sauer/Vuksanovic06},
Laflamme, Sauer, and Vuksanovic characterized the exact big Ramsey degrees for
the \Fraisse\ structures in finite binary relational languages
whose ages are
unrestricted.
We now show that
arbitrary
unrestricted
\Fraisse\ classes  satisfy \EEAP.

\begin{prop}\label{prop.LSVSFAP}
Let $\mathcal{U}_{\mathcal{C}}$ be an
unrestricted
\Fraisse\ class.
Then $\mathcal{U}_{\mathcal{C}}$ satisfies \EEAP, hence also its ordered expansion $\mathcal{U}^{<}_{\mathcal{C}}$ satisfies \EEAP.
Moreover, the
 \Fraisse\ limits of
 $\mathcal{U}_{\mathcal{C}}$ and
$\mathcal{U}^{<}_{\mathcal{C}}$
have \EEAP$^+$.
\end{prop}

\begin{proof}
Let  $\mathcal{L}$ be a finite relational language  with $n$ denoting the highest arity of any relation symbol in $\mathcal{L}$,
 and let  $\mathcal{C}=\bigcup_{1\le i\le n}\mathcal{C}_i$,
where each $\mathcal{C}_i$ is a universal constraint set.
Suppose $\bfA,\bfC\in \mathcal{U}_{\mathcal{C}}$ are given
such that $\bfC$ extends $\bfA$ by two vertices $v,w$.
Here, we simply let $\bfA'=\bfA$ and $\bfC'=\bfC$.
Suppose $\bfB\in \mathcal{U}_{\mathcal{C}}$  is any structure containing $\bfA$ as a substructure, and let $\sigma,\tau$ be $1$-types over $\bfB$ satisfying $\sigma\re \bfA=\type(v/\bfA)$ and $\tau\re\bfA=\type(w/\bfA)$.
Suppose further that $\bfD\in \mathcal{U}_{\mathcal{C}}$ extends $\bfB$ by one vertex, say $v'$, such that $\type(v'/\bfB)=\sigma$.

Let $\rho$ be the $1$-type of  $w$ over $\bfC \re( \mathrm{A} \cup \{v\})$.
Take  $\bfE$  to be any $\mathcal{L}$-structure extending $\bfD$ by one vertex, say $w'$, such that
the following hold:
$\type(w'/\bfB)=\tau$ and
$\type(w'/ (\bfD\re (A \cup \{v'\}) )$ is
the $1$-type obtained by substituting $v'$ for $v$ in $\rho$.
If $\mathcal{C}_1$ is non-empty,
then we simply
take a structure $\mathbf{Z}\in \mathcal{U}_{\mathcal{C}_1}$
and declare
$w'$ to satisfy the unary relation which the vertex in $\mathbf{Z}$ satisfies.
For  each subset $G\sse E$ of cardinality at most $n$  containing  $v'$ and $w'$ and at least one vertex of $B\setminus A$,
letting $i$ denote  the cardinality of $G$,
the
$\mathcal{L}_i$-reduct of the structure $\bfE\re G$
 is isomorphic to  a member of $\mathcal{C}_i$.
Then $\bfE$ is a member of  $\mathcal{U}_{\mathcal{C}}$, and  $\bfE\re (\mathrm{A}\cup\{v',w'\})\cong \bfC$.

Thus, $\mathcal{U}_{\mathcal{C}}$ satisfies \EEAP.
By Proposition \ref{prop.LO_n} below,  the \Fraisse\ class of finite linear orders satisfies \EEAP.
As  \EEAP\ is preserved under free superpositions,
the ordered expansion $\mathcal{U}^{<}_{\mathcal{C}}$  also satisfies \EEAP.

Let $\mathbf{U}_{\mathcal{C}}$ denote an enumerated \Fraisse\ limit of $\mathcal{U}_{\mathcal{C}}$.
Since $\mathcal{U}_{\mathcal{C}}$ is unrestricted,
the coding tree $\bS(\mathbf{U}_{\mathcal{C}})$
has the property that  all nodes  of the same length  have the same branching degree.
It is simple to construct a diagonal coding tree inside $\bS(\mathbf{U}_{\mathcal{C}})$,
because the universal constraint set allows each node $s$ in any subtree  $T$
of  $\bS(\mathbf{U}_{\mathcal{C}})$
to be extended independently of the substructure represented by the coding nodes in $T$ of length less or equal to that of $s$.
Thus, the Diagonal Coding Tree property trivially holds.
Further,  (1)
of
Definition \ref{defn.ExtProp}  is trivially satisfied, and hence  $\mathbf{U}_{\mathcal{C}}$ has the Extension Property.
Thus, $\mathbf{U}_{\mathcal{C}}$ satisfies \EEAP$^+$.

For an enumerated \Fraisse\ limit $\mathbf{U}_{\mathcal{C}}^{<}$ of
 $\mathcal{U}^{<}_{\mathcal{C}}$,
a diagonal coding tree can be constructed inside $\bS(\mathbf{U}^{<}_{\mathcal{C}})$
similarly  to the construction in Lemma 4.11 in \cite{CDPI}.%4.31 in \cite{CDP21}. %%fix ref 
%\ref{lem.SFAPplusoderimpliesDCT}.
Again,  (1)
of
Definition \ref{defn.ExtProp}  is trivially satisfied, so the Extension Property holds.
Thus,
 $\mathbf{U}_{\mathcal{C}}^{<}$ satisfies \EEAP$^+$.
\end{proof}

%By Proposition \ref{prop.LO_n} below,  the \Fraisse\ class of finite linear orders satisfies \EEAP. As \SFAP\ implies \EEAP, and \EEAP\ is preserved under free superpositions (see Remark \ref{rem.freesup}), the ordered expansion $\mathcal{K}^<$ of any \Fraisse\ class $\mathcal{K}$ with \SFAP\  will also have \EEAP. In Theorem 4.28
 %\ref{thm.SFAPimpliesDCT} 
 %%fix ref
% and  Lemma 4.31 %\ref{lem.SFAPplusoderimpliesDCT}
 %%fix ref
% from \cite{CDP21}
 %we will show that
%we showed that whenever a \Fraisse\ class $\mathcal{K}$ has \SFAP, the \Fraisse\ limits of $\mathcal{K}$ and $\mathcal{K}^<$ satisfy  \EEAP$^+$. 

It is straightforward to check that for any $n\ge 2$, the class of finite $n$-partite graphs satisfies SFAP.
Theorem 4.20 of Part I showed that  for any \Fraisse\ class $\mathcal{K}$ satisfying SFAP, both $\mathcal{K}$ and its ordered version $\mathcal{K}^<$ have \Fraisse\ limits satisfying SDAP$^+$.
Applying 
Theorems \ref{thm.bounds}  and \ref{thm.BRS}, Propositions \ref{prop.FA} and \ref{prop.LSVSFAP},
and Theorems 1.2 and 4.20  and from Part I,  
we  obtain  the following.

%\begin{thm}\label{thm.supercool}
%Let $\mathcal{L}$ be a finite relational language, $\mathcal{K}$ a \Fraisse\ class in language $\mathcal{L}$, and $\mathcal{K}^<$ the ordered expansion of $\mathcal{K}$. If $\mathcal{K}$ is an unrestricted \Fraisse\ class, then $\mathcal{K}$ and  $\mathcal{K}^<$ have \EEAP.  If $\mathcal{K} = \Forb(\mathcal{F})$ for some set $\mathcal{F}$ of finite irreducible and  $3$-irreducible $\mathcal{L}$-structures, then $\mathcal{K}$  has \SFAP\ and $\mathcal{K}^<$ has \EEAP.   All such classes have \Fraisse\ limits satisfying \EEAP$^+$, and hence are  are indivisible. Moreover, the \Fraisse\ limits of
%all such classes with only unary and binary relations  admit big Ramsey structures, and their exact big Ramsey degrees have a simple characterization.
%\end{thm}

%{\color{purple}  I simplified the presentation of this theorem to get rid of redundancies with previous statements.}

\begin{thm}\label{thm.supercool}
Let $\mathcal{L}$ be a finite relational language,
$\mathcal{K}$ a \Fraisse\ class in language $\mathcal{L}$, and $\mathcal{K}^<$ the ordered expansion of $\mathcal{K}$.
Suppose  $\mathcal{K}$ is one of the following: an unrestricted \Fraisse\ class,
 $\Forb(\mathcal{F})$ for some set $\mathcal{F}$ of finite
irreducible and
 $3$-irreducible $\mathcal{L}$-structures, 
 or the class of finite $n$-partite graphs for some $n\ge 2$.
 Then the
  \Fraisse\ limit $\bK$
  of $\mathcal{K}$ and  
  the \Fraisse\ limit $\bK^<$ 
  of $\mathcal{K}^<$
  both
  satisfy \EEAP$^+$, and hence are indivisible.
Moreover, if the language of $\mathcal{K}$ has  only
unary and binary relation symbols,
then  $\bK$ and $\bK^<$ both
 admit big Ramsey structures, and their exact big Ramsey degrees have a simple characterization.
\end{thm}

We now discuss previous
results
 recovered by Theorem \ref{thm.supercool}, as well as their  original proof methods.

 In
    \cite{Laflamme/Sauer/Vuksanovic06},
 Laflamme, Sauer, and Vuksanovic characterized  the
 exact big Ramsey degrees of the Rado graph, generic directed graph, and generic tournament.
 More generally,
 they characterized  exact big Ramsey degrees
for  the \Fraisse\ limit of  any
unrestricted \Fraisse\ class in a language
consisting of
  finitely many binary relations.
Their proof utilized Milliken's theorem
  for strong trees \cite{Milliken79}
  and the method of envelopes, building
  on exact upper bound results for big Ramsey degrees of the Rado graph due to Sauer in \cite{Sauer06}.  
The  characterization in \cite{Laflamme/Sauer/Vuksanovic06} is exactly recovered in our
Theorem \ref{thm.bounds}.
The result for ordered expansions is new to this paper. 
The indivisibility result in its full generality for unrestricted \Fraisse\ structures with relations in any arity, as well as  their ordered expansions,   is also  new.

 Theorem \ref{thm.supercool} also   extends a result of
El-Zahar and Sauer
 \cite{El-Zahar/Sauer94}, in which
they
proved indivisibility for
free amalgamation classes of
 $k$-uniform hypergraphs ($k\ge 3$)
with forbidden
$3$-irreducible substructures.
As these structures have only one isomorphism type of singleton substructure, their result says
that  for any $k\ge 3$ and any collection $\mathcal{F}$ of
irreducible,
$3$-irreducible $k$-uniform hypergraphs,
vertices  in  $\Forb(\mathcal{F})$ have big Ramsey degree
one.

We mention that for each $n\ge 2$, the \Fraisse\ class of finite
 $n$-partite graphs is easily seen to satisfy \SFAP.
 John Howe proved in his PhD thesis \cite{HoweThesis} that the generic  bipartite  graph has finite big Ramsey degrees; his  methods use an adjustment of  Milliken's theorem.
 Finite  big Ramsey degrees for $n$-partite graphs for  all $n\ge 2$    follow from  the more recent work of  Zucker   in \cite{Zucker20};
  his methods use a flexible version of coding trees and envelopes, but  lower bounds were not attempted in that
  paper.

Next we consider disjoint amalgamation classes which are ``$\bQ$-like''  in
 that their resemblance to
 linear orders
  makes them  in
some sense  rigid enough  to satisfy \EEAP.
Starting with
 the rationals  as a linear order $(\bQ,<)$,
 we shall show that the  \Fraisse\ class of finite   linear orders  satisfies
 \EEAP, and that $(\bQ, <)$ satisfies
 \EEAP$^+$.
 Further,
 the rational linear order with a vertex partition into finitely many dense pieces
 satisfies
 \EEAP$^+$.
We obtain a hierarchy  of
linear orders with
nested convexly ordered equivalence relations
that each satisfy
\EEAP$^+$.

Given $n\ge 1$, let $\mathcal{LO}_n$ denote  the \Fraisse\ class   of finite
structures
with  $n$-many independent linear orders. The
language
for $\mathcal{LO}_n$ is $\{<_i:i<n\}$, with each $<_i$ a binary relation symbol.
Let 
$\mathcal{LO}$ denote $\mathcal{LO}_1$, the class of finite linear orders.

\begin{prop}\label{prop.LO_n}
The \Fraisse\ limit of
$\mathcal{LO}$,
namely the rational linear order $\bQ$,
  satisfies \EEAP$^+$.
For each   $n\ge 2$, $\mathcal{LO}_n$  satisfies \EEAP.
\end{prop}

\begin{proof}
Fixing  $n\ge 1$,
suppose   $\bfA$ and $\bfC$   are in $\mathcal{LO}_n$ with
  $\bfA$   a substructure
of $\bfC$ and
$\mathrm{C}\setminus \mathrm{A}=\{v,w\}$.
Let $\bfC'$ be the extension of $\bfC$ by one
 vertex, $a'$, satisfying the following:
 For each $i<n$,
 if $v<_i w$ in $\bfC$, then
 $v<_i a'$ and $a'<_i w$ are in $\bfC'$;
 otherwise, $w<_i a'$ and $a'<_i v$ are  in $\bfC'$.
Define  $\bfA' $ to be the induced  substructure $\bfC'\re (\mathrm{A}\cup\{a'\})$ of $\bfC'$.

Suppose  that $\bfB$ is a finite linear order containing
 $\bfA'$ as a substructure, and let
$\sigma$ and $\tau$ be
$1$-types over $\bfB$ with the property that
$\sigma\re\bfA'=\type(v/\bfA')$ and $\tau\re\bfA'=\type(w/\bfA')$.
Suppose that $\bfD$  is a one-vertex extension of $\bfB$ by the vertex $v'$  so that $\type(v'/\bfB)=\sigma$ holds.
Now
let $\bfE$ be an extension of $\bfD$ by one vertex $w'$ satisfying  $\type(w'/\bfB)=\tau$.
For each $i<n$, $v<_i w$ holds  in $\bfC'$
 if and only if
$x<_i a'$ is in $\sigma$ and
 $a'<_i x$ is in $\tau$.
 (The opposite,
  $w<_i v$, holds  in $\bfC'$
 if and only if
$a'<_i x$ is in $\sigma$ and
 $x<_i a'$ is in $\tau$.)
 It follows that $v'<_i w'$ holds  in $\bfE$ if and only if $v<_i w$  holds in $\bfC$.
Therefore, we automatically obtain $\bfE\re(\mathrm{A}\cup\{v',w'\})\cong \bfC$.
Thus,  \EEAP\ holds.

The Diagonal Coding Tree Property  for $\bS(\bQ)$ is 
straightforward to prove
and follows from Lemma \ref{lem.DCTLO}.
(Recall Example 3.4 from Part I.)
The Extension Property trivially holds.
 Hence, $\bQ$ satisfies
 \EEAP$^+$.
\end{proof}

Next,  we consider \Fraisse\ classes
of structures with a linear order and
a finite vertex partition.
Following the notation in \cite{Laflamme/NVT/Sauer10},
for each $n\ge 2$, let  $\mathcal{P}_n$  denote the
\Fraisse\ class
with language $\{<,P_1,\dots,P_n\}$, where  $<$ is a
binary relation symbol and each $P_i$ a unary relation symbol, such that in any structure in $\mathcal{P}_n$,
$< $ is interpreted as a linear order and the interpretations of the $P_i$ partition the vertices.
The \Fraisse\ limit of $\mathcal{P}_n$, denoted by
$\bQ_n$,  is
the
rational linear order with a partition of its underlying set into $n$ definable pieces, each of which is dense in $\bQ$.

\begin{prop}\label{bQn}
For each $n\ge 1$,
the \Fraisse\ limit $\bQ_n$ of
the \Fraisse\ class $\mathcal{P}_n$ satisfies \EEAP$^+$.
\end{prop}

\begin{proof}
Fixing  $n\ge 1$,
suppose   $\bfA$ and $\bfC$   are in $\mathcal{P}_n$ with
  $\bfA$   a substructure
of $\bfC$ and
$\mathrm{C}\setminus \mathrm{A}=\{v,w\}$.
Let $\bfC'$ be the extension of $\bfC$ by one
 vertex, $a'$,  such that
 $v<w$ in $\bfC$ if and only if
 $v< a'$ and $a'< w$ in $\bfC'$;
 (otherwise, $w<v$ and
 $w<a'$ and $a'<v$ hold  in $\bfC'$).
Let  $\bfA'=\bfC'\re (\mathrm{A}\cup\{a'\})$.

Given any $\bfB,\sigma, \tau,\bfD,v''$ as in  (2) and (3)  of Part (B) of Definition \ref{defn.EEAP_new},
any extension of $\bfD$ by one vertex $w''$ to a structure $\bfE$ with $\type(w''/\bfB)=\tau$ automatically has $v''<w''$ holding in $\bfE$ if and only if $v'<a'<w'$ holds in $\bfA'$.
Since each $P_i$ is a unary relation,
 $P_i(x)$ is in $\sigma$ if and only if $P_i(v')$ holds.
 Thus,  it follows that $P_i(v'')$ holds  in $\bfE$ for that $i$ such that $P_i(v)$ holds in $\bfA$.
 Likewise for $w''$.
 Therefore, $\bfE\re(\mathrm{A}\cup\{v'',w''\})\cong \bfC$.
 Thus,  \EEAP\ holds.

The Extension Property trivially holds for $\bQ_n$. 
The Diagonal Coding Tree Property will follow from  the next
Lemma \ref{lem.DCTLO}.
\end{proof}

\begin{lem}\label{lem.DCTLO}
There is a diagonal coding tree
representing  $\bQ_n$, for each $n\ge 1$
Hence,
these structures
have the Diagonal Coding Tree Property.
\end{lem}

\begin{proof}
We have already seen in Figure 1.\ for  Example 3.5  in \cite{CDPI} 
that $\bS(\bQ)=\bU(\bQ)$ is a skew tree with binary splitting.
Similarly, for $n\ge 2$, $\bU(\bQ_n)$ is a skew tree with binary splitting.
In general,
 to construct a diagonal coding subtree  $\bT$ of $\bU(\bQ_n)$, it only remains to choose splitting nodes for $\bT$ (which are coding nodes in $\bU$ but not in $\bT$) and then choose other coding nodes in $\bU$ to be inherited as the coding nodes in $\bT$, so as to satisfy requirements (2) and (3) of Definition \ref{defn.sct}, the definition of diagonal coding subtree.
The construction is a slight modification of the one given in \cite{Laflamme/NVT/Sauer10}, where they constructed diagonal antichains of (non-coding) trees for $\bQ_n$.

Take the only node in $\bU(0)$, $c_0$, to be the least splitting node in $\bT$.
 Let $\bT\re 1$ consists of
 the two immediate successors of $c_0$ in $\bU$, say $s_0\prec s_1$.
Then extend $s_0$ to the next coding node in $\bU$, and label this node $c^{\bT}_0$.
If $n\ge 2$, we also require that
$c^{\bT}_0$   satisfies the same unary relations as $c_0$ does.
Take any extension $t_1\contains s_1$ in $\bU$ of length
 $|c^{\bT}_0|$.
The set $\{t_0,t_1\}$ make up the nodes in $\bT$ at the level of its least coding node,  $c^{\bT}_0$.
Extend $c^{\bT}_0$ $\prec$-leftmost in $\bU$, call this node $u_0$.
There is only one immediate successor of $t_1$ in $\bU$, call it $u_1$.
Let $\bT\re (|c^{\bT}_0|+1)=\{u_0,u_1\}$.

In general, given $n\ge 1$ and $\bT$ constructed up to nodes of length $|c^{\bT}_{n-1}|+1$,
enumerate these nodes in $\prec$-increasing order as $\lgl t_i:i<n+2\rgl$.
Let $j$ denote the index of the node that will be extended to the next coding node, $c^{\bT}_n$.
This is the only node that needs to branch before the level of $c^{\bT}_n$.
Let $s$  be the shortest   splitting node in $\bU$
extending $t_{j}$.
Denote  its immediate successors by $s_0,s_1$, where
$s_0\prec s_1$.
Let $c^{\bT}_n$ be the coding node of least length in $\bS$ extending $s_0$; if $n\ge 2$, also require that $c^{\bT}_n$ satisfies the same unary relation as $c_n$.
Extend all the nodes $s_0$ and $t_i$, $i\in (n+2)\setminus\{j\}$ to nodes in $\bS$ of length
$|c^{\bT}_n|$.
These nodes along with $c^{\bT}_n$ construct $\bT\re |c^{\bT}_n|$.
Take the $\prec$-leftmost extension of $c^{\bT}_n$ to be its immediate successor in
$\bT$.
All other nodes in $\bT\re |c^{\bT}_n|$ have only one immediate successor in $\bS$, so there is no choice to be made.

This constructs a diagonal tree $\bT$  representing a copy of $\bQ_n$.
Note that taking the $\prec$-leftmost extension of each coding node has the effect that
all extensions of any  coding node  $c^{\bT}_n$ in $\bT$ include the formula $x<v^{\bT}_n$,
satisfying  (3) of the definition of  diagonal coding tree.
\end{proof}

Next, we consider  \Fraisse\ classes with
a linear order and
finitely many
convexly ordered
equivalence relations:
An equivalence relation on a linearly ordered set is {\em convexly ordered} if each of its
equivalence classes is an interval with respect to the linear order.

Given  the
language
$\mathcal{L}=\{<,E\}$,
where $<$ and $E$ are binary relation symbols,
let $\mathcal{COE}$ denote the \Fraisse\ class
of
{\em convexly ordered equivalence relations},
$\mathcal{L}$-structures in which $<$ is interpreted as a linear order and $E$ as an equivalence relation that is convex with respect to that order.
The \Fraisse\ limit of  $\mathcal{COE}$, denoted by
$\bQ_{\bQ}$, is the dense linear order without endpoints with an equivalence relation
that has
infinitely many equivalence classes, each
an interval of order-type
$\bQ$, and with an induced order on the set of equivalence classes that is also of order-type $\bQ$.
One can  think of $\bQ_\bQ$ as $\bQ$  copies of $\bQ$ with the lexicographic order.
This  structure was
described
by Kechris, Pestov, and Todorcevic in  \cite{Kechris/Pestov/Todorcevic05}, where they proved that its automorphism group is extremely amenable;
 from
 the main result of
  \cite{Kechris/Pestov/Todorcevic05},
  it then follows that
$\mathcal{COE}$ has the Ramsey property.
This generated interest in  the question of
whether $\bQ_{\bQ}$ has finite big Ramsey degrees or big Ramsey structures.

Let $\mathcal{COE}_2$ denote the \Fraisse\ class in language $\{<,E_0,E_1\}$, where $<$,
$E_0$ and $E_1$
are binary relation symbols,
such that in any structure in $\mathcal{COE}_2$, $<$ is interpreted as a linear order, $E_0$ and $E_1$ as convexly ordered equivalence
relations, and with the additional property that the interpretation of $E_1$ is a coarsening of that of $E_0$; that is, for any
$\bfA$ in $\mathcal{COE}_2$, $a\, E_0^\bfA\, b$ implies $a\, E_1^\bfA\, b$.
Then Flim$(\mathcal{COE}_2)$ is
 $\bQ_{\bQ_{\bQ}}$, that is $\bQ$ copies of $\bQ_{\bQ}$; we shall  denote this as
  $(\bQ_{\bQ})_2$.
  One can see that this recursive construction gives rise to a hierarchy of dense linear orders without endpoints with finitely many convexly ordered  equivalence relations, where each successive equivalence relation coarsens the previous one.
In general,   let $\mathcal{COE}_{n}$ denote the \Fraisse\ class
in the  language $\{<,E_0,\dots, E_{n-1}\}$
where $<$ is interpreted as a linear order and
each $E_i$  $(i<n)$ is interpreted as a
convexly ordered
equivalence relation, and such that
 for each $i<n-2$, the interpretation of $E_{i+1}$ coarsens that of $E_i$.
Let
$(\bQ_{\bQ})_{n}$ denote
the \Fraisse\ limit of  $\mathcal{COE}_{n}$.

More generally, we may consider \Fraisse\ classes
that are a
blend
of the $\mathcal{COE}_n$ and $\mathcal{P}_p$, having
finitely many linear orders,
finitely many convexly ordered equivalence relations, and a partition into finitely many pieces (each of which, in the \Fraisse\ limit, will be dense).
 Let $\mathcal{L}_{m,n,p}$ denote the  language consisting of
 finitely many binary relation symbols, $<_0,\dots, <_{m-1}$, finitely many binary relation symbols $E_0,\dots, E_{n-1}$,
and finitely many unary relation symbols $P_0,\dots,P_{p-1}$.
 A \Fraisse\ class $\mathcal{K}$ in language
 $\mathcal{L}_{m,n,p}$  is a member of
 $\mathcal{LOE}_{m,n,p}$
 if each $<_i$, $i<m$, is interpreted as a linear order,
each $E_j$, $j<n$, is interpreted as a
convexly ordered
equivalence relation with respect to exactly one of the linear orders $<_{i_j}$, for some $i_j<\ell$,
and the interpretations of the $P_k$, $k<p$, induce a vertex partition into at most $p$ pieces.
Let $\mathcal{LOE}$ be the union over all triples $(m,n,p)$
of $\mathcal{LOE}_{m,n,p}$.
Let $\mathcal{COE}_{n,p}$
be
the  \Fraisse\ class
in
$\mathcal{LOE}_{1,n,p}$ for which  the reduct to the language $\{<_0,E_0,\dots, E_{n-1}\}$ is a member of $\mathcal{COE}_n$.

\begin{prop}\label{prop.loe}
For any $n,p$,
the \Fraisse\ limit of
 $\mathcal{COE}_{n,p}$
satisfies \EEAP, the Labeled  Diagonal Coding Tree Property, and the Labeled Extension Property. 
\end{prop}

Proposition \ref{prop.loe} will follow from the next two lemmas.

\begin{lem}\label{lem.COESDAP}
Each
\Fraisse\ class  in  $\mathcal{LOE}$ satisfies  \EEAP.
\end{lem}

\begin{proof}
Suppose   $\bfA$ and $\bfC$   are in $\mathcal{K}$ with
  $\bfA$   a substructure
of $\bfC$ and
$\mathrm{C}\setminus \mathrm{A}=\{v,w\}$.
The unary relations are handled exactly as they were in Proposition \ref{bQn}, so we  need to check that \EEAP\ holds for the binary relations.

Let $\bfC'$ be an  extension of $\bfC$
by   vertices $a'_k$ ($k<m+n$) satisfying the following:
For each $i<m$,
$v<_i w$ if and only if $v<_i a'_i$ and $a'_i<_i w$ in $\bfC'$.
Given  $j<n$,  if
$v\, E_j\, w$  holds  in $\bfC$, then require that $a'_{m+j}$ satisfies  $v\, E_j\, a'_{m+j}$ and $w\, E_j\, a'_{m+j}$   in $\bfC'$.
If $v \not \hskip-.06in  E_j\, w$  holds in $\bfC$,
then require that
$a'_{m+j}$ satisfies
 $v\, E_j\, a'_{m+j}$ and $w \not \hskip-.06in E_j\, a'_{m+j}$ in $\bfC'$.
Let $\bfA'=\bfC'\re(\mathrm{A}\cup\{a'_k:k<m+n\})$.

Suppose  that $\bfB\in\mathcal{K}$ contains
 $\bfA'$ as a substructure, and let
$\sigma$ and $\tau$ be  consistent realizable $1$-types over $\bfB$ with the property that   $\sigma\re\bfA'=\type(v/\bfA')$ and $\tau\re\bfA'=\type(w/\bfA')$.
Suppose that $\bfD$  is a one-vertex extension of $\bfB$ by the vertex $v'$  satisfying $\type(v'/\bfB)=\sigma$.
Now
let $\bfE$ be an extension of $\bfD$ by one vertex $w'$ satisfying  $\type(w'/\bfB)=\tau$.
The same argument as in the proof of Proposition \ref{prop.LO_n}
ensures that  for each $i<m$, $v'<_i w'$ in $\bfE$ if and only if $v<_i w$ in $\bfC$.

Fix $j<n$.
 If
$v\, E_j\, w$  in $\bfC$,  then  as $v\, E_j\, a'_{m+j}$ and $w\, E_j\, a'_{m+j}$  hold in $\bfC'$,  the formula
 $x\, E_j\, a'_{m+j}$  is in both $\sigma$ and $\tau$.
Since $v'$ satisfies $\sigma$ and $w'$ satisfies $\tau$,
it follows that $v\, E_j\, w$  in $\bfE$.
On the other hand, if
 $v \not \hskip-.06in E_j\, w$ holds  in $\bfC$,
 then  the formula
  $x\, E_j\, a'_{m+j}$ is in $\sigma$  and $x \not \hskip-.06in E_j\, a'_{m+j}$ is in $\tau$.
  Again, since
 $v'$ satisfies $\sigma$ and $w'$ satisfies $\tau$,
it follows that $v  \not \hskip-.06in E_j\, w$  in $\bfE$.
Thus,
$\bfE\re(\mathrm{A}\cup\{v',w'\})\cong \bfC$.
Hence \EEAP\ holds.
\end{proof}

\begin{lem}\label{lem.Q_Qbiskew}
$(\bQ_{\bQ})_n$, for each $n\ge 1$,
has  the Labeled Diagonal Coding Tree Property and the Labeled Extension Property.
Moreover,
the \Fraisse\ limit of
any  class $\mathcal{K}$ in
$\mathcal{LOE}_{1,n,p}$
also has the Labeled  Diagonal Coding Tree Property and the Labeled Extension Property.
\end{lem}

\begin{proof}
We present the construction for $\bQ_{\bQ}$ and then
discuss the construction for the more general case.
Let
$\bU$ denote $\bU(\bQ_{\bQ})$.
It may aid the reader to review Figure 3.\ in \cite{CDPI}, 
%to recall  Figure \ref{fig.QQtree}, 
where a graphic is   presented for a particular enumeration of
$\bQ_{\bQ}$.

We  construct a subtree $\bT$ of
$\bU$
 which is diagonal and such that
for each $m$, the immediate successors of the nodes in
$\bT\re |c^{\bT}_m|$
have $1$-types over $\bQ_{\bQ}\re \{v^{\bT}_j:j\le m\}$ which
are in one-to-one correspondence (in $\prec$-order)
with the $1$-types in
$\bU(m+1)$.
The idea is relatively simple:
We work our way from the outside (non-equivalence) inward (equivalence) in the way we construct the splitting nodes in $\bT$.
We will also define a function $\psi$ on the splitting nodes in $\bT$ with values in $\{0,1\}$ meeting  the Labeled Diagonal Coding Tree  requirements.

Given $\bT\re |c^{\bT}_{m-1}|$,
let
$\varphi:\bU(m)\ra \bT\re |c^{\bT}_{m-1}|$
 be the $\prec$-preserving bijection, and
let $t_*$ denote the node
$\varphi(c_m)$ in $\bT\re |c^{\bT}_{m-1}|$.
This $t_*$ is the node which we need to extend to the next coding node.
Recall that only the coding nodes in
$\bU$
 have more than one immediate successor; so $t_*$ is the only node we need to extend to one or three splitting nodes before making the level  $\bT\re |c^{\bT}_m|$.

The simplest case is when the coding node $c_m$ has two immediate successors:
these contain $\{x<v_m, xE v_m\}$ and $\{v_m<x, xE v_m\}$, respectively.
First extend $t_*$ to a coding node
$c_i\in \bU$,
 and then take extensions $s_0,s_1$ of this coding node  so that
$\{x<v_i, x E v_i\}\sse s_0$ and
$ \{v_i<x, x E v_i\}\sse s_1$.
Extend $s_0$ to a coding node
$c_j\in \bU$,
and define $c^{\bT}_m=c_j$ and $v^{\bT}_m=v_j$.
Let $u_0$ be the extension of $c^{\bT}_m$ in $\bU$
which contains $\{x<v_j, xE v_j\}$.
Extend $s_1$ to a node $t_1\in \bU\re |c^{\bT}_m|$, and let $u_1$ be the immediate successor of $t_1$ in $\bU$.
Extend  all other nodes in $\bT\re |c^{\bT}_{m-1}|$ (besides $t_*$) to a node in $\bU$ of length  $|c^{\bT}_m|$, and let $\bT\re |c^{\bT}_m|$ consist of these nodes along with $t_0$ and $t_1$.
Let $\bT\re (|c^{\bT}_m|+1)$ consist of $u_0, u_1$, and one immediate successor of each of the nodes in $\bT\re |c^{\bT}_m|$.
By the transitivity of both relations $<$ and $E$,
we obtain that the $\prec$-preserving  bijection between
 $\bU(m+1)$ and $\bT\re (|c^{\bT}_m|+1)$
preserves passing types over $\bQ_{\bQ}\re \{v^{\bT}_k:k\le m\}$.
Let $\psi(t_0\wedge t_1)=1$.

If the coding node $c_m$ has four immediate successors,
then these extensions consist  of all  choices from among $\{x<v_m,v_m<x\}$ and $\{xE v_m,x \hskip-.05in \not \hskip-.05in E v_m\}$.
We start on the outside with non-equivalence and work our way inside to equivalence.
First extend $t_*$ to a coding node $c_i\in \bU$ which has
four immediate successors, and let
$s_0$ denote the extension with
$\{x<v_i,x \hskip-.05in \not \hskip-.06in E v_i\}$
and
$s_3$ denote the extension with
$\{v_i<x,x \hskip-.05in \not \hskip-.06in E v_i\}$.
Again, extend $s_0$ to a coding node $c_j\in \bS$
 which has
four immediate successors and, abusing notation, let
$s_0$ denote the extension with
$\{x<v_j,x \hskip-.03in \not \hskip-.05in E v_j\}$
and
$s_1$ denote the extension with
$\{v_j<x,x \hskip-.03in \not \hskip-.05in E v_j\}$.
Then extend $s_1$ to any coding node $c_k$.
Take $c_\ell$ to be a coding node extending $c_k\cup\{x<v_k, x E v_k\}$,
and define $c^{\bT}_m=c_{\ell}$ and $v^{\bT}_m=v_{\ell}$.
Let $t_0$ be the  $\prec$-leftmost extension of $s_0$
in $\bU(\ell)$,
let $t_2$ be the $\prec$-leftmost extension of
$c_k\cup\{v_k<x, x E v_k\}$ in $\bU(\ell)$,
and let $t_3$ be the $\prec$-leftmost extension of $s_3$ in $\bU(\ell)$.
Finally, define $\bT\re |c^{\bT}_m|$ to consist of
$\{t_0,
c^{\bT}_m, t_2,t_3\}$ along with
the leftmost extensions in $\bU(\ell)$ of the nodes in
$(\bT\re |c^{\bT}_{m-1}|)\setminus\{t_*\}$.
Let the nodes in $\bT\re |c^{\bT}_{m}+1|$
consist of $c^{\bT}_{m}\cup\{x<v^{\bT}_m,xEv^{\bT}_m\}$, along with the immediate successors  in $\bU(\ell+1)$
of
the rest of the nodes in $\bT\re |c^{\bT}_{m}|$.
The three  new splitting nodes in $\bT$ are $t_0\wedge t_3=c_i$, $t_0\wedge t_2=t_0\wedge c^{\bT}_{m}=c_j$, and $t_2\wedge c^{\bT}_{m}=c_k$.
Define $\psi(c_i)=\psi(c_j)=1$ and $\psi(c_k)=0$.  It is straightforward to check that this satisfies the requirements of a Labeled Diagonal Coding Tree.

The idea for general $(\bQ_{\bQ})_n$ is similar.
Here we have a sequence of convex equivalence relations $\lgl E_i:i<n\rgl$, where for each $i<n-1$,
$E_{i+1}$ coarsens $E_i$.
Similarly to the above,  each coding node  $c_m$ has
$2(j+1)$ many immediate successors, for some $j\le n$.
The  immediate successors
run through all combinations of choices from among $\{x<v_m,v_m<x\}$
and $\{x  E_0 v_m\}\cup
\{(xE_{i+1} v_m\wedge$
$x \hskip-.05in \not \hskip-.04in E_{i} v_m) : i< j\}$.
When constructing skew splitting,
 in order to set up so that the desired passing types  are available  at the next coding node of $\bT$,
we start on the ``outside'' with  types containing
$(xE_{j} v_m\wedge
x \hskip-.05in \not \hskip-.04in E_{j-1} v_m)$
and work our way inward, with the increasingly finer equivalence relations,
analogously to how the case of four immediate successors was handled above for $\bQ_{\bQ}$.
The $\psi$ function takes values in $\{0,1,\dots, n\}$ and is defined as follows: 
For incomparable coding nodes $c^{\bT}_k$ and $c^{\bT}_{\ell}$ representing vertices $v_{k'}$ and $v_{\ell'}$, respectively, 
$\psi(c^{\bT}_k\wedge c^{\bT}_{\ell})=0$ if and only if $v_{k'}  E_{n-1} v_{\ell'}$;
for $1\le i< n$, $\psi(c^{\bT}_k\wedge c^{\bT}_{\ell})=n-i$ if and only if 
$v_{k'} E_{i-1} v_{\ell'}$ but 
$v_{k'} \hskip-.05in \not \hskip-.04in E_i v_{\ell'}$;
$\psi(c^{\bT}_k\wedge c^{\bT}_{\ell})=n$ if and only if $v_{k'} \hskip-.05in \not \hskip-.04in E_0 v_{\ell'}$.

The presence of any unary relations has no effect on the existence of labeled diagonal coding trees.
It is straightforward to check that the Labeled Extension Property is satisfied. 
\end{proof}

This brings us to our second collection of  big Ramsey structures.

\begin{thm}\label{thm.LOEqRels}
\begin{enumerate}
\item
The rationals, $\bQ$, satisfy \EEAP$^+$.
\item
$\bQ_n$, for each $n\ge 1$, satisfies \EEAP$^+$.
\item
$\bQ_{\bQ}$, and more generally,
$(\bQ_{\bQ})_n$ for each $n\ge 2$, satisfies LSDAP$^+$.
\item
The \Fraisse\ limit of any \Fraisse\ class
in $\mathcal{COE}_{n,p}$, for any $n,p\ge 1$, satisfies LSDAP$^+$.
\end{enumerate}
Hence
they
admit big Ramsey structures which are simply characterized.
\end{thm}

\begin{proof}
 This   follows from 
Theorems \ref{thm.onecolorpertype},
\ref{thm.bounds} and \ref{thm.BRS}, and 
Propositions \ref{prop.LO_n},
\ref{bQn} and
\ref{prop.loe}.
\end{proof}

We now discuss previous results which are recovered in Theorem \ref{thm.LOEqRels}, and results which  are new.

Part (1) of Theorem \ref{thm.LOEqRels} recovers
the following previously known results:
Upper bounds for  finite big Ramsey degrees of the rationals  were  found  by Laver \cite{LavUnp}
using Milliken's theorem.
The big Ramsey degrees were characterized and computed by Devlin in \cite{DevlinThesis}.
Zucker  interpreted
Devlin's characterization into a big Ramsey structure, from which he then constructed the universal completion flow  of the rationals
 in \cite{Zucker19}.

Exact big Ramsey degrees of the structures $\bQ_n$ were characterized  and calculated by Laflamme, Nguyen Van Th\'{e}, and Sauer in \cite{Laflamme/NVT/Sauer10}, using a colored level set version Milliken Theorem which they proved specifically for their application.
The work in this paper using coding trees of $1$-types   provides a new way to view and recover  their characterization of the big Ramsey degrees.
From their work on $\bQ_2$,
Laflamme, Nguyen Van Th\'{e}, and Sauer further
calculated the big Ramsey degrees of the circular
directed graph $\mathbf{S}(2)$
in \cite{Laflamme/NVT/Sauer10}.
Exact Ramsey degrees of $\mathbf{S}(n)$
for all $n\ge 3$  were recently calculated by  Barbosa in \cite{Barbosa20}
using category theory methods.
These structures $\mathbf{S}(n)$ have ages which do not satisfy \EEAP.

Part (3) of Theorem \ref{thm.LOEqRels} answers a question
 posed
 by Zucker during the open problem session
at the 2018 BIRS Workshop on {\em Unifying Themes in Ramsey Theory}:
He asked whether $\bQ_{\bQ}$ has finite big Ramsey degrees and whether it  admits a big Ramsey structure.
At that meeting,
proofs that $\bQ_{\bQ}$
 has finite big Ramsey degrees were found
  by  \Hubicka\ using unary functions and strong trees,
by Zucker using similar methods,
 and by Dobrinen
 using an approach that involved
developing a topological Ramsey space  with strong trees as bases, where each node in the given base is replaced with a strong tree.
None  of these proofs have been   published, nor were those upper bounds shown to be exact.
Independently, Howe also proved upper bounds for big Ramsey degrees in $\bQ_{\bQ}$ \cite{HoweThesis}.
The result  in this paper via  coding trees of $1$-types and LSDAP$^+$ 
characterizes exact big Ramsey degrees and
proves  that $\bQ_{\bQ}$ admits a big Ramsey structure, and
moreover, shows how it
 fits into a  broader scheme of structures which have easily described big Ramsey degrees.

Part (4) of Theorem \ref{thm.LOEqRels}  in its full generality is new.

We now mention some  \Fraisse\ classes that do not satisfy \EEAP.
These include \Fraisse\ classes of the form Forb$(\mathcal{F})$ 
in a language with at least one binary relation symbol 
where $\mathcal{F}$ contains 
 some forbidden irreducible  substructure
which is not $3$-irreducible.
 For instance, the  ages of the $k$-clique-free Henson graphs, most metric spaces, and the generic partial order  do not satisfy \EEAP.
We present two concrete examples of \Fraisse\ classes failing \EEAP\ to give an idea of how failure  can arise.

\begin{example}[\SFAP\ fails for triangle-free graphs]\label{ex.trianglefree}
  Let $\mathcal{G}_3$ denote
  the \Fraisse\ class of finite triangle-free graphs.
  Let $\bfA$ be the graph with  two  vertices $\{a_0,a_1\}$ forming a  non-edge,
   and let
  $\bfC$ be the graph with   vertices $\{a_0,a_1,v,w\}$ with   exactly  one  edge,
   $v\, E\, w$.
 Suppose
  $\bfB$ has vertices  $\{a_0,a_1,b\}$, where $b\not\in\{v,w\}$.
  Let $\sigma=\{\neg E(x,a_0)\wedge \neg E(x,a_1) \wedge E(x,b)\}$
  and
  $\tau=\{\neg E(x,a)\wedge E(x,a_1) \wedge E(x,b)\}$.
  Then $\sigma\re\bfA=\type(v/\bfA)$,
 $\tau\re\bfA=\type(w/\bfA)$, and $\sigma\ne\tau$.

Suppose $\bfE\in\mathcal{G}_3$ is a graph
satisfying
the conclusion of  Definition \ref{defn.SFAP}.
To simplify notation, suppose that $\bfE$ has
universe
$\mathrm{E}=\{a_0,a_1,b,v,w\}$, with the obvious inclusion maps being the amalgamation maps.
Then
$\type(v/\bfB)=\sigma$,  $\type(w/\bfB)=\tau$,
and $\bfE\re\{a_0,a_1,v,w\}\cong \bfC$,
 so  each pair in
$\{b,v,w\}$ has an edge in $\bfE$.
But this implies that  $\bfE$ has a triangle, contradicting
$\bfE\in\mathcal{G}_3$ .
Therefore, \SFAP\ fails for $\mathcal{G}_3$.
\end{example}

The failure of  \EEAP\ for partial orders can be proved similarly, by taking $\bfC$ to have two vertices not in $\bfA$ which are unrelated to each other, and constructing $\bfB$, $\sigma$, $\tau$ so that any extension $\bfE$ satisfying $\sigma$ and $\tau$ induces a relation between any $v'$, $w'$ satisfying $\sigma$, $\tau$ respectively,
in such a way
that transitivity forces there to be a relation between $v'$ and $w'$.

  We now give an example where \SFAP\ fails in a structure
  with  a relation of arity higher than two.

\begin{example}[\SFAP\ fails for
$3$-hypergraphs forbidding
the irreducible 3-hypergraph
on four vertices with three hyper-edges]\label{ex.hyperI}
Suppose our language has one ternary relation symbol $R$.
Let
$\mathbf{I}$  denote a ``pyramid'', the structure on four vertices with exactly three hyper-edges;
that is, say $\mathrm{I}=\{i,j,k,\ell\}$ and
$\mathbf{I}$ consists of the relation
$
\{R^\mathbf{I}(i,j,k), R^\mathbf{I}(i,j,\ell),R^\mathbf{I}(i,k,\ell)\}.
$
Then every two vertices in $\mathrm{I}$ are in some relation in $\mathbf{I}$, so $\mathbf{I}$ is irreducible.
However, the triple $\{j,k,\ell\}$ is not contained in any relation in $\mathbf{I}$.

The free amalgamation class Forb$(\{\mathbf{I}\})$ does not satisfy
\SFAP:
Let $\bfA$ be the singleton $\{a\}$,
with $R^\bfA = \emptyset$,
and let $\bfC$ have universe $\{a,c_0,c_1\}$ with
$R^\bfC = \{(a,c_0,c_1)\}$.
Let $\bfB$ have universe $\{a,b\}$, and let $\sigma$ and $\tau$ both be the $1$-types $\{R(x,a,b)\}$ over $\bfB$.
Suppose that $\bfE\in$ Forb$(\{\mathbf{I}\})$
satisfies the conclusion of
Definition \ref{defn.EEAP_new}.
Then $\bfE$ has universe $\{a,b,c_0,c_1\}$ and
 $R^\bfE = \{(a,c_0,c_1), (c_0,a,b), (c_1,a,b)\}$.
 Hence $\bfE$ contains a copy of $\mathbf{I}$, contradicting that $\bfE\in$ Forb$(\{\mathbf{I}\})$.
\end{example}

\begin{rem}
The same argument shows  that
\SFAP\ fails for
any  free amalgamation class Forb$(\mathcal{F})$
 where some $\mathbf{F}\in\mathcal{F}$
 is not 3-irreducible.
\end{rem}

%%%%%%%%%%%%%%%%%%%%%%%%%%%%%%%%%%%%%%%
%%%%%%%%%%%%%%%%%%%%%%%%%%%%%%%%%%%%%%%

We conclude our two papers by presenting  a catalogue of many (though not all) of the known results regarding indivisibility,
 finite  big Ramsey degrees  (upper bounds), and characterizations  of exact big Ramsey degrees
(canonical partitions).
A blank box means the property  has not yet been proved or disproved.
All previously known results for  \Fraisse\ classes
in languages with relations of arity at most two
with \Fraisse\ limits satisfying \EEAP$^+$ or LSDAP$^+$ are recovered by 
Theorem \ref{thm.main} in this paper.
Results which are new to our work in Parts I and II 
are indicated by the number of the theorem from which they follow.

In all cases where exact big Ramsey degrees have been characterized, this has been achieved via finding canonical partitions.
Moreover,
for structures in languages with relations of arity at most two,
these canonical partitions have been found in terms of similarity types of antichains
in
trees of $1$-types, either explicitly or implicitly.
Once one has such canonical partitions, the existence of a big Ramsey structure follows from
Theorem \ref{thm.apply}
in conjunction with Zucker's Theorem 7.1 in \cite{Zucker19}.
Thus, we do not include a column for existence of big Ramsey structures.
%\vfill\eject

\vspace{2mm}

\textbf{Key}
\vspace{2mm}

    \begin{tabular}{|l|}
    \hline
     $\bullet$
     DA: Disjoint Amalgamation\\
     $\bullet$ FA: Free Amalgamation\\
     $\bullet$ SDAP:
     strongest of \SFAP, \EEAP, \EEAP$^+$/LSDAP$^+$ known to hold\\
$\bullet$ IND: Indivisibility\\
     $\bullet$ FBRD: Finite big Ramsey degrees\\
     $\bullet$ CP: Exact big Ramsey degrees characterized via Canonical Partitions\\ \\
   \cmark \ Yes \qquad \xmark \ No \quad $\bigstar$ \ In some cases, not in all \\
    \hline
    \end{tabular}

\ \ \vskip.1in

%\eject

\begin{enumerate}[leftmargin=-.01cm]
    \item $\mathbb{Q}$-like structures
    \begin{center}
    \resizebox{12cm}{!}{
        \begin{tabular}{|p{3.45cm}|c|c|c|c|c|c|c|}
        \hline
        {\bf \Fraisse\ limit} & {\bf DA} & {\bf FA}  & {\bf SDAP}&
\bf{IND} &\bf{FBRD} & \bf{CP}\\
        \hline
        $\mathbb{Q}$ with no relations & \cmark & \cmark  &
         SDAP
& Pigeonhole
& \cite{Ramsey30} & \cite{Ramsey30}\\
        \hline
        $(\mathbb{Q},<)$ & \cmark & \xmark  & SDAP$^+$
& Folklore
& \cite{LavUnp} & \cite{DevlinThesis}\\
        \hline
        $\mathbb{Q}_n$ & \cmark & \xmark & SDAP$^+$
&  Folklore
& \cite{Laflamme/NVT/Sauer10} & \cite{Laflamme/NVT/Sauer10}\\
        \hline
        $\mathbf{S}(2)$& \cmark & \xmark & \xmark
&  \xmark
& \cite{Laflamme/NVT/Sauer10}& \cite{Laflamme/NVT/Sauer10}\\
        \hline
        $\mathbf{S}(3)$,$\mathbf{S}(4),\cdots$& \cmark & \xmark & \xmark
& \xmark
& \cite{Barbosa20} & \cite{Barbosa20}\\
 % \hline
% $\mathbb{Q}_{\mathbb{Q}}$  \NC{?? FBRD is hard to attribute - maybe just delete}
% & \cmark & \xmark  & SDAP$^+$
%& Folklore
%& \cite{HoweThesis} & [Thm \ref{thm.bounds}]\\
        \hline
        $\mathbb{Q}_{\mathbb{Q}}$, $\mathbb{Q}_{\mathbb{Q}_{\mathbb{Q}}}, \cdots$ & \cmark & \xmark  & LSDAP$^+$
& [Thm 5.12 of \cite{CDPII}]
%& [Thm \ref{thm.indivisibility}]
& [Thm 5.12 of \cite{CDPII}]
%& [Thm \ref{thm.onecolorpertype}] 
& [Thm 5.12 of \cite{CDPII}]\\
%[Thm \ref{thm.bounds}]\\

        \hline
 \Fraisse\ limit of $\mathcal{COE}_{n,p}$
 & \cmark & \xmark  & LSDAP$^+$
 & [Thm 5.12 of \cite{CDPII}]
%&  [Thm \ref{thm.indivisibility}]
& [Thm 5.12 of  \cite{CDPII}]
%& [Thm \ref{thm.onecolorpertype}] 
& [Thm 5.12 of \cite{CDPII}]\\
%[Thm \ref{thm.bounds}]\\
        \hline

       Main reducts of $(\bQ,<)$ & \cmark & \xmark &  SDAP$^+$
& Folklore
& \cite{Masulovic18} &  \\
      \hline
        Generic structures with two or more independent  linear relations & \cmark & \xmark & \EEAP\
& Folklore
&\cite{Hubicka_CS20} & \\
        \hline
        \end{tabular}}
    \end{center}
  % \vspace{2mm}
\vfill\eject

    \item Unconstrained  relational structures and their ordered expansions
        \begin{center}
        \resizebox{12cm}{!}{
        \begin{tabular}{|p{3.6cm}|c|c|c|c|c|c|c|}
                \hline
        {\bf \Fraisse\ limit} & {\bf DA} & {\bf FA} & {\bf SDAP}
&{\bf IND} & {\bf FBRD} & {\bf CP} \\
        \hline
            Rado graph & \cmark & \cmark  & SFAP
& Folklore
& \cite{Sauer06} & \cite{Sauer06}\\

            \hline
       Generic directed graph & \cmark & \cmark & SFAP
&   \cite{El-Zahar/Sauer93}
 & \cite{Laflamme/Sauer/Vuksanovic06} & \cite{Laflamme/Sauer/Vuksanovic06}\\
             \hline
             Generic tournament & \cmark & \xmark  & SDAP$^+$
&   \cite{El-Zahar/Sauer93}
& \cite{Laflamme/Sauer/Vuksanovic06} & \cite{Laflamme/Sauer/Vuksanovic06}\\
             \hline
             Generic unrestricted structures in a finite binary relational language
             %Finitely many binary relations with a finite universal constraint set
             & \cmark & %\xmark 
	$\bigstar$                        
             & SDAP$^+$
&  \cite{Laflamme/Sauer/Vuksanovic06}
 & \cite{Laflamme/Sauer/Vuksanovic06} & \cite{Laflamme/Sauer/Vuksanovic06} \\
             \hline
    %         Finite or countably infinite unions of $K_{\omega}$
             %, and their complements
       %      & \cmark & \xmark  & \cmark &  & \\
          %  \hline
Ordered expansions of any of the above structures
             & \cmark & \xmark  & SDAP$^+$
& [Thm 1.2 of \cite{CDPI}]
%& [Thm \ref{thm.indivisibility}]
& [Thm 3.10 of \cite{CDPII}]
%& [Thm \ref{thm.onecolorpertype}] 
& [Thm 4.8 of \cite{CDPII}]\\
%[Thm \ref{thm.bounds}]  \\
           %  \hline

         %   Generic two-graph & \cmark & %\cmark  & \xmark
%& \NC{Did M.
% prove FBRD?}
%& \cite{Masulovic18} & \\
 \hline

            Generic $3$-uniform hypergraph & \cmark & \cmark & SFAP
&  \cite{El-Zahar/Sauer94}
& \cite{Hubicka_et4_19withproofs} &  \\
            \hline
            Generic $k$-uniform hypergraph for $k > 3$ & \cmark & \cmark  & SFAP
&  \cite{El-Zahar/Sauer94}
&  \cite{Hubicka_et4_20} &
           \\
             \hline

Generic unrestricted structures  with relations in any arity, and their ordered expansions
             & \cmark & %\xmark  
$\bigstar$                 
             & SDAP$^+$
& [Thm 1.2 of \cite{CDPII}]
%& [Thm \ref{thm.indivisibility}]
& &  \\
             \hline
            %  Binary relational structures considered in [Zucker] & \cmark & \cmark & \cmark & \cmark \\
             %\hline
        \end{tabular}}
    \end{center}
    \vspace{2mm}

%\eject

    \item Constrained  structures with relations of arity at most two
        \begin{center}
        \resizebox{12cm}{!}{
        \begin{tabular}{|p{4.2cm}|c|c|c|c|c|c|c|}
                \hline
        {\bf \Fraisse\ limit} & {\bf DA} & {\bf FA} & {\bf SDAP} &  {\bf IND} &{\bf FBRD} & \bf{ CP} \\
        \hline
      Generic bipartite & \cmark & \cmark  & SFAP
&  Folklore
& \cite{HoweThesis} & [Thm 4.8 of \cite{CDPII}]\\
%[Thm \ref{thm.bounds}] \\
            \hline
             Generic $n$-partite for $n\geq 3$ & \cmark & \cmark & SFAP
&  Folklore
& \cite{Zucker20} & [Thm 4.8 of \cite{CDPII}]\\
%[Thm \ref{thm.bounds}]\\
            \hline
        Generic $K_3$-free graphs & \cmark & \cmark & \xmark
& \cite{Komjath/Rodl86}
& \cite{DobrinenJML20} & \cite{Balko7}\\
        \hline
        Generic $K_n$-free graphs for finite $n>3$ & \cmark & \cmark & \xmark
& \cite{El-Zahar/Sauer89}
& \cite{DobrinenH_k19} & \cite{Balko7}\\
        \hline
 \Fraisse\ limits with free amalgamation that are ``rank linear''
  & \cmark & \cmark & %\xmark
  $\bigstar$
& \cite{Sauer03}
 & \cite{Zucker20} & \cite{Balko7}\\
        \hline
        
\Fraisse\ limit of Forb$(\mathcal{F})$,
where
$\mathcal{F}$ is a finite  set of finite irreducible
structures

& \cmark & \cmark & %\xmark
$\bigstar$
& %\xmark
$\bigstar$
 & \cite{Zucker20} & \cite{Balko7}\\
        \hline
         Generic poset & \cmark & \xmark & \xmark
&  Folklore
& \cite{Hubicka_CS20} &  \\
        \hline
        \end{tabular}}
    \end{center}
    \vspace{2mm}

\vfill\eject

   \item Constrained %higher 
arbitrary  
   arity relational structures
   \vspace{2mm}
      \begin{center}
        \resizebox{12cm}{!}{
        \begin{tabular}{|p{4.3cm}|c|c|c|c|c|c|c|}
                \hline
        {\bf \Fraisse\ limit} & {\bf DA} & {\bf FA} & {\bf SDAP} & {\bf IND} & \bf{FBRD} & \bf{CP} \\
       \hline
        Generic $k$-hypergraph omitting a finite set of  finite $3$-irreducible  $k$-hypergraphs
for $k \ge 3$
         & \cmark & \cmark & %\xmark
SFAP                 
& \cite{El-Zahar/Sauer94}  & &  \\
        \hline
       \Fraisse\ limit of        $\Forb(\mathcal{F})$    where   all $F \in \mathcal{F}$ are irreducible and $3$-irreducible
      & \cmark & \cmark & SFAP
& %\cite{Sauer20}
 [Thm 1.2 of \cite{CDPI}]
%[Thm \ref{thm.indivisibility}]
& & \\
             \hline
            \Fraisse\ limit of
            $\Forb(\mathcal{F})^<$,
   where all $F \in \mathcal{F}$ are
   irreducible and
   $3$-irreducible
   & \cmark & \xmark & SDAP$^+$
& [Thm 1.2 of \cite{CDPI}]
%& [Thm \ref{thm.indivisibility}]
& &\\
             \hline

%\Fraisse\ limits with free amalgamation that are ``rank linear''
%   & \cmark & \cmark  & %\xmark
% $\bigstar$  
%& \cite{Sauer20}
%& &  \\
%             \hline

        \end{tabular}}
    \end{center}
\end{enumerate}
\ \vskip.1in

%\begin{rem}
%Sauer's result in \cite{Sauer20} is the state of the art for  indivisibility  of  free amalgamation classes:  He proved that a \Fraisse\ structure with free amalgamation in a finite relational language is indivisible if and only if it is what he calls ``rank linear''. For the definition of rank linear, see Definition 1.3 in \cite{Sauer20}.
%\end{rem}

\begin{rem}
Results on indivisibility and big Ramsey degrees of metric spaces appear in
\cite{DLPS07},
\cite{DLPS08},
 \cite{Hubicka_CS20},
\cite{Masulovic18}, \cite{Masulovic_RBS20},
\cite{NVTThesis},
\cite{NVT08}, \cite{NVTMem10}, and \cite{NVTSauer09}.
In his PhD thesis \cite{NVTThesis},
Nguyen Van Th\'{e} proved
results on indivisibility of Urysohn spaces
which were later published in \cite{NVTMem10}, including
that all Urysohn spaces with distance set $S$ of size four are indivisible (except for $S=\{1,2,3,4\}$).
A characterization of those
countable ultrametric spaces which are homogeneous and indivisible
was proved by
Delhomm\'{e}, Laflamme, Pouzet, and Sauer in
 \cite{DLPS08}.
Nguyen Van Th\'{e} showed finite big Ramsey degrees for finite $S$-submetric spaces of ultrametric $S$-spaces
in
\cite{NVT08}, with $S$ finite and nonnegative.
In
\cite{NVTSauer09},
Nguyen Van Th\'{e} and Sauer   proved that for each integer $m\ge 1$, the countable homogenous metric space with distances in $\{1,\dots, m\}$ is indivisible.
Sauer established indivisibility
 of Urysohn $S$-metric spaces with S finite  in \cite{Sauer12}.
Ma\v{s}ulovi\'{c}
proved finite big Ramsey degrees for
 Urysohn $S$-metric spaces,
 where $S$ is a finite distance set with no internal jumps and a property called ``compactness'' in that paper, meaning that the distances are not too far apart.
Recently, Hubi\v{c}ka  extended this to all  Urysohn $S$-metric spaces where
$S$ is tight in addition to finite and nonnegative \cite{Hubicka_CS20}.
As \EEAP\ fails for non-trivial metric spaces (for the same reason it fails for the triangle-free graphs and partial orders),  we mention no details here.
\end{rem}

\section{Concluding remarks and open problems}\label{sec.infdiml}

In  
Section \ref{sec.EEAPClasses},
we gave examples of \Fraisse\ classes with \Fraisse\ limits satisfying \EEAP$^+$ or LSDAP$^+$.
%By Theorem  \ref{thm.indivisibility} of Part I, any  such \Fraisse\ limit is indivisible, and by Theorem \ref{thm.main} in Part II,  any such \Fraisse\ limit with relations of arity at most two has finite big Ramsey degrees and admits a big Ramsey structure that has a simple characterization.
%ND:  I think this is redundant.  By now hopefully the reader knows the main theorems. 

\begin{question}
Which other \Fraisse\ classes either satisfy \SFAP, or more generally, have \Fraisse\ limits satisfying \EEAP$^+$ or LSDAP$^+$?
\end{question}

\Fraisse\ structures consisting of finitely many independent linear orders
present an interesting case as 
%we do not know whether they satisfy the
they do not 
%ND:  I have worked on this and they just never will have diagonal coding trees.
Diagonal Coding Property, 
%and hence whether they have \EEAP$^+$, 
but their
ages do have \EEAP\ 
and their coding trees have bounded branching.
This motivates the formulation of the following properties:
For $k\ge 2$,  we say that
the \Fraisse\ limit $\bK$ of
a \Fraisse\ class $\mathcal{K}$ satisfies
{\em $k$-\EEAP$^+$}
if
$\bK$
satisfies \EEAP; there is a perfect  subtree $\bT$ of the coding tree of $1$-types $\bU$
for $\bK$
such that $\bT$ represents a copy of $\bK$ and $\bT$ has splitting nodes with degree $\le k$; and the appropriately formulated Extension Property holds.
Note that here, 
%we are not requiring $\bT$ to be a skew tree:  
$\bT$ is allowed to have more than one splitting node on any given level.
 We let
 \BEEAP$^+$
 stand for
 {\em Bounded \EEAP$^+$},
meaning that there is a $k \ge 2$ such that
$k$-\EEAP$^+$
holds.
This brings us to the following implications.

\begin{fact}\label{SFAPEEAP}
\SFAP\ $\Lra$ \EEAP$^+$ $\Lra$ \TwoEEAP$^+$
$\Lra$ \BEEAP$^+$ $\Lra$ \EEAP.
\end{fact}

Theorem 4.20 in Part I, \cite{CDPI},  showed
that
\SFAP\  implies  \EEAP$^+$.
By definition, \EEAP$^+$\ implies
\TwoEEAP$^+$,
which in turn implies
\BEEAP$^+$.
 Each of these
 properties implies
 \EEAP, again by definition.
 The example of finitely many independent linear orders shows that
 \BEEAP$^+$ does not imply \EEAP$^+$.
On the other hand,
all examples considered in this paper satisfying \EEAP\ also satisfy
\BEEAP$^+$.
 It could well be the case   that \EEAP\   is equivalent to
  \BEEAP$^+$.
 The methods in this paper can be adjusted  to handle structures with
  \BEEAP$^+$,
 so the following question becomes interesting.

\begin{question}
Are \EEAP, \BEEAP$^+$, and \TwoEEAP$^+$ equivalent?
In other words, does  \EEAP\ imply \BEEAP$^+$, and does \BEEAP$^+$ imply \TwoEEAP$^+$?
\end{question}

Throughout this paper, we have mentioned  known results
regarding   finite big Ramsey degrees.
Actual calculations of  big Ramsey degrees, however, are still sparse, and have only been found
for the rationals by Devlin in \cite{DevlinThesis},
the Rado graph by Larson
 in
\cite{Larson08},
 the structures $\bQ_n$  and $\mathbf{S}(2)$ by Laflamme, Nguyen Van Th\'{e}, and Sauer in
 \cite{Laflamme/NVT/Sauer10},
 and the rest of the circular digraphs $\mathbf{S}(n)$, $n\ge 3$, by Barbosa in
\cite{Barbosa20}.
The canonical partitions in Theorem \ref{thm.bounds}
 provide a template for calculating the big Ramsey degrees for all \Fraisse\ structures satisfying
 \EEAP$^+$ or LSDAP$^+$.

\begin{problem}
Calculate the big Ramsey degrees $T(\bfA,\bK)$,
$\bfA\in\mathcal{K}$,
 for each \Fraisse\ class
 $\mathcal{K}$
with relations of arity at most two and
 with a \Fraisse\ limit
 satisfying
 \EEAP$^+$ or LSDAP$^+$.
\end{problem}

Lastly,
it is our hope that using combs in trees of $1$-types might lead to smaller bounds  for the ordered  Ramsey property.

\begin{problem}\label{prob.combbetterbound}
Suppose $\mathcal{K}$
is a \Fraisse\ class
with relations of arity at most two and
with \Fraisse\ limit satisfying
\EEAP$^+$.
Use Theorem \ref{thm.SESAPimpliesORP} to  find  better bounds for the smallest size of a structure $\bfC\in\mathcal{K}^{<}$
such that
\begin{equation}
\bfC\ra (\bfB)^{\bfA}
\end{equation}
for any given $\bfA\le\bfB$ inside $\mathcal{K}^{<}$.
\end{problem}

\newpage

\bibliography{references}

\providecommand{\MR}{\relax\ifhmode\unskip\space\fi MR }
\providecommand{\MRhref}[2]{%
  \href{http://www.ams.org/mathscinet-getitem?mr=#1}{#2}
}
\providecommand{\href}[2]{#2}
\begin{thebibliography}{10}

\bibitem{Balko7}
M.~Balko, D.~Chodounsk{\'{y}}, N.~Dobrinen, J.~Hubi{\v{c}}ka,
  M.~Kone{\v{c}}n{\'{y}}, L.~Vena, and A.~Zucker, \textsl{Exact big {R}amsey
  degrees for binary relational structures with forbidden irreducible
  substructures}, 2021, Submitted. arxiv:2110.08409, p.~97 pp.

\bibitem{Hubicka_et4_19withproofs}
M.~Balko, D.~Chodounsk{\'{y}}, J.~Hubi{\v{c}}ka, M.~Kone{\v{c}}n{\'{y}}, and
  L.~Vena, \textsl{Big {R}amsey degrees of 3-uniform hypergraphs are finite},
  2020, Submitted. arXiv:2008.00268, p.~9 pp.

\bibitem{Hubicka_et4_20}
M.~Balko, D.~Chodounsk{\'{y}}, J.~Hubi{\v{c}}ka, M.~Kone{\v{c}}n{\'{y}}, and
  L.~Vena, \textsl{Big {R}amsey degrees of unconstrained relational
  structures}, 2021, In preparation.

\bibitem{Barbosa20}
Keegan~Dasilva Barbosa, \textsl{A categorical notion of precompact expansions},
  Archive for Mathematical Logic (2020), 29 pp, Submitted. arXiv:2002.11751.

\bibitem{Conant17}
Gabriel Conant, \textsl{An axiomatic approach to free amalgamation}, Journal of
  Symbolic Logic \textbf{82} (2017), no.~2, 648--671.

\bibitem{CDP20}
Rebecca Coulson, Natasha Dobrinen and Rehana Patel, \textsl{Fra{\"{i}}ss{\'{e}}
  classes with simply characterized big {R}amsey structures}, 2020,
  arXiv:2010.02034, p.~69 pp.

\bibitem{CDPI}
Rebecca Coulson, Natasha Dobrinen and Rehana Patel, \textsl{Fra{\"{i}}ss{\'{e}}
  structures with {SDAP}$^+$, {P}art {I}: {I}ndivisibility}, 2022, p.~54 pp.

\bibitem{CDPII}
Rebecca Coulson, Natasha Dobrinen and Rehana Patel, \textsl{Fra{\"{i}}ss{\'{e}}
  structures with {SDAP}$^+$, {P}art {II}: {S}imply characterized big {R}amsey
  structures}, 2022, p.~58 pp.

\bibitem{DLPS07}
Christian Delhomm{\'{e}}, Claude Laflamme, Maurice Pouzet, and Norbert Sauer,
  \textsl{Divisibility of countable metric spaces}, European Journal of
  Combinatorics \textbf{28} (2007), no.~6, 1746--1769.

\bibitem{DLPS08}
Christian Delhomm{\'{e}}, Claude Laflamme, Maurice Pouzet, and Norbert Sauer,
  \textsl{Indivisible ultrametric spaces}, Topology and Its Applications
  \textbf{155} (2008), no.~14, 1462--1478.

\bibitem{DevlinThesis}
Dennis Devlin, \textsl{Some partition theorems for ultrafilters on $\omega$},
  Ph.D. thesis, Dartmouth College, 1979.

\bibitem{DobrinenRado19}
Natasha Dobrinen, \textsl{Borel sets of {R}ado graphs and {R}amsey's theorem},
  To appear. arXiv:1904.00266v1, p.~29 pp.

\bibitem{DobrinenH_k19}
Natasha Dobrinen, \textsl{Ramsey theory of the universal homogeneous
  k-clique-free graph}, Journal of Mathematical Logic (2020), 75 pp.

\bibitem{DobrinenJML20}
Natasha Dobrinen, \textsl{The {R}amsey theory of the universal homogeneous
  triangle-free graph}, Journal of Mathematical Logic \textbf{20} (2020),
  no.~2, 2050012, 75 pp.

\bibitem{El-Zahar/Sauer89}
Mohamed El-Zahar and Norbert Sauer, \textsl{The indivisibility of the
  homogeneous ${K}_n$-free graphs}, Journal of Combinatorial Theory, Series B
  \textbf{47} (1989), no.~2, 162--170.

\bibitem{El-Zahar/Sauer93}
Mohamed El-Zahar and Norbert Sauer, \textsl{On the divisibility of homogeneous
  directed graphs}, Canadian Journal of Mathematics \textbf{45} (1993), no.~2,
  284--294.

\bibitem{El-Zahar/Sauer94}
Mohamed El-Zahar and Norbert Sauer, \textsl{On the divisibility of homogeneous
  hypergraphs}, Combinatorica \textbf{14} (1994), no.~2, 159--165.

\bibitem{HoweThesis}
John Howe, \textsl{Big {R}amsey degrees in homogeneous structures}, Ph.D.
  thesis, University of Leeds, Expected 2020.

\bibitem{Hubicka_CS20}
Jan Hubi{\v{c}}ka, \textsl{Big {R}amsey degrees using parameter spaces}, 2020,
  Preprint. arXiv:2009.00967, 19 pp.

\bibitem{Hubicka/Nesetril19}
Jan Hubi{\v{c}}ka and Jaroslav Ne{\v{s}}et{\v{r}}il, \textsl{All those {R}amsey
  classes ({R}amsey classes with closures and forbidden homomorphisms)},
  Advances in Mathematics \textbf{356} (2019), 106791, 89 pp.

\bibitem{Kechris/Pestov/Todorcevic05}
Alexander Kechris, Vladimir Pestov and Stevo Todorcevic,
  \textsl{Fra{\"{i}}ss{\'{e}} limits, {R}amsey theory, and topological dynamics
  of automorphism groups}, Geometric and Functional Analysis \textbf{15}
  (2005), no.~1, 106--189.

\bibitem{Komjath/Rodl86}
P{\'{e}}ter Komj{\'{a}}th and Vojt{\v{e}}ch R{\"{o}}dl, \textsl{Coloring of
  universal graphs}, Graphs and Combinatorics \textbf{2} (1986), no.~1, 55--60.

\bibitem{Laflamme/NVT/Sauer10}
Claude Laflamme, Lionel Nguyen Van~Th{\'{e}} and Norbert Sauer,
  \textsl{Partition properties of the dense local order and a colored version
  of {M}illiken's theorem}, Combinatorica \textbf{30} (2010), no.~1, 83--104.

\bibitem{Laflamme/Sauer/Vuksanovic06}
Claude Laflamme, Norbert Sauer and Vojkan Vuksanovic, \textsl{Canonical
  partitions of universal structures}, Combinatorica \textbf{26} (2006), no.~2,
  183--205.

\bibitem{Larson08}
Jean Larson, \textsl{Counting canonical partitions in the random graph},
  Combinatorica \textbf{28} (2008), no.~6, 659--678.

\bibitem{LavUnp}
Richard Laver, unpublished.

\bibitem{Masulovic18}
Dragan Ma{\v{s}}ulovi{\'{c}}, \textsl{Finite big {R}amsey degrees in universal
  structures}, Journal of Combinatorial Theory, Series A \textbf{170} (2020),
  30 pp.

\bibitem{Masulovic_RBS20}
Dragan Ma{\v{s}}ulovi{\'{c}}, \textsl{{R}amsey degrees: big v. small}, European
  Journal of Combinatorics \textbf{95} (2021), 103323, 25 pp.

\bibitem{Milliken79}
Keith~R. Milliken, \textsl{A {R}amsey theorem for trees}, Journal of
  Combinatorial Theory, Series A \textbf{26} (1979), 215--237.

\bibitem{Nesetril/Rodl77}
Jaroslav Ne{\v{s}}et{\v{r}}il and Vojt{\v{e}}ch R{\"{o}}dl, \textsl{Partitions
  of finite relational and set systems}, Journal of Combinatorial Theory Series
  A \textbf{22} (1977), no.~3, 289--312.

\bibitem{Nesetril/Rodl83}
Jaroslav Ne{\v{s}}et{\v{r}}il and Vojt{\v{e}}ch R{\"{o}}dl, \textsl{Ramsey
  classes of set systems}, Journal of Combinatorial Theory Series A \textbf{34}
  (1983), no.~2, 183--201.

\bibitem{NVTThesis}
Lionel Nguyen Van~Th{\'{e}}, \textsl{Structural ramsey theory of metric spaces
  and topological dynamics of isometry groups}, Ph.D. thesis, Universit{\'{e}}
  Paris-Diderot - Paris 7, 2006.

\bibitem{NVT08}
Lionel Nguyen Van~Th{\'{e}}, \textsl{Big {R}amsey degrees and divisibility in
  classes of ultrametric spaces}, Canadian Mathematical Bulletin \textbf{51}
  (2008), no.~3, 413--423.

\bibitem{NVTMem10}
Lionel Nguyen Van~Th{\'{e}}, \textsl{Structural ramsey theory of metric spaces
  and topological dynamics of isometry groups}, no. 968, Memoirs of the
  American Mathematical Society 206, 2010.

\bibitem{NVTSauer09}
Lionel Nguyen Van~Th{\'{e}} and Norbert Sauer, \textsl{The {U}rysohn sphere is
  oscillation stable}, Geometric Functional Analysis \textbf{19} (2009), no.~2,
  536--557.

\bibitem{Ramsey30}
Frank~P. Ramsey, \textsl{On a problem of formal logic}, Proceedings of the
  London Mathematical Society \textbf{30} (1929), 264--296.

\bibitem{Sauer03}
Norbert Sauer, \textsl{Canonical vertex partitions}, Combinatorics,
  Probability, and Computing \textbf{12} (2003), no.~6, 671--704.

\bibitem{Sauer06}
Norbert Sauer, \textsl{Coloring subgraphs of the {R}ado graph}, Combinatorica
  \textbf{26} (2006), no.~2, 231--253.

\bibitem{Sauer12}
Norbert Sauer, \textsl{Vertex partitions of metric spaces with finite distance
  sets}, Discrete Mathematics \textbf{312} (2012), no.~1, 119--128.

\bibitem{Sierpinski}
Sierpi{\'{n}}ski, \textsl{Sur une probl{\`{e}}me de lat th{\'{e}}orie des
  relations}, Ann. Scuola Norm. Super. Pisa, Ser. 2 \textbf{2} (1933),
  239--242.

\bibitem{Siniora/Solecki20}
Daoud Siniora and S{\l}awomir Solecki, \textsl{Coherent extension of partial
  automorphisms, free amalgamation and automorphism groups}, Journal of
  Symbolic Logic \textbf{85} (2020), no.~1, 199--223.

\bibitem{TodorcevicBK10}
Stevo Todorcevic, \textsl{Introduction to {R}amsey {S}paces}, Princeton
  University Press, 2010.

\bibitem{Zucker19}
Andy Zucker, \textsl{Big {R}amsey degrees and topological dynamics}, Groups,
  Geometry and Dynamics \textbf{13} (2018), no.~1, 235--276.

\bibitem{Zucker20}
Andy Zucker, \textsl{A {N}ote on {B}ig {R}amsey degrees}, 2020, Submitted.
  arXiv:2004.13162, 21 pp.

\end{thebibliography}
\bibliographystyle{ijmart}

\end{document}